\documentclass[11pt,reqno]{amsart}
\usepackage{amsmath,amssymb,amsfonts,amsthm}
\usepackage{stmaryrd}
\usepackage{cite}
\usepackage{graphicx}
\usepackage{xcolor}
\usepackage{float}
\usepackage{placeins}
\usepackage{needspace}
\newtheorem{theorem}{Theorem}[section]
\newtheorem{lemma}[theorem]{Lemma}

\newtheorem{proposition}[theorem]{Proposition}
\newcounter{assumption}
\newenvironment{assumptionblock}[1]
{\par\medskip\renewcommand{\theassumption}{\ensuremath{#1}}%
 \refstepcounter{assumption}\noindent{\normalfont\bfseries (\theassumption).}\ \itshape\ignorespaces}
{\par\medskip}
\newcounter{proofstep}
\newcounter{figurepanel}
\renewcommand{\thefigurepanel}{\roman{figurepanel}}
\numberwithin{equation}{section}
\makeatletter
\renewenvironment{proof}[1][\proofname]{\par
  \pushQED{\qed}%
  \normalfont\topsep6\p@\@plus6\p@\relax
  \trivlist
  \item[\hskip\labelsep\bfseries #1\@addpunct{.}]\ignorespaces
}{\popQED\endtrivlist\@endpefalse}
\makeatother
\allowdisplaybreaks
\begin{document}

\title[Stability and Instability of Positive Steady States]
{\textbf{Stability and Instability of the Unique Positive Steady State for a Nonlocal Population Equation}}

\author[Y. Lou and M.-Z. Xin]{Yuan Lou$^1$ and Ming-Zhen Xin$^1$}

\thanks{\hspace{-.6cm}
    $^1$School of Mathematical Sciences, Shanghai Jiao Tong University,
    Shanghai 200240, P.R. China. }

\date{\today}

\begin{abstract}
This paper is concerned with the effect of spatial heterogeneity on the
linear stability of unique positive steady states for a nonlocal population
equation. The main difficulty is that, because of the nonlocal term, the
linearized operator need not be self-adjoint or satisfy the maximum
principle, so the usual stability arguments based on the principal
eigenvalue do not apply directly. We establish linear stability for sufficiently large diffusion and,
under suitable conditions on the maximum set of the local growth
ratio, for sufficiently small diffusion.  In one space dimension, we also construct unstable positive
steady states and obtain Hopf bifurcations for sufficiently small
diffusion. Keeping the local coefficients fixed, we obtain either a supercritical
or a subcritical Hopf bifurcation by changing the weight in the nonlocal integral.
Our results show that uniqueness of a positive steady state does not imply
its linear stability.
    
    
\end{abstract}

\maketitle

\section{Introduction}
\label{sec:introduction}

{\color{black}In population models, the growth rate at a given location may depend on the
total population size or resource consumption over the whole habitat. This
leads to nonlocal interactions in population dynamics, adaptive evolution, and
spatial ecology. In this paper, we study the effects of spatial heterogeneity
on the linear stability of positive steady states for the following nonlocal
population equation}
\begin{equation}
\begin{cases}
\partial_t u
=d\Delta u+
u\left(a(x)-b(x)\displaystyle\int_\Omega c(y)u(t,y)\,\mathrm{d}y\right),
& t>0,\ x\in\Omega,\\[2mm]
\partial_\nu u=0,
& t>0,\ x\in\partial\Omega,
\end{cases}
\label{eq:population-model}
\end{equation}
where \(\Omega\subset\mathbb R^N\) is a smooth bounded domain, \(d>0\) is
the diffusion rate, and \(u(t,x)\) is the population density at time \(t\)
and position \(x\). The coefficient \(a(x)\) is the local growth rate,
while \(\int_\Omega c(y)u(t,y)\,\mathrm{d}y\) measures the global crowding
or resource consumption. The coefficient \(b(x)\) describes how this global
quantity affects growth at each location. We assume that
\begin{equation}
a\in C(\overline\Omega),\qquad
b,c\in C(\overline\Omega),\qquad\text{and}\qquad
b>0\quad\text{and}\quad c>0\quad\hbox{on }\overline\Omega .
\label{eq:coefficient-assumptions}
\end{equation}

Equations of this type arise in selection-mutation models. For constant
\(b\) and \(c\), Magal and Webb \cite[Theorem~3.1]{MagalWebb2000} established
convergence to the positive equilibrium in the one-dimensional setting
whenever this equilibrium exists; Lam and Lou
\cite[Theorem~10.1.1 and Remark~10.1.3]{LamLou2022} use semiclassical analysis
to describe the concentration of the positive equilibrium as \(d\to0^+\)
at a unique interior global maximum point of the growth rate. Related nonlocal
interactions also appear in stem cell regeneration models
\cite{SuLiLouWang2023}. {\color{black}For further studies of selection-mutation equations and nonlocal population
models, we refer to
\cite{Perthame2007,PerthameBarles2008,
DesvillettesJabinMischlerRaoul2008,BarlesMirrahimiPerthame2009,
LorzMirrahimiPerthame2011,
LorenziPouchol2020,ShiShiSong2022,LiJiangShi2013,SunShiWang2013}.}

{\color{black}Spatial heterogeneity in food availability, shelter, and environmental quality
affects local growth and interaction rates. The effects of dispersal and spatial
heterogeneity on population persistence, equilibrium abundance, and competition
have been studied in}
\cite{CantrellCosner1989,CantrellCosner1991,CantrellCosner2003,Lou2006,DockeryHutsonMischaikowPernarowski1998,
LiNakashimaNi2008,Lam2017,LamLou2017,HaoLamLou2019}.
{\color{black}For \eqref{eq:population-model} with spatially varying \(b\), the existence
and uniqueness of a positive steady state are determined by the weighted
principal eigenvalue problem \eqref{eq:principal-eigenpair}. To study its
stability, we also need to consider the nonlocal term.} Let \(u_d^*\) be a
positive steady state and set
\(S_d^*=\int_\Omega c(y)u_d^*(y)\,\mathrm{d}y\). The linearized operator is
\begin{equation}
\mathcal L_d\varphi
=d\Delta\varphi+\bigl(a(x)-b(x)S_d^*\bigr)\varphi
-b(x)u_d^*(x)\int_\Omega c(y)\varphi(y)\,\mathrm{d}y,
\qquad \partial_\nu\varphi=0.
\label{eq:intro-linearized-operator}
\end{equation}
In \eqref{eq:intro-linearized-operator}, the differential part
\(A_d:=d\Delta+\bigl(a(x)-b(x)S_d^*\bigr)\) is self-adjoint. The nonlocal
term need not be self-adjoint or sign preserving.  Even when a
positive eigenfunction exists, the corresponding eigenvalue need not
be the rightmost spectral value (see  \cite{SunShiWang2013,SunShiWang2013Erratum}). 
{\color{black}Thus, we need to study the full spectrum of \(\mathcal L_d\) to determine stability.} Let \(0=\mu_1(d)>\mu_2(d)\ge\cdots\) be the
eigenvalues of \(A_d\), with a corresponding orthonormal basis of real
eigenfunctions \(\{\psi_k^d\}_{k\ge1}\) in \(L^2(\Omega)\),
chosen so that \(\psi_1^d>0\).
By \cite[Theorem~3.4]{LiCovilleWang2017},
\begin{equation}
\begin{aligned}
\sigma(\mathcal L_d)
={}&\{\mu_k(d):\mu_k(d)\text{ is a multiple eigenvalue of }A_d\}
\\
&\cup\{\mu_k(d):\mu_k(d)\text{ is simple and }
\langle b u_d^*,\psi_k^d\rangle\langle c,\psi_k^d\rangle=0\}
\\
&\cup\left\{z\notin\sigma(A_d):
\sum_{k=1}^{\infty}
\frac{\langle b u_d^*,\psi_k^d\rangle\langle c,\psi_k^d\rangle}
{z-\mu_k(d)}=-1\right\},
\end{aligned}
\label{eq:rank-one-spectrum}
\end{equation}
where \(\langle\cdot,\cdot\rangle\) denotes the \(L^2(\Omega)\) inner product. 

If \(b\) is constant, then \(b u_d^*\) is a positive multiple of
\(\psi_1^d\). Thus \eqref{eq:rank-one-spectrum} gives
\(\sigma(\mathcal L_d)\subset(-\infty,0)\), and then the positive steady state is
linearly stable whenever it exists. This is
consistent with the convergence result of Magal and Webb \cite{MagalWebb2000}.
When \(b\) varies in space, the scalar equation in
\eqref{eq:rank-one-spectrum} does not directly determine stability.
{\color{black}Coville and Girardin \cite{CovilleGirardin2025} established abstract results
on the existence, uniqueness, and stability of steady states. They also give
    a three-dimensional discrete example with symmetric mutations whose unique
    positive equilibrium is linearly unstable. For \eqref{eq:population-model}, applying these abstract criteria requires
    estimates in terms of the coefficients \(a\), \(b\), and \(c\) and the
    diffusion rate \(d\). We therefore study how spatial heterogeneity and
    diffusion affect the stability of \(u_d^*\).

{\color{black}By using the resolvent approach in \cite{LiCovilleWang2017}, we first derive
a Krein-type criterion for the linearized operator. This criterion compares
the component of the nonlocal term along the principal eigenfunction with its
orthogonal component. We obtain linear stability for all sufficiently large
diffusion rates whenever the positive steady state exists. We also prove
small-diffusion stability under suitable constancy assumptions on \(b\)
near or on the maximum set of \(a/b\), including Morse--Bott maximum sets
of dimensions \(0,\ldots,N-1\) and plateau maxima of positive measure;
see Theorem~\ref{thm:small-diffusion-stability}.}

{\color{black}We then construct unstable positive steady states and Hopf bifurcations
in one space dimension. We reduce the spectral problem for the linearized operator $\mathcal{L}_{d}$
to that of a matrix whose dimension equals the number of an effective growth
profile. The one- and two-dimensional matrices in our reduction
are stable; see Proposition~\ref{prop:one-two-mode-reduction} and
\cite{BierkensRan2014}. Hence, we consider an effective growth
profile with three maximum points. By Agmon estimates and the
Helffer--Sj\"ostrand tunneling method
\cite{Agmon1982,HelfferSjostrand1984,Helffer1988,FH2010}, the spectrum near
zero can be studied through a \(3\times3\) matrix. We choose the weight
functions so that this matrix is unstable and then prove instability for the
full linearized operator. By varying a single parameter, we obtain a Hopf
bifurcation of positive time-periodic solutions. Moreover, keeping \(a\) and \(b\) fixed, we can obtain either a
supercritical or a subcritical Hopf bifurcation by changing only \(c\).}

{\color{black}The rest of this paper is organized as follows. In Section~\ref{sec:main-results},
we give the positive steady-state formula, define linear stability, and
state the main results. In Section~\ref{sec:stability-proofs}, we prove the
stability criterion and the stability theorems. In
Section~\ref{sec:semiclassical-reduction}, we establish the semiclassical
reduction. In Section~\ref{sec:instability-hopf}, we use this reduction to
construct instability examples and both supercritical and subcritical Hopf
bifurcations.}

\section{Preliminaries and main results}
\label{sec:main-results}

{\color{black}In this section, we introduce some notation and the positive steady-state
formula, and then state our main results on stability, instability, and Hopf
bifurcation.}
We use \(\|\cdot\|_2\) for the Euclidean norm. Norms of linear and bilinear
maps are denoted by \(\|\cdot\|\) and are induced by the relevant
\(L^2\) or Euclidean norms.

Denote by \((\lambda_d,\psi_1^d)\) the principal eigenpair of
\begin{equation}
\begin{cases}
d\Delta\psi_1^d+a(x)\psi_1^d=\lambda_d b(x)\psi_1^d,
& x\in\Omega,\\
\partial_\nu\psi_1^d=0,
& x\in\partial\Omega.
\end{cases}
\label{eq:principal-eigenpair}
\end{equation}
We choose \(\psi_1^d>0\) and normalize it by
\(\|\psi_1^d\|_{L^2(\Omega)}=1\).
For the stability results in Sections~\ref{sec:main-results} and~\ref{sec:stability-proofs}, we impose the standing
assumption
\begin{equation}
\int_\Omega a(x)\,\mathrm{d}x>0.
\label{eq:positive-mean-stability}
\end{equation}
This assumption is not imposed on the instability and Hopf bifurcation
construction in Section~\ref{sec:instability-hopf}, where the coefficients are chosen separately.
{\color{black}By an argument similar to that in \cite[Theorem~2]{AllenBolkerLouNevai2008}, we have
$\lambda_d>0$ for any $d>0$.}
Any positive steady state solves the weighted eigenvalue problem
with eigenvalue \(S_d^*=\int_\Omega c(y)u_d^*(y)\,\mathrm{d}y\).
Since the eigenfunction is positive, the eigenvalue is \(\lambda_d\) and the
eigenfunction is a multiple of \(\psi_1^d\). Conversely, a positive multiple
of \(\psi_1^d\) with \(S_d^*=\lambda_d\) satisfies the steady-state equation
by direct substitution. Thus, for every \(d>0\), the unique positive steady state is
\begin{equation}
u_d^*=\alpha_d\psi_1^d,\qquad\text{and}\qquad
\alpha_d=\frac{\lambda_d}{\int_\Omega c(y)\psi_1^d(y)\,\mathrm{d}y}.
\label{eq:positive-steady-state}
\end{equation}

The linearized operator at \(u_d^*\) is
\[
\begin{cases}
\mathcal L_d\eta
=
A_d\eta
-
b(x)u_d^*(x)
\displaystyle\int_\Omega c(y)\eta(y)\,\mathrm{d}y,
& x\in\Omega,\\[2mm]
\partial_\nu\eta=0,
& x\in\partial\Omega,
\end{cases}
\]
where
\(A_d\eta=d\Delta\eta+\bigl(a(x)-\lambda_d b(x)\bigr)\eta\),
\(x\in\Omega\).
Since \(A_d\psi_1^d=0\), the principal eigenvalue of \(A_d\) is \(0\).
We call \(u_d^*\) linearly stable if
\(\operatorname{Re}z<0\) for every \(z\in\sigma(\mathcal L_d)\),
and linearly unstable if
\(\operatorname{Re}z>0\) for some \(z\in\sigma(\mathcal L_d)\).

We now state the stability results for large and small diffusion. 
\begin{theorem}
    \label{thm:large-diffusion-stability}
    For all
    sufficiently large \(d>0\), \(u_d^*\) is linearly stable.
\end{theorem}

{\color{black}For small diffusion, \(\psi_1^d\) concentrates near the maximum set of the
local ratio. We therefore consider the structure of this set. Put}
\[
r(x):=\frac{a(x)}{b(x)},\qquad
r_*:=\max_{\overline\Omega}r,\qquad\text{and}\qquad
M:=\{x\in\overline\Omega:\ r(x)=r_*\}.
\]

\begin{assumptionblock}{H_{\rm m}}
\label{ass:mixed-maximum}
Assume that \(M=M_1\cup\cdots\cup M_J\) is a disjoint finite union. For each
\(j\), one of the following two alternatives holds.
\begin{itemize}
    \item \(M_j\Subset\Omega\) is a compact \(C^2\)-embedded submanifold of
    dimension \(k_j\), \(0\le k_j\le N-1\). There exist a neighborhood
    \(V_j\Subset\Omega\) and constants \(\theta_j^-,\theta_j^+>0\) such that
    \[
    \theta_j^-\operatorname{dist}(x,M_j)^2
    \le r_*-r(x)\le
    \theta_j^+\operatorname{dist}(x,M_j)^2,\qquad x\in V_j .
    \]
    
    \item \(M_j\subset\partial\Omega\) is a compact \(C^2\)-embedded submanifold
    of dimension \(k_j\), \(0\le k_j\le N-1\), and after flattening the boundary
    and reflecting \(r\) evenly, the reflected function \(\widetilde r_e\)
    satisfies
    \[
    \theta_j^-\operatorname{dist}(y,\widetilde M_j)^2
    \le r_*-\widetilde r_e(y)\le
    \theta_j^+\operatorname{dist}(y,\widetilde M_j)^2
    \]
    near the reflected set \(\widetilde M_j\).
\end{itemize}
In addition, \(b(x)=b_*\) on \(U_0\cap\overline\Omega\), where \(b_*>0\),
\(U_0\subset\mathbb R^N\) is open, and \(M\subset U_0\).
\end{assumptionblock}

For a maximum set of positive measure, we use the alternative condition
\begin{assumptionblock}{H_{\rm p}}
\label{ass:plateau}
   \(|M|>0\), \(b|_M\equiv b_*>0\), and there exists a nonempty open
    set \(D\Subset\Omega\) such that \(D\subset M\).
\end{assumptionblock}

\begin{theorem}
    \label{thm:small-diffusion-stability}
    Assume that either \eqref{ass:mixed-maximum} or \eqref{ass:plateau} holds. Then, for all
    sufficiently small \(d>0\), the positive steady state
    \(u_d^*\) is linearly stable.
\end{theorem}

{\color{black}Theorems~\ref{thm:large-diffusion-stability} and
\ref{thm:small-diffusion-stability} give sufficient conditions for linear
stability. Under \eqref{ass:mixed-maximum} or \eqref{ass:plateau},
\(M=\bigcup_{j=1}^{J}M_j\), where
\(M_j=\{x_j\}\Subset\Omega\), \(M_j=\{x_j\}\subset\partial\Omega\),
\(M_j\Subset\Omega\) with \(1\le\dim M_j\le N-1\),
\(M_j\subset\partial\Omega\) with \(1\le\dim M_j\le N-1\),
or \(|M_j|>0\).}

We next consider the case without these constancy assumptions on \(b\).
The spatial variation of \(b\) may also lead to instability.
For fixed sufficiently small \(\varepsilon>0\), let
\(\Omega=(0,1)\). The following problem is parametrized by \(\eta\):
\begin{equation}
\label{eq:hopf-family}
\begin{cases}
\partial_t u
=
\varepsilon^2u''
+
u\left(
a_{\varepsilon,\eta}(x)
-
b_{\varepsilon,\eta}(x)
\displaystyle\int_\Omega c_\varepsilon(y)u(t,y)\,\mathrm{d}y
\right),
& t>0,\ x\in\Omega,\\[2mm]
u'=0,
& t>0,\ x\in\partial\Omega.
\end{cases}
\end{equation}
Here \(\eta\) is the bifurcation parameter and \(c_\varepsilon\) is independent
of \(\eta\). Write the effective local growth profile as
\(q=q_*-V\), where \(q_*\in\mathbb R\) and \(V\) is chosen in
\eqref{eq:counterexample-potential} below. 
Let \((\nu_\varepsilon,\phi_\varepsilon)\) be the principal eigenpair of
\(\varepsilon^2\frac{\mathrm{d}^2}{\mathrm{d}x^2}+q(x)\) with Neumann boundary
condition and \(\|\phi_\varepsilon\|_{L^2(\Omega)}=1\). {\color{black}In the family
\eqref{eq:hopf-family}, the three coefficients satisfy}
\[
a_{\varepsilon,\eta}(x)
-
b_{\varepsilon,\eta}(x)
\int_\Omega c_\varepsilon(y)\phi_\varepsilon(y)\,\mathrm{d}y
=
q(x)-\nu_\varepsilon .
\]

\Needspace{12\baselineskip}
\begin{theorem}
    \label{thm:three-peak-hopf}
    For all sufficiently small \(\varepsilon>0\), there exist smooth
    coefficients \(a_{\varepsilon,\eta}\), \(b_{\varepsilon,\eta}\), and
    \(c_\varepsilon\) in \eqref{eq:hopf-family}, for which
    \(u_\varepsilon^*:=\phi_\varepsilon\) is a positive steady state for
    \(\eta\) near \([0,1]\). There are \(\eta_0\in(0,1)\),
    a critical value \(\eta_\varepsilon=\eta_0+o(1)\) as
    \(\varepsilon\to0^+\), and \(\delta_\varepsilon>0\) such that:
    \begin{enumerate}
        \item\label{item:hopf-crossing}
        The steady state \(u_\varepsilon^*\) is linearly stable for
        \(\eta\in(\eta_\varepsilon-\delta_\varepsilon,\eta_\varepsilon)\)
        and linearly unstable for
        \(\eta\in(\eta_\varepsilon,\eta_\varepsilon+\delta_\varepsilon)\).

        \item\label{item:hopf-branches}
        At \(\eta=\eta_\varepsilon\), the positive steady state
        \(u_\varepsilon^*\) undergoes a Hopf bifurcation.
        The coefficients can be chosen so that the bifurcation is
        supercritical and the bifurcating periodic solutions are orbitally
        asymptotically stable, or subcritical and the bifurcating periodic
        solutions are unstable.
    \end{enumerate}
\end{theorem}

In our construction, we reduce the spectral problem for the linearized operator $\mathcal{L}_{d}$
to that of a matrix whose dimension equals the number of maximum
points of \(q\). Proposition~\ref{prop:one-two-mode-reduction}
shows that matrices of this form are stable in dimensions one and two.
We therefore choose \(q\) with three maximum points \(x_1<x_2<x_3\) to construct unstable
positive steady states and Hopf bifurcations. 

The two constructions use the same \(q\), \(a_{\varepsilon,\eta}\), and
\(b_{\varepsilon,\eta}\). Two different choices of \(c_\varepsilon\)
near \(x_2\) and \(x_3\) yield supercritical and subcritical Hopf
bifurcations, while \(c_\varepsilon\) remains unchanged in a fixed
neighborhood of \(x_1\); see Figure~\ref{fig:three-peak-coeff-profiles}.
In this construction, for each fixed sufficiently small \(\varepsilon\),
the values \(c_\varepsilon(x_2)\) and \(c_\varepsilon(x_3)\) satisfy
an affine relation; see the proof of Lemma~\ref{lem:positive-profiles}.
But we do not investigate the occurrence or type of Hopf bifurcation
for the intermediate choices of \(c_\varepsilon\).
\begin{figure}[!htbp]
\centering
\includegraphics[trim=2pt 2pt 2pt 2pt,clip]{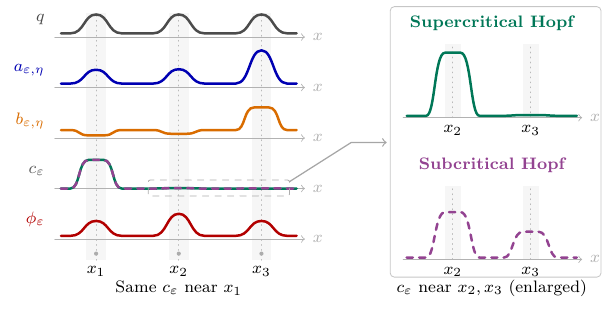}
\caption{Schematic coefficient profiles.}
\label{fig:three-peak-coeff-profiles}
\end{figure}

\section{Proofs of the stability results}
\label{sec:stability-proofs}

{\color{black}In this section, we prove the stability results in
Section~\ref{sec:main-results}. We first establish a spectral criterion, and then
consider large and small diffusion separately.}

Write \(B_d:=b\,u_d^*\). Then
\(\mathcal L_d=A_d-B_d\otimes c\), where
\((B_d\otimes c)\eta
=B_d(x)\int_\Omega c(y)\eta(y)\,\mathrm{d}y\). Let
\(0=\mu_1(d)>\mu_2(d)\ge\mu_3(d)\ge\cdots\) be the eigenvalues of \(A_d\),
and complete \(\psi_1^d\) to an orthonormal basis
\(\{\psi_k^d\}_{k\ge1}\) of eigenfunctions.

{\color{black}Motivated by the spectral results of Coville and
Girardin~\cite{CovilleGirardin2025}, we obtain the following stability
criterion for the present operator with a self-adjoint diffusion part.}

\begin{lemma}
    \label{thm:krein-criterion}
    The spectrum \(\sigma(\mathcal L_d)\) lies in
    \(\{\lambda\in\mathbb C:\operatorname{Re}\lambda<0\}\) provided that
    \begin{equation}
    \begin{aligned}
    &\langle c,B_d\rangle
    +\langle c,\psi_1^d\rangle\langle B_d,\psi_1^d\rangle
    >
    \big\|c-\langle c,\psi_1^d\rangle\psi_1^d\big\|_{L^2(\Omega)}
    \big\|B_d-\langle B_d,\psi_1^d\rangle\psi_1^d\big\|_{L^2(\Omega)}.
    \end{aligned}
    \label{eq:krein-condition}
    \end{equation}
\end{lemma}

\begin{proof}
   We first show that \(0\notin\sigma(\mathcal L_d)\). If
    \((A_d-B_d\otimes c)u=0\), taking the inner product with
    \(\psi_1^d\) gives
    \begin{equation}
    \langle B_d,\psi_1^d\rangle\langle c,u\rangle=0.
    \label{eq:zero-eigenvalue-projection}
    \end{equation}
    Since \(\langle B_d,\psi_1^d\rangle>0\),
    \eqref{eq:zero-eigenvalue-projection} implies \(\langle c,u\rangle=0\),
    and consequently \(A_du=0\). Hence \(u=\theta\psi_1^d\) for some
    \(\theta\in\mathbb C\). As \(\langle c,\psi_1^d\rangle>0\), the equality
    \(\theta\langle c,\psi_1^d\rangle=0\) gives \(u=0\).
    Thus \(0\) is not an eigenvalue of \(\mathcal L_d\).

    {\color{black}Suppose, by contradiction, that
    \(z\in\sigma(\mathcal L_d)\) and \(\operatorname{Re}z\ge0\).}
    Since \(A_d\) has compact resolvent and \(B_d\otimes c\) is bounded,
    \(\mathcal L_d=A_d-B_d\otimes c\) also has compact resolvent.
    {\color{black}Thus, \(z\) is an eigenvalue, and the above argument implies \(z\ne0\).}
    The operator \(A_d\) is self-adjoint and nonpositive, with
    \(\sigma(A_d)\setminus\{0\}\subset(-\infty,0)\).
    It follows that \(z\in\rho(A_d)\). For a corresponding eigenfunction \(u\ne0\), one has
    \[
    (z-A_d)u=-B_d\langle c,u\rangle_{L^2(\Omega)}.
    \]
    {\color{black}We note that \(\langle c,u\rangle_{L^2(\Omega)}\) cannot vanish, since
    otherwise \((z-A_d)u=0\). Applying \((z-A_d)^{-1}\), taking the inner
    product with \(c\), and dividing by this scalar, we obtain}
    \begin{equation}
    1+\langle c,(z-A_d)^{-1}B_d\rangle_{L^2(\Omega)}=0.
    \label{eq:rank-one-resolvent}
    \end{equation}

    Decompose \(B_d\) and \(c\) along \(\psi_1^d\) and its orthogonal
    complement. Since \((z-A_d)^{-1}\psi_1^d=z^{-1}\psi_1^d\),
    equation \eqref{eq:rank-one-resolvent} becomes
    \begin{equation}
    \begin{aligned}
    0={}&1+\frac{\langle B_d,\psi_1^d\rangle\langle c,\psi_1^d\rangle}{z}
    +\Big\langle c-\langle c,\psi_1^d\rangle\psi_1^d,
    (z-A_d)^{-1}\bigl(B_d-\langle B_d,\psi_1^d\rangle\psi_1^d\bigr)\Big\rangle.
    \end{aligned}
    \label{eq:rank-one-split}
    \end{equation}
    The identity
    \begin{equation}
    \begin{aligned}
    \langle c,B_d\rangle
    ={}&\langle c,\psi_1^d\rangle\langle B_d,\psi_1^d\rangle
   +\Big\langle c-\langle c,\psi_1^d\rangle\psi_1^d,
    B_d-\langle B_d,\psi_1^d\rangle\psi_1^d\Big\rangle
    \end{aligned}
    \label{eq:principal-orthogonal-identity}
    \end{equation}
    together with \eqref{eq:rank-one-split} gives
    \begin{equation}
    \begin{aligned}
    &2z+\langle c,B_d\rangle
    +\langle c,\psi_1^d\rangle\langle B_d,\psi_1^d\rangle
    \quad=\Big\langle c-\langle c,\psi_1^d\rangle\psi_1^d,
    [I-2z(z-A_d)^{-1}]
    \bigl(B_d-\langle B_d,\psi_1^d\rangle\psi_1^d\bigr)\Big\rangle.
    \end{aligned}
    \label{eq:krein-identity}
    \end{equation}

    {\color{black}For every \(k\ge2\), we have}
    \[
    [I-2z(z-A_d)^{-1}]\psi_k^d
    =
    \frac{\mu_k(d)+z}{\mu_k(d)-z}\psi_k^d
    \]
    and
    \[
    |\mu_k(d)-z|^2-|\mu_k(d)+z|^2
    =-4\mu_k(d)\operatorname{Re}z\ge0.
    \]
    {\color{black}By Parseval's identity, we obtain}
    \begin{equation}
    \|[I-2z(z-A_d)^{-1}]f\|_{L^2(\Omega)}
    \le\|f\|_{L^2(\Omega)},
    \qquad f\perp\psi_1^d.
    \label{eq:cayley-contraction}
    \end{equation}
    Apply \eqref{eq:cayley-contraction} to
    \(f=B_d-\langle B_d,\psi_1^d\rangle\psi_1^d\). By Cauchy--Schwarz, the
    modulus of the right-hand side of \eqref{eq:krein-identity} is at most
    \(\|c-\langle c,\psi_1^d\rangle\psi_1^d\|_{L^2(\Omega)}
    \|B_d-\langle B_d,\psi_1^d\rangle\psi_1^d\|_{L^2(\Omega)}\).
    On the other hand, \(\operatorname{Re}z\ge0\) and
    \eqref{eq:krein-condition} imply that the real part of its left-hand
    side is strictly larger than this quantity. This contradiction proves
    \(\sigma(\mathcal L_d)\subset\{z:\operatorname{Re}z<0\}\).
\end{proof}

\subsection{Large-diffusion stability}
\label{sec:large-diffusion}

{\color{black}In this subsection, we prove that \(u_d^*\) is linearly stable for all sufficiently large \(d\).}

\begin{lemma}
    \label{lem:large-diffusion-gap}
    There exist constants \(\gamma_0>0\) and \(d_0>0\) such that
    \[
    -\mu_2(d)\ge\gamma_0d,
    \qquad d>d_0 .
    \]
\end{lemma}

\begin{proof}
    {\color{black}Suppose, by contradiction, that there is a sequence \(d_n\to\infty\) as \(n\to\infty\)}
    such that \(-\mu_2(d_n)=o(d_n)\). Let \(v_n\) be a real-valued,
    \(L^2\)-normalized eigenfunction
    of \(A_{d_n}\) corresponding to \(\mu_2(d_n)\). Then
    \[
    \langle v_n,\psi_1^{d_n}\rangle_{L^2(\Omega)}=0,
    \qquad
    \|v_n\|_{L^2(\Omega)}=1,
    \]
    and
    \[
    d_n\int_\Omega|\nabla v_n(x)|^2\,\mathrm{d}x
    -
    \int_\Omega\bigl(a(x)-\lambda_{d_n}b(x)\bigr)v_n^2(x)\,\mathrm{d}x
    =
    -\mu_2(d_n)
    =
    o(d_n).
    \]
    The principal eigenvalue in \eqref{eq:principal-eigenpair} satisfies
    \begin{equation}
    \lambda_d
    =
    \sup_{0\ne\psi\in H^1(\Omega)}
    \frac{\displaystyle
        \int_\Omega a(x)\psi^2(x)\,\mathrm{d}x
        -d\int_\Omega|\nabla\psi(x)|^2\,\mathrm{d}x}
    {\displaystyle\int_\Omega b(x)\psi^2(x)\,\mathrm{d}x}.
    \label{eq:weighted-variational}
    \end{equation}

    Formula \eqref{eq:weighted-variational} and assumption
    \eqref{eq:positive-mean-stability} give \(0<\lambda_d\le r_*\).
    Since \(a\) and \(b\) are bounded, there is a constant \(C_1>0\),
    independent of \(n\), such that
    \(\|a-\lambda_{d_n}b\|_{L^\infty(\Omega)}\le C_1\).
    It follows from the above energy identity that
    \[
    \int_\Omega|\nabla v_n(x)|^2\,\mathrm{d}x
    \le
    \frac{-\mu_2(d_n)+C_1}{d_n}\to0
    \qquad(n\to\infty).
    \]
    Set \(\overline v_n:=|\Omega|^{-1}\int_\Omega v_n(x)\,\mathrm{d}x\).
    Poincar\'e's inequality gives
    \(\|v_n-\overline v_n\|_{L^2(\Omega)}\to0\) as \(n\to\infty\). Since
    \(\|v_n\|_{L^2(\Omega)}=1\), after passing to a subsequence one has
    \(\overline v_n\to\kappa\) as \(n\to\infty\), with \(|\kappa|=|\Omega|^{-1/2}\).
    Hence \(v_n\to\kappa\) in \(L^2(\Omega)\) as \(n\to\infty\). On the other hand,
    the argument in \cite[Theorem~2]{AllenBolkerLouNevai2008}, applied
    to \eqref{eq:principal-eigenpair}, gives
    \(\psi_1^{d_n}\to|\Omega|^{-1/2}\) in \(H^1(\Omega)\) as \(n\to\infty\).
    Together with \(\langle v_n,\psi_1^{d_n}\rangle_{L^2(\Omega)}=0\),
    this gives
    \[
    0
    =
    \lim_{n\to\infty}\langle v_n,\psi_1^{d_n}\rangle_{L^2(\Omega)}
    =
    \kappa|\Omega|^{1/2},
    \]
    a contradiction.
\end{proof}

\begin{proof}[Proof of Theorem~\ref{thm:large-diffusion-stability}]
    {\color{black}Suppose, by contradiction, that \(z\) is an eigenvalue of}
    \(\mathcal L_d\) with \(\operatorname{Re}z\ge0\).
    The proof of Lemma~\ref{thm:krein-criterion} shows that \(0\) is not
    an eigenvalue, so \(z\ne0\).
    Equation~\eqref{eq:rank-one-split} gives
    \begin{equation}
    1+\frac{m_d}{z}+R_d(z)=0,
    \label{eq:large-resolvent-scalar}
    \end{equation}
    where
    \begin{equation}
    \begin{aligned}
    m_d&:=\langle B_d,\psi_1^d\rangle\langle c,\psi_1^d\rangle,\quad
    R_d(z):=\bigl\langle
    c-\langle c,\psi_1^d\rangle\psi_1^d,
    (z-A_d)^{-1}(B_d-\langle B_d,\psi_1^d\rangle\psi_1^d)
    \bigr\rangle.
    \end{aligned}
    \label{eq:large-resolvent-coefficients}
    \end{equation}
    Since \(\operatorname{Re}z\ge0\), Lemma~\ref{lem:large-diffusion-gap} gives
    \(|z-\mu_k(d)|\ge-\mu_k(d)\ge-\mu_2(d)\ge\gamma_0d\)
    for every \(k\ge2\) and \(d>d_0\). Consequently,
    \[
    \|(z-A_d|_{(\psi_1^d)^\perp})^{-1}\|
    \le
    \frac{1}{\gamma_0d},
    \qquad d>d_0.
    \]

    Applying the argument in \cite[Theorem~2]{AllenBolkerLouNevai2008}
    to \eqref{eq:principal-eigenpair}, we have
    \(\psi_1^d\to|\Omega|^{-1/2}\) in \(H^1(\Omega)\) as \(d\to\infty\).
    Consequently,
    \[
    \langle c,\psi_1^d\rangle\longrightarrow
    |\Omega|^{-1/2}\int_\Omega c(x)\,\mathrm{d}x>0
    \qquad(d\to\infty),
    \]
    and
    \[
    \int_\Omega b(x)\bigl(\psi_1^d(x)\bigr)^2\,\mathrm{d}x
    \longrightarrow
    |\Omega|^{-1}\int_\Omega b(x)\,\mathrm{d}x>0
    \qquad(d\to\infty).
    \]
    Since \(B_d=\lambda_d b\psi_1^d/\langle c,\psi_1^d\rangle\), we have
    \(m_d=\lambda_d\int_\Omega b(x)\bigl(\psi_1^d(x)\bigr)^2\,\mathrm{d}x\).
    Moreover, the argument in \cite[Theorem~2]{AllenBolkerLouNevai2008} implies
    \begin{equation}
    \lambda_d\to
    \int_\Omega a(x)\,\mathrm{d}x\big/
    \int_\Omega b(x)\,\mathrm{d}x
    \qquad(d\to\infty).
    \label{eq:large-diffusion-eigenvalue-limit}
    \end{equation}
    Together with \(\|\psi_1^d\|_{L^2(\Omega)}=1\),
    \eqref{eq:large-diffusion-eigenvalue-limit},
    and \eqref{eq:positive-mean-stability}, these limits imply that there exist
    \(d_1\ge d_0\), \(m_0>0\), and \(C_1,C_2>0\) such that, for every
    \(d>d_1\), one has
    \[
    m_d\ge m_0,
    \qquad
    \|B_d\|_{L^2(\Omega)}\le C_1,
    \qquad\text{and}\qquad
    \|c\|_{L^2(\Omega)}\le C_2.
    \]
    Orthogonal projection does not increase the norm. Consequently,
    \eqref{eq:large-resolvent-coefficients} and the resolvent bound above give
    \[
    |R_d(z)|
    \le
    \frac{1}{\gamma_0d}
    \|c\|_{L^2(\Omega)}
    \|B_d\|_{L^2(\Omega)}
    \le
    \frac{C_3}{d},
    \]
    where \(C_3:=C_1C_2/\gamma_0\).
    Choose \(d_2\ge d_1\) so that \(C_3/d_2\le1/2\). For \(d>d_2\),
    \(|R_d(z)|\le1/2\), and therefore
    \begin{equation}
    \operatorname{Re}\frac{1}{1+R_d(z)}
    =
    \frac{1+\operatorname{Re}R_d(z)}{|1+R_d(z)|^2}
    \ge
    \frac{2}{9}>0.
    \label{eq:large-reciprocal-positive}
    \end{equation}
    Multiplying \eqref{eq:large-resolvent-scalar} by \(z\) gives
    \(z=-m_d(1+R_d(z))^{-1}\). Since \(m_d>0\),
    \eqref{eq:large-reciprocal-positive} implies \(\operatorname{Re}z<0\),
    a contradiction. As \(\mathcal L_d\) has compact resolvent, this proves
    \(\sigma(\mathcal L_d)\subset\{z:\operatorname{Re}z<0\}\).
\end{proof}

{\color{black}Next, we consider the small-diffusion case. For a lower-dimensional maximum
set, we first estimate the concentration of the principal eigenfunction. For
a maximum set of positive measure, we obtain the concentration from an energy
identity. In both cases, we verify the criterion in
Lemma~\ref{thm:krein-criterion}.}

\subsection{Small-diffusion concentration and stability}
\label{sec:small-diffusion}

{\color{black}Near the boundary components of \(M\), we choose finitely many
\(C^2\) charts \(x=\Theta(y)\) mapping half-balls in \(\{y_N>0\}\)
onto local portions of \(\Omega\). In each chart, let
\(\widetilde r_e(y',y_N)=r(\Theta(y',|y_N|))\). The local image
\(\widetilde M_j\) lies in \(\{y_N=0\}\), and the chart bounds can be
chosen uniformly over the finite collection.}

Recall the \(L^2\)-normalized eigenfunction \(\psi_1^d\) fixed in \eqref{eq:principal-eigenpair}. For
the pointwise estimates below, define
\[
\varphi_d:=\frac{\psi_1^d}{\|\psi_1^d\|_{L^\infty(\Omega)}}
\]
and choose \(y_d\in\overline\Omega\) such that \(\varphi_d(y_d)=1\). Thus
\(\|\varphi_d\|_{L^\infty(\Omega)}=1\) and
\(\psi_1^d=\varphi_d/\|\varphi_d\|_{L^2(\Omega)}\).

\begin{lemma}
    \label{lem:mixed-maximum-concentration}
    Assume \eqref{ass:mixed-maximum}. Then there are constants
    \(K_\lambda,K_y,K_1,K_2,d_0>0\) such that, for
    \(0<d<d_0\), the following assertions hold.
    \begin{enumerate}
        \item\label{item:concentration-eigenvalue} \(0\le r_*-\lambda_d\le K_\lambda d^{1/2}\).
        
        \item\label{item:concentration-maximum} \(\operatorname{dist}(y_d,M)\le K_y d^{1/4}\).
        
        \item\label{item:concentration-mass} \(\|\varphi_d\|_{L^2(\Omega)}^2\ge K_2d^{N/4}\) and
        \(\|\varphi_d\|_{L^1(\Omega)}\ge K_1d^{N/4}\).
        
        \item\label{item:concentration-tail} For every relatively open neighborhood \(U\) of \(M\) in
        \(\overline\Omega\), there exist constants \(K_U,\kappa_U>0\) such that
        \[
        \|\psi_1^d\|_{L^2(\Omega\setminus U)}^2
        \le K_Ue^{-\kappa_Ud^{-1/2}}.
        \]
    \end{enumerate}
\end{lemma}

\begin{proof}
    Set \(b_-:=\min\limits_{x\in\overline\Omega}b(x)>0\).

    \noindent\textup{(\ref{item:concentration-eigenvalue})}
    Since \(a=br\) and \(r\le r_*\), formula
    \eqref{eq:weighted-variational} gives \(\lambda_d\le r_*\).
    For the lower bound, choose one component \(M_j\). Suppose first that
    \(M_j\Subset\Omega\). Fix \(x_j\in M_j\), take
    \(\eta\in C_c^\infty(B_1(0))\), \(\eta\not\equiv0\), and define
    \[
    \eta_d(x)=\eta\left(\frac{x-x_j}{d^{1/4}}\right).
    \]
    For sufficiently small \(d>0\), \(\operatorname{supp}\eta_d\) is contained in the
    neighborhood \(V_j\) from \eqref{ass:mixed-maximum}. On this support, one has
    \[
    \operatorname{dist}(x,M_j)\le |x-x_j|\le d^{1/4},
    \qquad\text{and}\qquad
    r(x)\ge r_*-\theta_j^+d^{1/2}.
    \]
    Inserting \(\psi=\eta_d\) in \eqref{eq:weighted-variational} gives
    \[
    \lambda_d
    \ge
    r_*-\theta_j^+d^{1/2}
    -
    \frac{d}{b_-}\,
    \frac{\int_\Omega|\nabla\eta_d|^2\,\mathrm{d}x}
    {\int_\Omega\eta_d^2\,\mathrm{d}x}.
    \]
    The change of variables \(x=x_j+d^{1/4}z\) gives
    \[
    \int_\Omega\eta_d^2(x)\,\mathrm{d}x
    =d^{N/4}\int_{\mathbb R^N}\eta^2(z)\,\mathrm{d}z,
    \qquad\text{and}\qquad
    \int_\Omega|\nabla\eta_d(x)|^2\,\mathrm{d}x
    =d^{N/4-1/2}\int_{\mathbb R^N}|\nabla\eta(z)|^2\,\mathrm{d}z.
    \]
    Hence
    \[
    \frac{\int_\Omega|\nabla\eta_d(x)|^2\,\mathrm{d}x}
    {\int_\Omega\eta_d^2(x)\,\mathrm{d}x}
    =
    K_\eta d^{-1/2},
    \qquad
    K_\eta:=
    \frac{\int_{\mathbb R^N}|\nabla\eta(z)|^2\,\mathrm{d}z}
    {\int_{\mathbb R^N}\eta^2(z)\,\mathrm{d}z},
    \]
    and therefore
    \[
    \lambda_d\ge r_*-K_{\lambda,I}d^{1/2},
    \qquad
    K_{\lambda,I}:=\theta_j^++b_-^{-1}K_\eta.
    \]
    
    Now suppose that \(M_j\subset\partial\Omega\). Fix \(x_j\in M_j\) and
    choose one of the boundary charts \(x=\Theta(y)\) such that
    \[
    \Theta(0)=x_j,\qquad
    \Theta(B_\rho^+)=\Omega\cap U,\qquad
    \text{and}\qquad
    B_\rho^+=B_\rho(0)\cap\{y_N>0\}.
    \]
    Take \(\eta\in C_c^\infty(B_1^+)\), \(\eta\not\equiv0\), and set
    \[
    \psi_d(\Theta(y))=\eta\left(\frac{y}{d^{1/4}}\right),
    \qquad y\in B_\rho^+,
    \]
    with \(\psi_d=0\) outside \(\Theta(B_\rho^+)\). The reflected quadratic
    assumption in \eqref{ass:mixed-maximum} gives
    \(r(\Theta(y))\ge r_*-K_{r,j}d^{1/2}\) on
    \(\operatorname{supp}\psi_d\) for some \(K_{r,j}>0\).
    The fixed chart and its inverse have bounded derivatives. Thus there are
    constants \(K_{j,0},K_{j,1},K_{j,2}>0\), independent of \(d\), such that
    \[
    K_{j,0}d^{N/4}\le \int_\Omega\psi_d^2(x)\,\mathrm{d}x
    \le K_{j,1}d^{N/4},
    \qquad\text{and}\qquad
    \int_\Omega|\nabla\psi_d(x)|^2\,\mathrm{d}x
    \le K_{j,2}d^{N/4-1/2}.
    \]
    {\color{black}Taking \(\psi=\psi_d\) in \eqref{eq:weighted-variational}, we again obtain}
    \[
    \lambda_d\ge r_*-K_{\lambda,B}d^{1/2},
    \qquad
    K_{\lambda,B}:=K_{r,j}+\frac{K_{j,2}}{b_-K_{j,0}}.
    \]
    {\color{black}Since \(M\) has finitely many components, we obtain~\textup{(\ref{item:concentration-eigenvalue})} by taking
    \(K_\lambda\) to be the largest of the constants for the interior and
    boundary components.}
    
    \medskip
    \noindent\textup{(\ref{item:concentration-maximum})} Recall that
    \(\varphi_d(y_d)=1\). If \(a(y_d)-\lambda_db(y_d)<0\), continuity
    and positivity give \(\Delta\varphi_d>0\) near \(y_d\), in the weak
    sense. The maximum principle excludes an interior maximum there,
    while the Hopf boundary point lemma contradicts the Neumann condition
    if \(y_d\in\partial\Omega\); see
    \cite[Chapter~3]{GilbargTrudinger2001}. Hence
    \(a(y_d)-\lambda_db(y_d)\ge0\), so \(r(y_d)\ge\lambda_d\).
    By assertion~\textup{(\ref{item:concentration-eigenvalue})}, one has
    \begin{equation}
    0\le r_*-r(y_d)\le r_*-\lambda_d\le K_\lambda d^{1/2}.
    \label{eq:concentration-maximum-gap}
    \end{equation}
    For each component \(M_j\), choose a relative neighborhood
    \(V_j\subset\overline\Omega\) on which the corresponding lower bound in
    \eqref{ass:mixed-maximum} holds, in the original variables if \(M_j\Subset\Omega\) and
    in flattened reflected variables if \(M_j\subset\partial\Omega\).
    Choose these neighborhoods small enough that
    \(V:=\bigcup_j V_j\ne\overline\Omega\). Since \(r<r_*\) on
    \(\overline\Omega\setminus V\), one has
    \[
    m_V:=\min_{\overline\Omega\setminus V}(r_*-r)>0.
    \]
    For \(0<d<(m_V/K_\lambda)^2\), estimate
    \eqref{eq:concentration-maximum-gap} gives
    \(y_d\in V\). Hence
    \(y_d\in V_j\) for some \(j\). If \(M_j\Subset\Omega\), then
    \[
    \theta_j^-\operatorname{dist}(y_d,M_j)^2
    \le
    r_*-r(y_d)
    \le
    K_\lambda d^{1/2},
    \]
    and therefore
    \[
    \operatorname{dist}(y_d,M)
    \le
    \left(\frac{K_\lambda}{\theta_j^-}\right)^{1/2}d^{1/4}.
    \]
    If
    \(M_j\subset\partial\Omega\), let \(z_d\) be the coordinate of
    \(y_d\) in one of the fixed boundary charts. Then
    \[
    \theta_j^-\operatorname{dist}(z_d,\widetilde M_j)^2
    \le
    r_*-\widetilde r_e(z_d)
    =
    r_*-r(y_d)
    \le
    K_\lambda d^{1/2}.
    \]
    Since the finitely many flattening maps are \(C^2\)-diffeomorphisms, there
    is a constant \(K_{\rm F}>0\), independent of \(d\) and \(j\), such that
    \[
    \operatorname{dist}(y_d,M)
    \le
    K_{\rm F}\operatorname{dist}(z_d,\widetilde M_j)
    \le
    K_{\rm F}
    \left(\frac{K_\lambda}{\theta_j^-}\right)^{1/2}d^{1/4}.
    \]
    {\color{black}Taking \(K_y\) to be the maximum of the finitely many constants in the
    last two estimates, we obtain~\textup{(\ref{item:concentration-maximum})}.}
    
    \medskip
    \noindent\textup{(\ref{item:concentration-mass})} Fix \(R>0\). In the interior case, if
    \(|x-y_d|\le4Rd^{1/4}\), then
    \[
    \operatorname{dist}(x,M)
    \le
    |x-y_d|+\operatorname{dist}(y_d,M)
    \le
    K_{R,0}d^{1/4},
    \qquad K_{R,0}:=4R+K_y.
    \]
    In a boundary chart, the corresponding distance estimate follows from
    the uniform Lipschitz bounds for the chart and its inverse. The quadratic
    upper bounds in \eqref{ass:mixed-maximum} therefore give
    \[
    0\le r_*-r(x)\le K_{R,1}d^{1/2},
    \]
    with \(K_{R,1}>0\) independent of \(d\). Combining this bound with
    \textup{(\ref{item:concentration-eigenvalue})} and using \(a=br\) and the
    boundedness of \(b\), we obtain a constant \(K_{R,2}>0\), independent
    of \(d\), such that
    \begin{equation}
    |a(x)-\lambda_db(x)|\le K_{R,2}d^{1/2}
    \label{eq:concentration-local-potential}
    \end{equation}
    on the corresponding ball or flattened half-ball.
    
    After rescaling by \(d^{1/4}\),
    \eqref{eq:concentration-local-potential} gives a uniform bound on the
    {\color{black}zeroth-order coefficient. We apply the interior Harnack inequality
    \cite[Theorem~8.20]{GilbargTrudinger2001}. In the boundary case, we first
    flatten the boundary and use Neumann reflection as in
    \cite[Remark~3.10]{BassHsu1991}. Since \(\varphi_d(y_d)=1\), we obtain
    constants} \(\rho_R,\theta_R>0\), independent of \(d\), such that,
    for all sufficiently small \(d>0\),
    \begin{equation}
    \varphi_d(x)\ge\theta_R>0,
    \qquad x\in\Omega\cap B_{\rho_Rd^{1/4}}(y_d)=:E_d.
    \label{eq:concentration-harnack-lower}
    \end{equation}
    The \(C^2\) regularity of the boundary gives the volume bound below,
    while \eqref{eq:concentration-harnack-lower} gives the pointwise bound:
    for some \(K_E>0\) independent of \(d\),
    \[
    |E_d|\ge K_E d^{N/4}.
    \]
    Consequently, one has
    \[
    \|\varphi_d\|_{L^2(\Omega)}^2
    \ge
    \int_{E_d}\varphi_d^2(x)\,\mathrm{d}x
    \ge \theta_R^2K_E d^{N/4},
    \qquad\text{and}\qquad
    \|\varphi_d\|_{L^1(\Omega)}
    \ge
    \int_{E_d}\varphi_d(x)\,\mathrm{d}x
    \ge \theta_RK_E d^{N/4}.
    \]
    Thus~\textup{(\ref{item:concentration-mass})} holds with \(K_2:=\theta_R^2K_E\) and
    \(K_1:=\theta_RK_E\).
    
    \medskip
    \noindent\textup{(\ref{item:concentration-tail})} Fix a relatively open neighborhood \(U\) of \(M\). Then
    the assertion is immediate if \(U=\overline\Omega\). Otherwise, define
    \[
    \sigma_U:=\min_{\overline\Omega\setminus U}(r_*-r)>0.
    \]
    By~\textup{(\ref{item:concentration-eigenvalue})}, \(r_*-\lambda_d\le\sigma_U/2\) for all small \(d\). Therefore, for
    \(x\in\overline\Omega\setminus U\), one has
    \[
    r(x)-\lambda_d
    \le
    r_*-\sigma_U-\lambda_d
    \le
    -\frac{\sigma_U}{2}.
    \]
    Since \(a-\lambda_db=b(r-\lambda_d)\), one has
    \[
    a(x)-\lambda_db(x)\le-\delta_U
    \quad (x\in\overline\Omega\setminus U),
    \qquad\text{where}\qquad
    \delta_U:=\frac{\sigma_U}{2}\min_{\overline\Omega}b>0.
    \]
    By \textup{(\ref{item:concentration-eigenvalue})},
    \(a-\lambda_db\to b(r-r_*)\) uniformly on \(\overline\Omega\) as \(d\to0^+\).
    Since the limit is negative outside \(M\), the localization argument
    in the proof of \cite[Proposition~1.3.17]{LamLou2022}, including its
    Neumann boundary argument, gives, for some \(\kappa_{U,0}>0\) and all
    sufficiently small \(d>0\),
    \begin{equation}
    \sup_{\Omega\setminus U}\varphi_d
    \le e^{-\kappa_{U,0}d^{-1/2}}.
    \label{eq:concentration-pointwise-tail}
    \end{equation}
    By~\textup{(\ref{item:concentration-mass})}, one has
    \[
    \|\varphi_d\|_{L^2(\Omega)}^2\ge K_2d^{N/4}.
    \]
    Using \eqref{eq:concentration-pointwise-tail} and
    \(\psi_1^d=\varphi_d/\|\varphi_d\|_{L^2(\Omega)}\), one obtains
    \[
    \int_{\Omega\setminus U}\bigl(\psi_1^d(x)\bigr)^2\,\mathrm{d}x
    =
    \frac{\displaystyle\int_{\Omega\setminus U}\varphi_d^2(x)\,\mathrm{d}x}
    {\displaystyle\|\varphi_d\|_{L^2(\Omega)}^2}
    \le
    \frac{|\Omega|}{K_2}d^{-N/4}e^{-2\kappa_{U,0}d^{-1/2}}.
    \]
    Choose \(\kappa_U\in(0,2\kappa_{U,0})\). Since
    \(d^{-N/4}e^{-(2\kappa_{U,0}-\kappa_U)d^{-1/2}}\) is bounded for
    \(0<d<d_0\), there is a constant \(K_U>0\) such that, for all
    sufficiently small \(d>0\),
    \[
    \|\psi_1^d\|_{L^2(\Omega\setminus U)}^2
    \le K_Ue^{-\kappa_Ud^{-1/2}}.
    \]
    {\color{black}Increasing \(K_U\) if necessary, we obtain the estimate for all \(0<d<d_0\), since
    \(\|\psi_1^d\|_{L^2(\Omega)}=1\).}
    This proves~\textup{(\ref{item:concentration-tail})}.
\end{proof}

\begin{proof}[Proof of Theorem~\ref{thm:small-diffusion-stability}]
    Recall that \(\|\psi_1^d\|_{L^2(\Omega)}=1\),
    \(u_d^*=\alpha_d\psi_1^d\), and
    \(\alpha_d=\lambda_d/\langle c,\psi_1^d\rangle\), with
    \(\langle c,\psi_1^d\rangle>0\). {\color{black}In both cases, set}
    \[
    \bar b_d:=\int_\Omega b(x)\bigl(\psi_1^d(x)\bigr)^2\,\mathrm{d}x .
    \]
    Then \(\langle B_d,\psi_1^d\rangle=\alpha_d\bar b_d\), and one has
    \begin{equation}
    B_d-\langle B_d,\psi_1^d\rangle\psi_1^d
    =\alpha_d(b-\bar b_d)\psi_1^d .
    \label{eq:small-b-perp}
    \end{equation}

    \noindent Assume \eqref{ass:mixed-maximum} and set
    \(U=U_0\cap\overline\Omega\).
    Lemma~\ref{lem:mixed-maximum-concentration}\textup{(\ref{item:concentration-tail})} gives constants
    \(d_{b,0},K_U,\kappa_U>0\) such that, for \(0<d<d_{b,0}\),
    \(\int_{\Omega\setminus U}\bigl(\psi_1^d(x)\bigr)^2\,\mathrm{d}x
    \le K_Ue^{-\kappa_Ud^{-1/2}}\). 
    Since \(b=b_*\) on \(U\), splitting the integrals over \(U\) and
    \(\Omega\setminus U\) and using \(\|\psi_1^d\|_{L^2(\Omega)}=1\) yield
    \begin{equation}
    |\bar b_d-b_*|
    \le K_{b,1}e^{-\kappa_b d^{-1/2}},
    \qquad\text{and}\qquad
    \|(b-\bar b_d)\psi_1^d\|_{L^2(\Omega)}
    \le K_{b,1}e^{-\kappa_b d^{-1/2}},
    \label{eq:small-b-flat}
    \end{equation}
    where
    \[
    \kappa_b:=\frac{\kappa_U}{2},
    \qquad
    K_{b,1}:=2\|b\|_{L^\infty(\Omega)}\bigl(K_U+K_U^{1/2}\bigr).
    \]
    Since \(\psi_1^d=\varphi_d/\|\varphi_d\|_{L^2(\Omega)}\) and
    \(\|\varphi_d\|_{L^2(\Omega)}\le|\Omega|^{1/2}\),
    Lemma~\ref{lem:mixed-maximum-concentration}\textup{(\ref{item:concentration-mass})} gives
    \begin{equation}
    \langle c,\psi_1^d\rangle
    \ge
    \left(\min_{x\in\overline\Omega}c(x)\right)
    \int_\Omega\psi_1^d(x)\,\mathrm{d}x
    \ge K_\gamma d^{N/4},
    \label{eq:small-gamma-lower}
    \end{equation}
    where one may take
    \(K_\gamma=(\min_{\overline\Omega}c)K_1|\Omega|^{-1/2}>0\).
    Since \(0<\lambda_d\le r_*\) under the standing stability assumption,
    \eqref{eq:small-gamma-lower} yields
    \(0<\alpha_d\le(r_*/K_\gamma)d^{-N/4}\) for all sufficiently small \(d>0\).
    Consequently, \eqref{eq:small-b-perp} and
    \eqref{eq:small-b-flat} imply, as \(d\to0^+\),
    \begin{equation}
    \|B_d-\langle B_d,\psi_1^d\rangle\psi_1^d\|_{L^2(\Omega)}
    \le
    \frac{r_*K_{b,1}}{K_\gamma}\,d^{-N/4}e^{-\kappa_b d^{-1/2}}
    \longrightarrow0.
    \label{eq:small-bperp-zero}
    \end{equation}
    Orthogonal projection gives
    \(\|c-\langle c,\psi_1^d\rangle\psi_1^d\|_{L^2(\Omega)}
    \le\|c\|_{L^2(\Omega)}\). Moreover,
    Lemma~\ref{lem:mixed-maximum-concentration}\textup{(\ref{item:concentration-eigenvalue})}
    and \eqref{eq:small-b-flat} give
    \begin{equation}
    \begin{aligned}
    \langle c,\psi_1^d\rangle\langle B_d,\psi_1^d\rangle
    &=\lambda_d\bar b_d\longrightarrow r_*b_*>0
    \qquad(d\to0^+).
    \end{aligned}
    \label{eq:small-principal-limit}
    \end{equation}
    {\color{black}Since \(c>0\) and \(B_d=bu_d^*>0\), we have
    \(\langle c,B_d\rangle>0\). By \eqref{eq:small-principal-limit}, the left-hand side of
    \eqref{eq:krein-condition} is bounded away from zero, while its right-hand
    side tends to zero as \(d\to0^+\) by \eqref{eq:small-bperp-zero}. It follows from
    Lemma~\ref{thm:krein-criterion} that the conclusion holds under \eqref{ass:mixed-maximum}.}

    \medskip
    \noindent Assume \eqref{ass:plateau} and choose
    \(0\ne\psi\in C_c^\infty(D)\). Since \(a=r_*b\) on
    \(\operatorname{supp}\psi\), formula \eqref{eq:weighted-variational}
    yields
    \begin{equation}
    0\le r_*-\lambda_d
    \le
    K_Dd,
    \qquad
    K_D:=
    \frac{\displaystyle\int_D|\nabla\psi(x)|^2\,\mathrm{d}x}
    {\displaystyle\int_D b(x)\psi^2(x)\,\mathrm{d}x}>0.
    \label{eq:plateau-eigen-rate}
    \end{equation}
    For any relatively open neighborhood \(U\) of \(M\) in
    \(\overline\Omega\) with \(U\ne\overline\Omega\), define
    \[
    \eta_U:=\min_{x\in\overline\Omega\setminus U}
    \bigl(r_*-r(x)\bigr)>0.
    \]
    Multiplying the first equation in \eqref{eq:principal-eigenpair} by
    \(\psi_1^d\) and integrating over \(\Omega\), one obtains
    \begin{equation}
    \int_\Omega b(x)\bigl(r_*-r(x)\bigr)\bigl(\psi_1^d(x)\bigr)^2\,\mathrm{d}x
    =
    (r_*-\lambda_d)\int_\Omega b(x)\bigl(\psi_1^d(x)\bigr)^2\,\mathrm{d}x
    -d\int_\Omega|\nabla\psi_1^d(x)|^2\,\mathrm{d}x .
    \label{eq:plateau-energy}
    \end{equation}
    Combining \eqref{eq:plateau-energy} with \eqref{eq:plateau-eigen-rate}
    and using \(\|\psi_1^d\|_{L^2(\Omega)}=1\), we obtain, for fixed \(U\) as \(d\to0^+\),
    \begin{equation}
    0
    \le
    \left(\min_{x\in\overline\Omega}b(x)\right)\eta_U
    \int_{\Omega\setminus U}\bigl(\psi_1^d(x)\bigr)^2\,\mathrm{d}x
    \le
    (r_*-\lambda_d)\|b\|_{L^\infty(\Omega)}
    \longrightarrow0.
    \label{eq:plateau-concentration}
    \end{equation}
    The integral over \(\Omega\setminus U\) is zero when
    \(U=\overline\Omega\), so the same concentration conclusion holds in that case.
    Fix \(\delta>0\). By the continuity of \(b\) and the identity \(b=b_*\)
    on \(M\), one may choose a relatively open neighborhood \(U\) of \(M\) with
    \(|b(x)-b_*|\le\delta\) for \(x\in U\). Hence
    \[
    \|(b-b_*)\psi_1^d\|_{L^2(\Omega)}
    \le
    \delta
    +
    2\|b\|_{L^\infty(\Omega)}
    \left(\int_{\Omega\setminus U}\bigl(\psi_1^d(x)\bigr)^2\,\mathrm{d}x\right)^{1/2}.
    \]
    {\color{black}By \eqref{eq:plateau-concentration}, letting \(d\to0^+\) and then
    \(\delta\to0^+\), we obtain the first limit below. The other two limits
    follow from Cauchy--Schwarz and the triangle inequality, all as \(d\to0^+\):}
    \begin{equation}
    \|(b-b_*)\psi_1^d\|_{L^2(\Omega)}\longrightarrow0,
    \qquad
    |\bar b_d-b_*|\longrightarrow0,
    \qquad\text{and}\qquad
    \|(b-\bar b_d)\psi_1^d\|_{L^2(\Omega)}\longrightarrow0.
    \label{eq:plateau-b-flat}
    \end{equation}
    Since \(\|\psi_1^d\|_{L^2(\Omega)}=1\), equations
    \eqref{eq:plateau-energy} and \eqref{eq:plateau-eigen-rate} give
    \[
    d\int_\Omega|\nabla\psi_1^d(x)|^2\,\mathrm{d}x
    \le
    K_Dd\int_\Omega b(x)\bigl(\psi_1^d(x)\bigr)^2\,\mathrm{d}x
    \le K_D\|b\|_{L^\infty(\Omega)}d.
    \]
    Consequently,
    \(\|\psi_1^d\|_{H^1(\Omega)}\le \bigl(1+K_D\|b\|_{L^\infty(\Omega)}\bigr)^{1/2}\)
    for every \(d>0\). Choose \(p>2\)
    for which
    \(H^1(\Omega)\hookrightarrow L^p(\Omega)\), let \(K_S>0\) be an embedding
    constant, and set \(K_p:=K_S\bigl(1+K_D\|b\|_{L^\infty(\Omega)}\bigr)^{1/2}\). Thus
    \(\|\psi_1^d\|_{L^p(\Omega)}\le K_p\). Next, choose
    \(\theta\in(0,1)\) so that
    \(1/2=\theta+(1-\theta)/p\). The interpolation inequality gives
    \[
    1=\|\psi_1^d\|_{L^2(\Omega)}
    \le
    \|\psi_1^d\|_{L^1(\Omega)}^\theta
    \|\psi_1^d\|_{L^p(\Omega)}^{1-\theta}
    \le
    K_p^{1-\theta}\|\psi_1^d\|_{L^1(\Omega)}^\theta
    \]
    for every \(d>0\).
    {\color{black}Thus, the \(L^1\) norm is bounded below by a positive constant independent of
    \(d\), and we have}
    \begin{equation}
    \langle c,\psi_1^d\rangle
    \ge
    \left(\min_{x\in\overline\Omega}c(x)\right)
    \int_\Omega\psi_1^d(x)\,\mathrm{d}x
    \ge
    \left(\min_{x\in\overline\Omega}c(x)\right)K_p^{-(1-\theta)/\theta}>0.
    \label{eq:plateau-gamma}
    \end{equation}
    Since \(0<\lambda_d\le r_*\), estimate \eqref{eq:plateau-gamma} shows
    that \(\alpha_d=\lambda_d/\langle c,\psi_1^d\rangle\) is uniformly bounded.
    Equations
    \eqref{eq:small-b-perp} and \eqref{eq:plateau-b-flat}, together with
    the orthogonal projection bound, give, as \(d\to0^+\),
    \begin{equation}
    \begin{aligned}
    \|B_d-\langle B_d,\psi_1^d\rangle\psi_1^d\|_{L^2(\Omega)}&\longrightarrow0.
    \end{aligned}
    \label{eq:plateau-projection-limit}
    \end{equation}
    Finally, \eqref{eq:plateau-eigen-rate} and \eqref{eq:plateau-b-flat} give
    \begin{equation}
    \begin{aligned}
    \langle c,\psi_1^d\rangle\langle B_d,\psi_1^d\rangle
    &=\lambda_d\bar b_d\longrightarrow r_*b_*>0
    \qquad(d\to0^+).
    \end{aligned}
    \label{eq:plateau-principal-limit}
    \end{equation}
    {\color{black}Since \(c>0\) and \(B_d=bu_d^*>0\), we have
    \(\langle c,B_d\rangle>0\). 
    Note that $
    \|c-\langle c,\psi_1^d\rangle\psi_1^d\|_{L^2(\Omega)}\le\|c\|_{L^2(\Omega)}$. By \eqref{eq:plateau-projection-limit} and
    \eqref{eq:plateau-principal-limit}, condition
    \eqref{eq:krein-condition} holds for all sufficiently small
    \(d\). Then Lemma~\ref{thm:krein-criterion} gives the conclusion under \eqref{ass:plateau}.}
\end{proof}

\section{Semiclassical reduction near three maximum points}
\label{sec:semiclassical-reduction}

{\color{black}In this section, we study the spectral reduction used to construct
unstable steady states. We compare the subspace spanned by the first three
eigenfunctions with the span of three local Dirichlet ground states and
represent the compressed operator by a \(3\times3\) interaction matrix.}

\subsection{The local spectral setting}
\label{sec:three-well-setting}

{\color{black}We first introduce the potential and the local operators, and then prove
the localization and eigenvalue estimates needed below.}

Let \(\Omega=(0,1)\) and choose
$
0<x_1<x_2<x_3<1 .
$
Let \(V\in C^\infty(\overline\Omega;\mathbb R)\) satisfy
\[
V\ge0,\qquad V^{-1}(0)=\{x_1,x_2,x_3\},\qquad\text{and}\qquad
V''(x_j)>0,\quad j=1,2,3.
\]
Set
\[
    d_A(x,y):=\inf_{\gamma}\int_0^1
    \sqrt{V(\gamma(s))}\,|\dot\gamma(s)|\,\mathrm{d}s,
\]
where the infimum is taken over piecewise \(C^1\) paths
\(\gamma:[0,1]\to\overline\Omega\) with \(\gamma(0)=x\) and \(\gamma(1)=y\);
see Agmon \cite{Agmon1982} and Helffer \cite[Chapter~3]{Helffer1988}.
In one space dimension,
\[
d_A(x,y)=\left|\int_x^y\sqrt{V(s)}\,\mathrm{d}s\right|.
\]
For nonempty sets \(E,F\subset\overline\Omega\), we use
\[
d_A(x,F):=\inf_{y\in F}d_A(x,y),
\qquad
d_A(E,F):=\inf_{x\in E,\,y\in F}d_A(x,y).
\]

Define
\[
H_\varepsilon^N=-\varepsilon^2\frac{\mathrm{d}^2}{\mathrm{d}x^2}+V(x),
\qquad\text{and}\qquad
D(H_\varepsilon^N)=\{u\in H^2(\Omega):u'=0\hbox{ on }\partial\Omega\}.
\]
Let \(\lambda_1^N(\varepsilon)\le\lambda_2^N(\varepsilon)\le\cdots\) be its
eigenvalues, and let \(\phi_\varepsilon\) be the positive, \(L^2\)-normalized
first eigenfunction. As in Section~\ref{sec:main-results}, set
\[
q(x):=q_*-V(x),
\qquad\text{where}\qquad q_*\in\mathbb R .
\]
Thus \(x_1,x_2,x_3\) are precisely the maximum points of \(q\), with
\(q(x_j)=q_*\). Define
\[
L_\varepsilon=q_*I-H_\varepsilon^N,\qquad
\nu_\varepsilon=q_*-\lambda_1^N(\varepsilon),
\qquad\text{and}\qquad
A_\varepsilon=-(H_\varepsilon^N-\lambda_1^N(\varepsilon)I).
\]
Then \(L_\varepsilon\phi_\varepsilon=\nu_\varepsilon\phi_\varepsilon\) and
\(A_\varepsilon\phi_\varepsilon=0\).

Write \(S_{ij}:=d_A(x_i,x_j)\) and \(S:=\min_{i\ne j}S_{ij}\).
\begingroup
\color{black}
Fix \(S/2<s<r<S\), with \(S-r\) and \(r-s\) sufficiently small,
independently of \(\varepsilon\), and set
\endgroup
\begin{equation}
\Omega_j:=\{x\in\Omega:d_A(x,x_j)<r\},\qquad j=1,2,3.
\label{eq:local-dirichlet-intervals}
\end{equation}
Under the boundary condition in \eqref{ass:tunneling} below,
\(\Omega_j\Subset\Omega\), and each interval contains exactly one zero of
\(V\). Adjacent intervals overlap when \(S_{12}=S_{23}=S\).
Let
\[
P_j(\varepsilon)=-\varepsilon^2\frac{\mathrm{d}^2}{\mathrm{d}x^2}+V(x),
\qquad\text{and}\qquad
D(P_j(\varepsilon))=H^2(\Omega_j)\cap H_0^1(\Omega_j),
\]
and denote its \(k\)-th eigenvalue by \(\mu_k^{(j)}(\varepsilon)\).
Denote by \(\nu_k^{(j)}\) the \(k\)-th eigenvalue of
\[
\mathcal H_j
=
-\frac{\mathrm{d}^2}{\mathrm{d}y^2}
+
\frac12V''(x_j)y^2
\quad\hbox{on }\mathbb R .
\]
\begingroup
\color{black}
Let \(\xi_{j,\varepsilon}\) be the \(L^2\)-normalized positive first
eigenfunction of \(P_j(\varepsilon)\). To define the interaction matrix,
choose fixed cutoff functions \(\vartheta_j\in C_c^\infty(\Omega_j)\) satisfying
\begin{equation}
0\le\vartheta_j\le1,\qquad
\vartheta_j(x)=1\quad\text{if }d_A(x,x_j)\le s,
\qquad j=1,2,3.
\label{eq:interaction-cutoffs}
\end{equation}
Define the real symmetric matrix \(W_\varepsilon=(w_{jk})_{1\le j,k\le3}\) by
\begin{equation}
\begin{aligned}
w_{jk}:={}&\frac{\varepsilon^2}{2}
\int_{\Omega_j\cap\Omega_k}
\bigl(\vartheta_j\vartheta_k'-\vartheta_k\vartheta_j'\bigr)
\bigl(\xi_{k,\varepsilon}\xi_{j,\varepsilon}'
-\xi_{j,\varepsilon}\xi_{k,\varepsilon}'\bigr)\,\mathrm dx,
\qquad 1\le j,k\le3.
\end{aligned}
\label{eq:interaction-matrix-definition}
\end{equation}
{\color{black}The coefficients in \eqref{eq:interaction-matrix-definition}
are the symmetrized coefficients in
Helffer \cite[(4.3.11)--(4.3.12)]{Helffer1988}. In particular,}
\(w_{jj}=0\) and \(w_{jk}=w_{kj}\).
\par
\endgroup

\begingroup
\color{black}
\begin{assumptionblock}{H_{\rm t}}
\label{ass:tunneling}
The following three conditions hold.
\begin{enumerate}
    \item\label{item:tunneling-actions} \(S_{12}=S_{23}=S<S_{13}\).

    \item\label{item:tunneling-local-eigenvalues} The local first eigenvalues satisfy
    \begin{equation}
    \mu_1^{(j)}(\varepsilon)-\mu_1^{(1)}(\varepsilon)
    =o\!\left(\varepsilon^{1/2}e^{-S/\varepsilon}\right),
    \qquad j=1,2,3.
    \label{eq:local-eigenvalue-assumption}
    \end{equation}

    \item\label{item:tunneling-boundary}
    \(S_\partial:=d_A(\partial\Omega,V^{-1}(0))>S\).
\end{enumerate}
Under \textup{(\ref{item:tunneling-actions})--(\ref{item:tunneling-boundary})},
Helffer \cite[Theorem~4.4.8]{Helffer1988} gives positive constants
\(b_{12}^{\,0},b_{23}^{\,0}\) such that
\begin{equation}
w_{jk}
=-\varepsilon^{1/2}\bigl(b_{jk}^{\,0}+O(\varepsilon)\bigr)
e^{-S_{jk}/\varepsilon},\qquad
(j,k)=(1,2),(2,3).
\label{eq:interaction-entry-asymptotics}
\end{equation}
We impose one further condition:
\begin{enumerate}
    \setcounter{enumi}{3}
    \item\label{item:tunneling-coefficients}
    \(b_{12}^{\,0}=b_{23}^{\,0}=:\mathfrak b_0>0\).
\end{enumerate}
\end{assumptionblock}

Set
\begin{equation}
\tau_\varepsilon:=\varepsilon^{1/2}\mathfrak b_0e^{-S/\varepsilon}.
\label{eq:tunneling-scale}
\end{equation}
{\color{black}Since \(\mathfrak b_0\) is a fixed positive constant,
it follows from \eqref{eq:tunneling-scale} that
\eqref{eq:local-eigenvalue-assumption} is equivalent to}
\(\mu_1^{(j)}(\varepsilon)-\mu_1^{(1)}(\varepsilon)=o(\tau_\varepsilon)\).
\par
\endgroup

Let \((\mu_\varepsilon,\psi_\varepsilon)\) be an eigenpair with a real,
\(L^2\)-normalized eigenfunction of \(H_\varepsilon^N\) satisfying
\begin{equation}
H_\varepsilon^N\psi_\varepsilon=\mu_\varepsilon\psi_\varepsilon,
\qquad\text{and}\qquad
\|\psi_\varepsilon\|_{L^2(\Omega)}=1 .
\label{eq:agmon-eigenpair}
\end{equation}

\begin{lemma}
    \label{lem:agmon-neumann}
    Assume that \(\mu_\varepsilon\le C_0\varepsilon\) for some \(C_0>0\). If
    \(K\subset\overline\Omega\) is closed, then, for every
    \(0<\kappa_K<d_A(K,V^{-1}(0))\), there exist
    \(C_K,\varepsilon_K>0\) such that
    \begin{equation}
    \|\psi_\varepsilon\|_{L^2(K)}
    +\varepsilon\|\psi_\varepsilon'\|_{L^2(K)}
    \le C_K e^{-\kappa_K/\varepsilon},
    \qquad 0<\varepsilon<\varepsilon_K .
    \label{eq:agmon-neumann-decay}
    \end{equation}
\end{lemma}

\begin{proof}
    We may assume that \(K\ne\varnothing\).
    Let \(\Phi\in W^{1,\infty}(\Omega;\mathbb R)\). {\color{black}Multiplying \eqref{eq:agmon-eigenpair} by
    \(e^{2\Phi/\varepsilon}\psi_\varepsilon\) and integrating by parts using
    \(\psi_\varepsilon'=0\) on \(\partial\Omega\), we obtain}
    \[
    \varepsilon^2
    \int_\Omega
    \psi_\varepsilon'
    (e^{2\Phi/\varepsilon}\psi_\varepsilon)'\,\mathrm{d}x
    +
    \int_\Omega
    (V-\mu_\varepsilon)e^{2\Phi/\varepsilon}|\psi_\varepsilon|^2\,\mathrm{d}x
    =0,
    \]
    {\color{black}and, with
    \(v_\varepsilon=e^{\Phi/\varepsilon}\psi_\varepsilon\), we have}
    \begin{equation}
    \varepsilon^2
    \int_\Omega|v_\varepsilon'|^2\,\mathrm{d}x
    +
    \int_\Omega
    (V-\mu_\varepsilon-|\Phi'|^2)
    |v_\varepsilon|^2\,\mathrm{d}x
    =0.
    \label{eq:agmon-weighted-identity}
    \end{equation}
    
    Set \(d_K:=d_A(K,V^{-1}(0))\) and fix
    \(\kappa_K\in(0,d_K)\). Choose \(\delta\in(0,1)\) and
    \(\theta\in(0,d_K)\) such that
    \[
    \kappa_K<(1-\delta)(d_K-\theta).
    \]
    Set \(d(x):=d_A(x,V^{-1}(0))\),
    \(\Phi_\delta(x):=(1-\delta)d(x)\), and
    \(\alpha_\delta:=1-(1-\delta)^2>0\).
    The function \(d\) is Lipschitz and satisfies \(|d'|^2\le V\) a.e. in
    \(\Omega\); see Fournais--Helffer \cite[Sections~7.2.2--7.2.4]{FH2010}.
    Thus \(\Phi_\delta\in W^{1,\infty}(\Omega)\), and
    \[
    V-\mu_\varepsilon-|\Phi_\delta'|^2
    \ge
    \alpha_\delta V-\mu_\varepsilon .
    \]
    
    Choose a relatively open
    neighborhood \(U\) of \(V^{-1}(0)\) such that
    \[
    \overline U\subset\{x\in\overline\Omega:d(x)<\theta\}.
    \]
    Then \(K\cap\overline U=\varnothing\). Since \(U\) contains every zero of
    \(V\), compactness gives \(v_0:=\min_{x\in\overline\Omega\setminus U}V(x)>0\).
    Set \(\varepsilon_0:=\min\{1,\alpha_\delta v_0/(2C_0)\}\).
    For \(0<\varepsilon<\varepsilon_0\), one has
    \(\alpha_\delta V-\mu_\varepsilon\ge\alpha_\delta v_0/2\)
    on \(\Omega\setminus U\), whereas
    \(\alpha_\delta V-\mu_\varepsilon\ge-C_0\varepsilon\) on \(U\).
    {\color{black}It follows from \eqref{eq:agmon-weighted-identity} with \(\Phi=\Phi_\delta\) that}
    \begin{equation}
    \frac{\alpha_\delta v_0}{2}
    \int_{\Omega\setminus U}
    e^{2\Phi_\delta/\varepsilon}|\psi_\varepsilon|^2\,\mathrm{d}x
    \le
    C_0\varepsilon\int_U
    e^{2\Phi_\delta/\varepsilon}|\psi_\varepsilon|^2\,\mathrm{d}x .
    \label{eq:agmon-weighted-mass}
    \end{equation}
    Since \(\Phi_\delta\le(1-\delta)\theta\) on \(U\) and
    \(\|\psi_\varepsilon\|_{L^2(\Omega)}=1\), estimate \eqref{eq:agmon-weighted-mass} gives
    \[
    \int_{\Omega\setminus U}
    e^{2\Phi_\delta/\varepsilon}|\psi_\varepsilon|^2\,\mathrm{d}x
    \le
    \frac{2C_0\varepsilon}{\alpha_\delta v_0}
    e^{2(1-\delta)\theta/\varepsilon} .
    \]
    Since \(d(x)\ge d_K\) on \(K\),
    \[
    \Phi_\delta\ge(1-\delta)d_K
    \qquad\hbox{on }K.
    \]
    {\color{black}Combining these two estimates, we obtain}
    \[
    \|\psi_\varepsilon\|_{L^2(K)}^2
    \le
    \frac{2C_0\varepsilon}{\alpha_\delta v_0}
    \exp\!\left(-\frac{2(1-\delta)(d_K-\theta)}{\varepsilon}\right).
    \]
    Since \(\kappa_K<(1-\delta)(d_K-\theta)\), after decreasing
    \(\varepsilon_0\), if necessary, one has
    \begin{equation}
    \|\psi_\varepsilon\|_{L^2(K)}
    \le
    C_{K,0}e^{-\kappa_K/\varepsilon},
    \qquad 0<\varepsilon<\varepsilon_0,
    \label{eq:agmon-local-mass}
    \end{equation}
    for some \(C_{K,0}>0\).
    
    {\color{black}Next, we estimate the gradient. Choose \(d_1\in(\kappa_K,d_K)\),}
    a relatively open set \(U_K\subset\overline\Omega\), and
    \(\chi\in C^\infty(\overline\Omega;\mathbb R)\) such that \(0\le\chi\le1\)
    and
    \[
    K\subset U_K,\qquad \chi=1\ \hbox{on }U_K,
    \qquad
    \operatorname{supp}\chi\subset
    \{x\in\overline\Omega:d(x)>d_1\}.
    \]
    {\color{black}Multiplying \eqref{eq:agmon-eigenpair} by \(\chi^2\psi_\varepsilon\) and
    integrating by parts using the Neumann boundary condition, we obtain}
    \[
    \varepsilon^2
    \int_\Omega\chi^2|\psi_\varepsilon'|^2\,\mathrm{d}x
    +
    \int_\Omega\chi^2(V-\mu_\varepsilon)|\psi_\varepsilon|^2\,\mathrm{d}x
    =
    -2\varepsilon^2
    \int_\Omega
    \chi \psi_\varepsilon\chi'\psi_\varepsilon'\,\mathrm{d}x .
    \]
    Since \(\operatorname{supp}\chi\cap V^{-1}(0)=\varnothing\), one has
    \[
    v_\chi:=\min_{x\in\operatorname{supp}\chi}V(x)>0 .
    \]
    Thus \(V-\mu_\varepsilon\ge0\) on \(\operatorname{supp}\chi\) whenever
    \(0<\varepsilon<\varepsilon_\chi:=\min\{1,v_\chi/(2C_0)\}\). For such
    \(\varepsilon\), discarding the nonnegative integral containing
    \(V-\mu_\varepsilon\) from the integration-by-parts identity gives
    \[
    \varepsilon^2
    \int_\Omega\chi^2|\psi_\varepsilon'|^2\,\mathrm{d}x
    \le
    2\varepsilon^2
    \left|
    \int_\Omega
    \chi\psi_\varepsilon\chi'\psi_\varepsilon'\,\mathrm{d}x
    \right|.
    \]
    {\color{black}By \(2ab\le a^2/2+2b^2\), we obtain}
    \[
    \varepsilon^2
    \int_\Omega\chi^2|\psi_\varepsilon'|^2\,\mathrm{d}x
    \le
    \frac{\varepsilon^2}{2}
    \int_\Omega\chi^2|\psi_\varepsilon'|^2\,\mathrm{d}x
    +
    2\varepsilon^2\|\chi'\|_{L^\infty(\Omega)}^2
    \int_{\operatorname{supp}\chi'}
    |\psi_\varepsilon|^2\,\mathrm{d}x .
    \]
    Therefore
    \[
    \varepsilon^2
    \int_K|\psi_\varepsilon'|^2\,\mathrm{d}x
    \le
    4\varepsilon^2\|\chi'\|_{L^\infty(\Omega)}^2
    \int_{\operatorname{supp}\chi'}
    |\psi_\varepsilon|^2\,\mathrm{d}x .
    \]
    {\color{black}We apply the \(L^2\)-estimate \eqref{eq:agmon-local-mass} to the closed set
    \(K_\chi:=\operatorname{supp}\chi'\). We can take the exponent \(\kappa_K\),
    since \(d_A(K_\chi,V^{-1}(0))\ge d_1>\kappa_K\).}
    Thus there exist \(C_{\chi,0},\varepsilon_{\chi,0}>0\) such that
    \[
    \|\psi_\varepsilon\|_{L^2(K_\chi)}
    \le C_{\chi,0}e^{-\kappa_K/\varepsilon},
    \qquad 0<\varepsilon<\varepsilon_{\chi,0}.
    \]
    Since \(\varepsilon<\varepsilon_\chi\le1\), it follows that
    \begin{equation}
    \varepsilon\|\psi_\varepsilon'\|_{L^2(K)}
    \le
    2\|\chi'\|_{L^\infty(\Omega)}C_{\chi,0}
    e^{-\kappa_K/\varepsilon}.
    \label{eq:agmon-local-gradient}
    \end{equation}
    Combining \eqref{eq:agmon-local-mass} and \eqref{eq:agmon-local-gradient}
    proves the assertion with
    \(C_K:=C_{K,0}+2\|\chi'\|_{L^\infty(\Omega)}C_{\chi,0}\)
    and
    \(\varepsilon_K:=\min\{\varepsilon_0,\varepsilon_\chi,
    \varepsilon_{\chi,0}\}\).
\end{proof}

\begin{lemma}
    \label{lem:bound-first-four}
    There exist constants \(C_{\rm E},\varepsilon_{\rm E}>0\) such that
    \[
    \lambda_j^N(\varepsilon)\le C_{\rm E}\varepsilon,
    \qquad j=1,2,3,4,\quad 0<\varepsilon<\varepsilon_{\rm E}.
    \]
\end{lemma}

\begin{proof}
    Choose \(r_j>0\) so that the intervals
    \(I_j:=(x_j-2r_j,x_j+2r_j)\) are contained in \(\Omega\) and are pairwise
    disjoint. Since \(V(x_j)=V'(x_j)=0\), there is \(M_j>0\) such that
    \[
    0\le V(x)\le M_j|x-x_j|^2,
    \qquad x\in I_j,\quad j=1,2,3.
    \]
    Choose \(\chi_j\in C_c^\infty(I_j)\) satisfying
    \(0\le\chi_j\le1\) and \(\chi_j=1\) on
    \([x_j-r_j,x_j+r_j]\). We also choose \(\chi_1\) to be even with respect
    to \(x_1\). Define
    \[
    \psi_{j,\varepsilon}(x)
    =
    \chi_j(x)e^{-|x-x_j|^2/(2\varepsilon)},
    ~j=1,2,3,~  
    \text{and}~
    \widetilde\psi_{1,\varepsilon}(x)
    =
    \chi_1(x)(x-x_1)e^{-|x-x_1|^2/(2\varepsilon)}.
    \]
    Let
    \[
    m_0:=\int_{-1}^{1}e^{-y^2}\,\mathrm{d}y,
    \qquad
    m_2:=\int_{\mathbb R}y^2e^{-y^2}\,\mathrm{d}y,
    \qquad
    m_4:=\int_{\mathbb R}y^4e^{-y^2}\,\mathrm{d}y.
    \]
    For \(0<\varepsilon<\min_j r_j^2\), the change of variables
    \(x=x_j+\varepsilon^{1/2}y\) gives
    \[
    \|\psi_{j,\varepsilon}\|_{L^2(\Omega)}^2
    \ge m_0\varepsilon^{1/2},
    \qquad
    \int_\Omega |x-x_j|^2
    e^{-|x-x_j|^2/\varepsilon}\,\mathrm{d}x
    \le m_2\varepsilon^{3/2},
    \]
    and
    \[
    \int_\Omega |x-x_j|^4
    e^{-|x-x_j|^2/\varepsilon}\,\mathrm{d}x
    \le m_4\varepsilon^{5/2}.
    \]
    On \(\operatorname{supp}\chi_j'\), one has
    \(|x-x_j|\ge r_j\). Hence there are constants
    \(K_j>0\) and \(\varepsilon_j\in(0,r_j^2)\) such that
    \[
    2\varepsilon^2\|\chi_j'\|_{L^\infty(\Omega)}^2
    |\Omega|e^{-r_j^2/\varepsilon}
    \le K_j\varepsilon^{3/2},
    \qquad 0<\varepsilon<\varepsilon_j.
    \]
    Since
    \[
    \psi_{j,\varepsilon}'(x)
    =
    \left(\chi_j'(x)-\frac{x-x_j}{\varepsilon}\chi_j(x)\right)
    e^{-|x-x_j|^2/(2\varepsilon)},
    \]
    the cutoff estimate and the definition of \(m_2\) give
    \[
    \varepsilon^2\int_\Omega
    |\psi_{j,\varepsilon}'(x)|^2\,\mathrm{d}x
    \le
    (K_j+2m_2)\varepsilon^{3/2},
    \qquad 0<\varepsilon<\varepsilon_j,
    \]
    and
    \[
    \int_\Omega V(x)\psi_{j,\varepsilon}^2(x)\,\mathrm{d}x
    \le
    M_jm_2\varepsilon^{3/2}.
    \]
    Thus, with
    \(A_j:=(K_j+(2+M_j)m_2)/m_0\), one has
    \[
    \varepsilon^2\int_\Omega
    |\psi_{j,\varepsilon}'(x)|^2\,\mathrm{d}x
    +
    \int_\Omega V(x)\psi_{j,\varepsilon}^2(x)\,\mathrm{d}x
    \le
    A_j\varepsilon\|\psi_{j,\varepsilon}\|_{L^2(\Omega)}^2,
    \qquad 0<\varepsilon<\varepsilon_j.
    \]

    {\color{black}Applying the same change of variables to the odd test function, we obtain constants}
    \(\widetilde A_1>0\) and
    \(\widetilde\varepsilon_1\in(0,r_1^2)\) such that
    \[
    \varepsilon^2\int_\Omega
    |\widetilde\psi_{1,\varepsilon}'(x)|^2\,\mathrm{d}x
    +
    \int_\Omega V(x)\widetilde\psi_{1,\varepsilon}^2(x)\,\mathrm{d}x
    \le
    \widetilde A_1\varepsilon
    \|\widetilde\psi_{1,\varepsilon}\|_{L^2(\Omega)}^2,
    \qquad 0<\varepsilon<\widetilde\varepsilon_1.
    \]
    For \(f\in H^1(\Omega)\), define the nonnegative quadratic form
    \[
    Q_\varepsilon(f)
    :=
    \varepsilon^2\int_\Omega|f'(x)|^2\,\mathrm{d}x
    +
    \int_\Omega V(x)|f(x)|^2\,\mathrm{d}x.
    \]
    Set
    \[
    \mathcal T_\varepsilon
    :=
    \operatorname{span}\{\psi_{1,\varepsilon},
    \widetilde\psi_{1,\varepsilon},
    \psi_{2,\varepsilon},\psi_{3,\varepsilon}\}.
    \]
    Since \(\chi_1\) is even about \(x_1\), the first two functions are
    \(L^2\)-orthogonal. The other pairs have disjoint supports, so
    \(\dim\mathcal T_\varepsilon=4\).
    Using \(Q_\varepsilon(f+g)\le2Q_\varepsilon(f)+2Q_\varepsilon(g)\) for
    the first two functions and the disjoint supports for the remaining terms,
    we obtain the bound
    \[
    Q_\varepsilon(f)
    \le C_{\rm E}\varepsilon\|f\|_{L^2(\Omega)}^2,
    \qquad
    f\in\mathcal T_\varepsilon,\quad
    0<\varepsilon<\varepsilon_{\rm E},
    \]
    where
    $
    C_{\rm E}:=
    \max\{2A_1,2\widetilde A_1,A_2,A_3\},
   $ and $
    \varepsilon_{\rm E}:=
    \min\{\varepsilon_1,\varepsilon_2,\varepsilon_3,
    \widetilde\varepsilon_1\}.
   $
    {\color{black}By the min--max principle \cite[Appendix~A]{FH2010}, we have}
    \[
    \lambda_4^N(\varepsilon)
    \le
    \sup_{0\ne f\in\mathcal T_\varepsilon}
    \frac{Q_\varepsilon(f)}{\|f\|_{L^2(\Omega)}^2}
    \le
    C_{\rm E}\varepsilon.
    \]
    Since the eigenvalues are nondecreasing, the same bound holds for
    \(j=1,2,3\).
\end{proof}

By \eqref{ass:tunneling}\textup{(\ref{item:tunneling-boundary})} and continuity
of \(d_A\), choose \(0<\alpha_0<x_1<x_2<x_3<\beta_0<1\) so that
$
\Omega_0=(\alpha_0,\beta_0)
$
contains \(\bigcup_{j=1}^3\overline\Omega_j\) and satisfies
\[
d_A(\partial\Omega_0,V^{-1}(0))
=
\min\{d_A(\alpha_0,V^{-1}(0)),d_A(\beta_0,V^{-1}(0))\}
>S .
\]
Let
\[
H_\varepsilon^{D}
=
-\varepsilon^2\frac{\mathrm{d}^2}{\mathrm{d}x^2}+V(x),
\qquad\text{and}\qquad
D(H_\varepsilon^{D})=H^2(\Omega_0)\cap H_0^1(\Omega_0).
\]
Write \(\lambda_k^D(\varepsilon)\) for the ordered eigenvalues of
\(H_\varepsilon^{D}\).
Choose \(L^2(\Omega)\)-orthonormal eigenfunctions \(u_k^N\) for
\(\lambda_k^N(\varepsilon)\), and \(L^2(\Omega_0)\)-orthonormal eigenfunctions
\(e_k^D\) for \(\lambda_k^D(\varepsilon)\). For \(k\ge1\), put
\[
E_{\varepsilon,k}^{N}:=\operatorname{span}\{u_1^N,\ldots,u_k^N\},
\qquad\text{and}\qquad
E_{\varepsilon,k}^{D}:=\operatorname{span}\{e_1^D,\ldots,e_k^D\}.
\]
Let \(\Pi_\varepsilon^N\) and \(\Pi_{\varepsilon,0}^{D}\) be the
orthogonal projections onto \(E_{\varepsilon,3}^{N}\subset L^2(\Omega)\) and
\(E_{\varepsilon,3}^{D}\subset L^2(\Omega_0)\), respectively. Once an
orthonormal basis of \(E_{\varepsilon,3}^{D}\) is specified,
\(\mathcal M_{\varepsilon,0}^{D}\) denotes the matrix of
\(\Pi_{\varepsilon,0}^{D}H_\varepsilon^{D}\Pi_{\varepsilon,0}^{D}\) on this
space.
For \(\chi\in C_c^\infty(\Omega_0)\), define
\(J_\chi:L^2(\Omega_0)\to L^2(\Omega)\) by extending \(\chi f\) by zero:
\[
(J_\chi f)(x)=
\begin{cases}
\chi(x)f(x),&x\in\Omega_0,\\
0,&x\in\Omega\setminus\Omega_0,
\end{cases}
\qquad f\in L^2(\Omega_0).
\]

Fix \(S<\rho<d_A(\partial\Omega_0,V^{-1}(0))\) and choose
\(\chi\in C_c^\infty(\Omega_0)\), \(0\le\chi\le1\), such that
\begin{equation}
\chi=1
\quad\hbox{if}\quad
d_A(x,V^{-1}(0))\le\rho.
\label{eq:nd-cutoff}
\end{equation}
Set \(\widetilde E_{\varepsilon,0}^D=J_\chi E_{\varepsilon,3}^{D}\), and let
\(\widetilde\Pi_{\varepsilon,0}^D\) be the orthogonal projection onto
\(\widetilde E_{\varepsilon,0}^D\) in \(L^2(\Omega)\).

\begin{lemma}
    \label{lem:ND-comparison}
    Let \(\Omega_0=(\alpha_0,\beta_0)\Subset\Omega\) satisfy
    \(V^{-1}(0)\Subset\Omega_0\) and
    \(d_A(\partial\Omega_0,V^{-1}(0))>S\). The following assertions hold.
    \begin{enumerate}
        \item\label{item:nd-eigenvalues} There exist \(C_{\rm V}>0\), \(\kappa_{\rm V}>S\), and
        \(\varepsilon_{\rm V}>0\) such that, whenever
        \(0<\varepsilon<\varepsilon_{\rm V}\), one has
        \[
        |\lambda_k^N(\varepsilon)-\lambda_k^D(\varepsilon)|
        \le C_{\rm V}e^{-\kappa_{\rm V}/\varepsilon},
        \qquad 1\le k\le4.
        \]
        
        \item\label{item:nd-spectral-spaces} Assume in addition that, for some \(\gamma>0\) and all sufficiently
        small \(\varepsilon>0\),
        \(\lambda_4^D(\varepsilon)-\lambda_3^D(\varepsilon)
        \ge\gamma\varepsilon\). Then there exist \(C_{\rm P}>0\), \(\kappa_{\rm P}>S\), and
        \(\varepsilon_{\rm P}>0\) such that, whenever
        \(0<\varepsilon<\varepsilon_{\rm P}\), one has
        \[
        \|\widetilde\Pi_{\varepsilon,0}^D-\Pi_\varepsilon^N\|
        \le C_{\rm P}e^{-\kappa_{\rm P}/\varepsilon}.
        \]
    \end{enumerate}
\end{lemma}

\begin{proof}
    \noindent\textup{(\ref{item:nd-eigenvalues})}
    Let \(\chi\) be the cutoff fixed in \eqref{eq:nd-cutoff}. Choose
    \(S<\rho_0<\rho_1<\rho\) and \(U\Subset\Omega_0\) such that
    \[
\{d_A(\cdot,V^{-1}(0))\le\rho_0\}\subset U
\subset
\{d_A(\cdot,V^{-1}(0))<\rho_1\},
\qquad\text{and}\qquad
0\le\chi\le1\quad\text{and}\quad \chi=1\ {\rm on}\ U.
    \]
    Extend \(\chi\) by zero from \(\Omega_0\) to \(\Omega\), and define
    \(K_\chi\), with all supports taken in \(\overline\Omega\), by
    \[
    K_\chi
    =
    \operatorname{supp}_{\overline\Omega}(1-\chi)
    \cup
    \operatorname{supp}_{\overline\Omega}\chi'
    \cup
    \operatorname{supp}_{\overline\Omega}\chi'' .
    \]
  
    Then \(d_A(K_\chi,V^{-1}(0))\ge\rho_0>S\). Lemma~\ref{lem:bound-first-four}
    and the same test functions supported in \(\Omega_0\) give constants
    \(C_{N,0},C_{D,0}>0\) and \(\varepsilon_0>0\) such that
    \[
\lambda_k^N(\varepsilon)\le C_{N,0}\varepsilon,
\qquad\text{and}\qquad
\lambda_k^D(\varepsilon)\le C_{D,0}\varepsilon,
\qquad 1\le k\le4,\quad 0<\varepsilon<\varepsilon_0 .
    \]
    {\color{black}Write}
    \[
Q_N(z)=\int_\Omega
\bigl(\varepsilon^2|z'(x)|^2+V(x)|z(x)|^2\bigr)\,\mathrm{d}x,
\qquad\text{and}\qquad
Q_D(y)=\int_{\Omega_0}
\bigl(\varepsilon^2|y'(x)|^2+V(x)|y(x)|^2\bigr)\,\mathrm{d}x .
    \]
    
    If \(u\in E_{\varepsilon,k}^{N}\), \(1\le k\le4\), write
    \(u=\sum_{m=1}^k a_m u_m^N\). Since
    \(\lambda_m^N(\varepsilon)\le C_{N,0}\varepsilon\) for \(1\le m\le4\),
    {\color{black}we can apply \eqref{eq:agmon-neumann-decay} to each \(u_m^N\), with constants
    independent of the choice among these four functions. By the Cauchy--Schwarz
    inequality, there exist} \(C_{N,1}>0\), \(\kappa_N>S\), and
    \(\varepsilon_N\in(0,\varepsilon_0)\) such that
    \begin{equation}
    \|u\|_{L^2(K_\chi)}
    +
    \varepsilon\|u'\|_{L^2(K_\chi)}
    \le
    C_{N,1} e^{-\kappa_N/\varepsilon}\|u\|_{L^2(\Omega)},
    \qquad 0<\varepsilon<\varepsilon_N .
    \label{eq:neumann-cutoff-decay}
    \end{equation}
    Since \(\operatorname{supp}_{\overline\Omega}(1-\chi^2)\subset K_\chi\),
    \eqref{eq:neumann-cutoff-decay} gives
    \[
    0\le
    \|u\|_{L^2(\Omega)}^2-\|\chi u\|_{L^2(\Omega_0)}^2
    =
    \int_\Omega(1-\chi^2(x))|u(x)|^2\,\mathrm{d}x
    \le
    C_{N,1}^2e^{-2\kappa_N/\varepsilon}\|u\|_{L^2(\Omega)}^2 .
    \]
    After decreasing \(\varepsilon_N\), if necessary, one has
    \(C_{N,1}^2e^{-2\kappa_N/\varepsilon}\le1/2\), and hence
    \[
    \|\chi u\|_{L^2(\Omega_0)}^2
    \ge
    \bigl(1-C_{N,1}^2e^{-2\kappa_N/\varepsilon}\bigr)
    \|u\|_{L^2(\Omega)}^2
    \ge\frac12\|u\|_{L^2(\Omega)}^2 .
    \]
    Since \(\chi\in C_c^\infty(\Omega_0)\), integration by parts gives
    \[
    \int_{\Omega_0}|(\chi u)'(x)|^2\,\mathrm{d}x
    =
    \int_{\Omega_0}\chi^2(x)|u'(x)|^2\,\mathrm{d}x
    -
    \int_{\Omega_0}\chi(x)\chi''(x)|u(x)|^2\,\mathrm{d}x .
    \]
    Put \(M_\chi:=\|\chi\chi''\|_{L^\infty(\Omega_0)}\). Since
    \(\operatorname{supp}\chi''\subset K_\chi\), \eqref{eq:neumann-cutoff-decay} yields
    \[
    Q_D(\chi u)
    \le
    Q_N(u)+M_\chi C_{N,1}^2\varepsilon^2e^{-2\kappa_N/\varepsilon}
    \|u\|_{L^2(\Omega)}^2 .
    \]
    The lower bound for \(\|\chi u\|_{L^2(\Omega_0)}\) shows that
    \(u\mapsto\chi u\) is injective on \(E_{\varepsilon,k}^{N}\). Hence
    \(\dim\chi E_{\varepsilon,k}^{N}=k\), and  for any $1\leq k\leq 4$
    \[
    \begin{aligned}
    \lambda_k^D(\varepsilon)
    &\le
    \sup_{0\ne u\in E_{\varepsilon,k}^{N}}
    \frac{Q_D(\chi u)}{\|\chi u\|_{L^2(\Omega_0)}^2}\le
    \frac{\lambda_k^N(\varepsilon)
    +M_\chi C_{N,1}^2\varepsilon^2e^{-2\kappa_N/\varepsilon}}
    {1-C_{N,1}^2e^{-2\kappa_N/\varepsilon}}
    \le
    \lambda_k^N(\varepsilon)+C_{N,3}e^{-2\kappa_N/\varepsilon}.
    \end{aligned}
    \]
    Here one may take
    \(C_{N,3}:=2(M_\chi+C_{N,0})C_{N,1}^2\), after decreasing
    \(\varepsilon_N\) once more so that \(0<\varepsilon<1\).
    
    {\color{black}To prove the reverse inequality, we interchange the global Neumann problem
    and the interior Dirichlet problem.} By the test functions
    from Lemma~\ref{lem:bound-first-four}, after decreasing \(\varepsilon_0\),
    one has
    \[
    \lambda_m^D(\varepsilon)\le C_{D,0}\varepsilon,
    \qquad 1\le m\le4,\quad 0<\varepsilon<\varepsilon_0.
    \]
    If \(w\in E_{\varepsilon,k}^{D}\), expand \(w\) in the first \(k\)
    Dirichlet eigenfunctions. The weighted identity
    \eqref{eq:agmon-weighted-identity} has zero boundary term here because
    \(w=0\) on \(\partial\Omega_0\). After fixing
    \(\kappa_D\in(S,\rho_0)\), there exist \(C_{D,1}>0\) and
    \(\varepsilon_D\in(0,\varepsilon_0)\), independent of
    \(k\in\{1,2,3,4\}\), such that
    \begin{equation}
    \|w\|_{L^2(K_\chi\cap\Omega_0)}
    +
    \varepsilon\|w'\|_{L^2(K_\chi\cap\Omega_0)}
    \le
    C_{D,1}e^{-\kappa_D/\varepsilon}
    \|w\|_{L^2(\Omega_0)},
    \qquad 0<\varepsilon<\varepsilon_D .
    \label{eq:dirichlet-cutoff-decay}
    \end{equation}
    Since \(J_\chi w=\chi w\) on \(\Omega_0\) and is zero outside
    \(\Omega_0\), \eqref{eq:dirichlet-cutoff-decay} gives
    \[
    \begin{aligned}
    \|J_\chi w\|_{L^2(\Omega)}^2
    &=
    \|w\|_{L^2(\Omega_0)}^2
    -
    \int_{\Omega_0}(1-\chi^2(x))|w(x)|^2\,\mathrm{d}x\ge
    \bigl(1-C_{D,1}^2e^{-2\kappa_D/\varepsilon}\bigr)
    \|w\|_{L^2(\Omega_0)}^2 .
    \end{aligned}
    \]
    After decreasing \(\varepsilon_D\), if necessary, the factor in
    parentheses is at least \(1/2\). Moreover, integration by parts gives
    \[
    \begin{aligned}
    Q_N(J_\chi w)
    &=
    Q_D(w)
    -
    \int_{\Omega_0}(1-\chi^2(x))
    \bigl(\varepsilon^2|w'(x)|^2+V(x)|w(x)|^2\bigr)\,\mathrm{d}x\\
    &\hspace{2em}
    -
    \varepsilon^2\int_{\Omega_0}\chi(x)\chi''(x)|w(x)|^2\,\mathrm{d}x\\
    &\le
    Q_D(w)
    +M_\chi C_{D,1}^2\varepsilon^2e^{-2\kappa_D/\varepsilon}
    \|w\|_{L^2(\Omega_0)}^2.
    \end{aligned}
    \]
    {\color{black}It follows from the norm estimate that
    \(J_\chi\) is injective} on \(E_{\varepsilon,k}^{D}\), so
   \(\dim J_\chi E_{\varepsilon,k}^{D}=k\). By the min--max principle, for any $1\le k\le4$,
    we have
    \[
    \begin{aligned}
    \lambda_k^N(\varepsilon)
    &\le
    \sup_{0\ne w\in E_{\varepsilon,k}^{D}}
    \frac{Q_N(J_\chi w)}{\|J_\chi w\|_{L^2(\Omega)}^2}\le
    \frac{\lambda_k^D(\varepsilon)
    +M_\chi C_{D,1}^2\varepsilon^2e^{-2\kappa_D/\varepsilon}}
    {1-C_{D,1}^2e^{-2\kappa_D/\varepsilon}}\le
    \lambda_k^D(\varepsilon)+C_{D,3}e^{-2\kappa_D/\varepsilon},
    \end{aligned}
    \]
    where one may take
    \(C_{D,3}:=2(M_\chi+C_{D,0})C_{D,1}^2\), after decreasing
    \(\varepsilon_D\) so that \(0<\varepsilon<1\). Choose
    $
    S<\kappa_{\rm V}<
    \min\{\kappa_N,\kappa_D\}.
    $
    The two eigenvalue inequalities prove~\textup{(\ref{item:nd-eigenvalues})} with
   $
    C_{\rm V}:=C_{N,3}+C_{D,3},
 $ and $
    \varepsilon_{\rm V}:=\min\{\varepsilon_N,\varepsilon_D\}.
   $
    
    \medskip
    \noindent\textup{(\ref{item:nd-spectral-spaces})}
    Apply \eqref{eq:dirichlet-cutoff-decay} with \(k=3\) and the cutoff \(\chi\) in \eqref{eq:nd-cutoff}. After
    decreasing \(\varepsilon_D\), if necessary, one has
    \[
    \|J_\chi w\|_{L^2(\Omega)}
    \ge
    2^{-1/2}\|w\|_{L^2(\Omega_0)},
    \qquad w\in E_{\varepsilon,3}^{D}.
    \]
    Thus \(J_\chi:E_{\varepsilon,3}^{D}\to L^2(\Omega)\) is injective and
    \(\dim\widetilde E_{\varepsilon,0}^D=3\).
    
    By~\textup{(\ref{item:nd-eigenvalues})} and
    \(\lambda_4^D(\varepsilon)-\lambda_3^D(\varepsilon)\ge\gamma\varepsilon\), one has
    \[
    \lambda_4^D(\varepsilon)-\lambda_3^N(\varepsilon)
    \ge
    \frac{\gamma}{2}\varepsilon
    \]
    for all sufficiently small \(\varepsilon>0\). For \(j=1,2,3\), set
    \(v_j=\chi u_j^N|_{\Omega_0}\in H_0^1(\Omega_0)\). The product rule gives
    \[
    (H_\varepsilon^{D}-\lambda_j^N)v_j
    =
    -\varepsilon^2\chi''u_j^N
    -
    2\varepsilon^2\chi'(u_j^N)',
    \qquad x\in\Omega_0.
    \]
    By the definition of \(J_\chi\) and the inclusion
    \(\operatorname{supp}(1-\chi^2)\cup
    \operatorname{supp}\chi'\cup\operatorname{supp}\chi''
    \subset K_\chi\), we have
    \[
    \begin{aligned}
    \|u_j^N-J_\chi v_j\|_{L^2(\Omega)}
    &=
    \|(1-\chi^2)u_j^N\|_{L^2(\Omega)},
     \end{aligned}
    \]
    and
    \[
    \begin{aligned}
    \|(H_\varepsilon^{D}-\lambda_j^N)v_j\|_{L^2(\Omega_0)}
    &\le
    \varepsilon^2\|\chi''\|_{L^\infty(\Omega_0)}
    \|u_j^N\|_{L^2(K_\chi)}
    +
    2\varepsilon^2\|\chi'\|_{L^\infty(\Omega_0)}
    \|(u_j^N)'\|_{L^2(K_\chi)}.
    \end{aligned}
    \]
    By \eqref{eq:neumann-cutoff-decay}, there exists \(C_{\rm R}>0\) such that, for
    \(0<\varepsilon<\min\{1,\varepsilon_N\}\),
    \begin{equation}
    \|u_j^N-J_\chi v_j\|_{L^2(\Omega)}
    +
    \|(H_\varepsilon^{D}-\lambda_j^N)v_j\|_{L^2(\Omega_0)}
    \le
    C_{\rm R}e^{-\kappa_N/\varepsilon}.
    \label{eq:cutoff-residual}
    \end{equation}
    Put
    $
    w_j:=(I-\Pi_{\varepsilon,0}^D)v_j .
    $
    If \(v_j=\sum_{k\ge1}a_{jk}e_k^D\), where
    \(H_\varepsilon^{D}e_k^D=\lambda_k^D e_k^D\), then
    \[
    \begin{aligned}
    \|(H_\varepsilon^{D}-\lambda_j^N)w_j\|_{L^2(\Omega_0)}^2
   =
    \sum_{k\ge4}
    |\lambda_k^D-\lambda_j^N|^2|a_{jk}|^2
    &\ge
    \frac{\gamma^2\varepsilon^2}{4}
    \|w_j\|_{L^2(\Omega_0)}^2 .
    \end{aligned}
    \]
    Since \(\Pi_{\varepsilon,0}^D\) is a spectral projection of
    \(H_\varepsilon^{D}\), the two operators commute, and hence
    \[
    (H_\varepsilon^{D}-\lambda_j^N)w_j
    =
    (I-\Pi_{\varepsilon,0}^D)
    (H_\varepsilon^{D}-\lambda_j^N)v_j .
    \]
    Hence
    \[
    \begin{aligned}
    \frac{\gamma\varepsilon}{2}\|w_j\|_{L^2(\Omega_0)}
    &\le
    \|(H_\varepsilon^{D}-\lambda_j^N)w_j\|_{L^2(\Omega_0)}\le
    \|(H_\varepsilon^{D}-\lambda_j^N)v_j\|_{L^2(\Omega_0)}
    \le
    C_{\rm R} e^{-\kappa_N/\varepsilon}.
    \end{aligned}
    \]
    Choose \(S<\kappa<\kappa_N\). Since
    \(\varepsilon^{-1}e^{-\kappa_N/\varepsilon}\le
    C_{\rm A}e^{-\kappa/\varepsilon}\) for some \(C_{\rm A}>0\) and all
    sufficiently small \(\varepsilon>0\), one has
    \[
    \|(I-\Pi_{\varepsilon,0}^D)v_j\|_{L^2(\Omega_0)}
    \le
    \frac{2C_{\rm R}C_{\rm A}}{\gamma}
    e^{-\kappa/\varepsilon}.
    \]
    Therefore, there exists $ C_{\rm I}>0$ such that
    \[
    \begin{aligned}
    \|(I-\widetilde\Pi_{\varepsilon,0}^D)u_j^N\|_{L^2(\Omega)}
    &\le
    \|u_j^N-J_\chi\Pi_{\varepsilon,0}^Dv_j\|_{L^2(\Omega)}\\
    &\le
    \|u_j^N-J_\chi v_j\|_{L^2(\Omega)}
    +
    \|J_\chi(I-\Pi_{\varepsilon,0}^D)v_j\|_{L^2(\Omega)}\\
    &\le
    C_{\rm I} e^{-\kappa/\varepsilon}.
    \end{aligned}
    \]
    For \(u=\sum_{j=1}^3a_j u_j^N\), the Cauchy--Schwarz inequality gives
    \(\sum_{j=1}^3|a_j|\le\sqrt3\|u\|_{L^2(\Omega)}\). Consequently, one has
    \[
    \|(I-\widetilde\Pi_{\varepsilon,0}^D)\Pi_\varepsilon^N\|
    \le
    \sqrt3\,C_{\rm I} e^{-\kappa/\varepsilon}.
    \]
    Set \(\eta_\varepsilon=\sqrt3\,C_{\rm I}e^{-\kappa/\varepsilon}\). For small
    \(\varepsilon\), \(\eta_\varepsilon<1/2\). If \(u\in E_{\varepsilon,3}^{N}\), then
    \[
    \|\widetilde\Pi_{\varepsilon,0}^D u\|_{L^2(\Omega)}
    \ge
    (1-\eta_\varepsilon)\|u\|_{L^2(\Omega)} .
    \]
    Hence the map
    \(\widetilde\Pi_{\varepsilon,0}^D:E_{\varepsilon,3}^{N}\to
    \widetilde E_{\varepsilon,0}^D\) is injective. Since both spaces have
    dimension \(3\), it is onto. Thus, for every
    \(z\in\widetilde E_{\varepsilon,0}^D\), there is \(u\in E_{\varepsilon,3}^{N}\)
    such that \(z=\widetilde\Pi_{\varepsilon,0}^Du\), and
    \[
    \|u\|_{L^2(\Omega)}
    \le
    (1-\eta_\varepsilon)^{-1}\|z\|_{L^2(\Omega)} .
    \]
    Since \(\Pi_\varepsilon^N\) is the \(L^2(\Omega)\)-orthogonal projection
    onto \(E_{\varepsilon,3}^{N}\) and \(u\in E_{\varepsilon,3}^{N}\), the Pythagorean identity gives
    \[
    \begin{aligned}
    \|z-u\|_{L^2(\Omega)}^2
    &=
    \|(I-\Pi_\varepsilon^N)z\|_{L^2(\Omega)}^2
    +
    \|\Pi_\varepsilon^Nz-u\|_{L^2(\Omega)}^2\ge
    \|(I-\Pi_\varepsilon^N)z\|_{L^2(\Omega)}^2 .
    \end{aligned}
    \]
    Thus
    \[
    \begin{aligned}
    \|(I-\Pi_\varepsilon^N)z\|_{L^2(\Omega)}
    &\le
    \|z-u\|_{L^2(\Omega)}
    =
    \|(I-\widetilde\Pi_{\varepsilon,0}^D)u\|_{L^2(\Omega)}
   \le
    \eta_\varepsilon(1-\eta_\varepsilon)^{-1}
    \|z\|_{L^2(\Omega)} .
    \end{aligned}
    \]
    Taking the supremum over
    \(0\ne z\in\widetilde E_{\varepsilon,0}^D\) gives
    \[
    \|(I-\Pi_\varepsilon^N)\widetilde\Pi_{\varepsilon,0}^D\|
    \le
    \frac{\eta_\varepsilon}{1-\eta_\varepsilon}
    \le 2\sqrt3\,C_{\rm I}e^{-\kappa/\varepsilon}.
    \]
    For the two orthogonal projections used here, one has
    \[
    (\Pi_\varepsilon^N-\widetilde\Pi_{\varepsilon,0}^D)^2
    =
    \Pi_\varepsilon^N(I-\widetilde\Pi_{\varepsilon,0}^D)\Pi_\varepsilon^N
    +
    (I-\Pi_\varepsilon^N)\widetilde\Pi_{\varepsilon,0}^D
    (I-\Pi_\varepsilon^N).
    \]
    The two terms on the right act on the orthogonal subspaces
    \(E_{\varepsilon,3}^{N}\) and \((E_{\varepsilon,3}^{N})^\perp\), respectively. More
    precisely, one has
    \[
    \Pi_\varepsilon^N(I-\widetilde\Pi_{\varepsilon,0}^D)\Pi_\varepsilon^N
    =
    \bigl[(I-\widetilde\Pi_{\varepsilon,0}^D)\Pi_\varepsilon^N\bigr]^*
    \bigl[(I-\widetilde\Pi_{\varepsilon,0}^D)\Pi_\varepsilon^N\bigr]
    \]
    and
    \[
    (I-\Pi_\varepsilon^N)\widetilde\Pi_{\varepsilon,0}^D(I-\Pi_\varepsilon^N)
    =
    \bigl[(I-\Pi_\varepsilon^N)\widetilde\Pi_{\varepsilon,0}^D\bigr]
    \bigl[(I-\Pi_\varepsilon^N)\widetilde\Pi_{\varepsilon,0}^D\bigr]^*.
    \]
    Since
    \(\Pi_\varepsilon^N-\widetilde\Pi_{\varepsilon,0}^D\) is self-adjoint, one has
    \[
    \begin{aligned}
    \|\Pi_\varepsilon^N-\widetilde\Pi_{\varepsilon,0}^D\|^2
    &=
    \|(\Pi_\varepsilon^N-\widetilde\Pi_{\varepsilon,0}^D)^2\|=
    \max\left\{
    \|(I-\widetilde\Pi_{\varepsilon,0}^D)\Pi_\varepsilon^N\|^2,
    \|(I-\Pi_\varepsilon^N)\widetilde\Pi_{\varepsilon,0}^D\|^2
    \right\}.
    \end{aligned}
    \]
    {\color{black}By the two one-sided estimates, we obtain}
    \[
    \|\widetilde\Pi_{\varepsilon,0}^D-\Pi_\varepsilon^N\|
    =
    \max
    \left\{
    \|(I-\widetilde\Pi_{\varepsilon,0}^D)\Pi_\varepsilon^N\|,
    \|(I-\Pi_\varepsilon^N)\widetilde\Pi_{\varepsilon,0}^D\|
    \right\}
    \le
    2\sqrt3\,C_{\rm I} e^{-\kappa/\varepsilon}.
    \]
    This proves~\textup{(\ref{item:nd-spectral-spaces})}, with
    \(C_{\rm P}=2\sqrt3\,C_{\rm I}\) and \(\kappa_{\rm P}=\kappa\).
\end{proof}

Put
\[
B=
\begin{pmatrix}
    0&-1&0\\
    -1&0&-1\\
    0&-1&0
\end{pmatrix}.
\]

\begin{lemma}
    \label{lem:tunneling-matrix}
    Assume \eqref{ass:tunneling}. Then
   $
    W_\varepsilon=\tau_\varepsilon B+o(\tau_\varepsilon).
   $
\end{lemma}

\begin{proof}
For each \(j\in\{1,2,3\}\), extend
\(V|_{\overline\Omega_j}\) to a nonnegative function
\(\widetilde V_j\in C^\infty(\mathbb R)\) such that \(x_j\) is its only
zero and \(\widetilde V_j=1\) outside a compact set. Let
\(\widetilde\mu_k^{(j)}(\varepsilon)\) be the \(k\)-th eigenvalue of
\[
-\varepsilon^2\frac{\mathrm d^2}{\mathrm dx^2}+\widetilde V_j
\quad\hbox{on }L^2(\mathbb R).
\]
Since \(\widetilde V_j''(x_j)=V''(x_j)>0\),
\cite[Theorem~1.1]{Simon1983}, applied with \(\lambda=\varepsilon^{-1}\),
gives
\begin{equation}
\widetilde\mu_k^{(j)}(\varepsilon)
=\varepsilon\nu_k^{(j)}+o(\varepsilon),
\qquad k=1,2,3.
\label{eq:extended-local-harmonic-asymptotics}
\end{equation}
Choose \(C_{*,j}\) strictly between the second and third eigenvalues of
\(\mathcal H_j\). It follows that the full-line operator has exactly two
eigenvalues in \([0,C_{*,j}\varepsilon]\) for all sufficiently small
\(\varepsilon\). The multiplicity-preserving domain comparison in
\cite[Lemma~3.4.3]{Helffer1988} identifies these eigenvalues with the
first two eigenvalues of \(P_j(\varepsilon)\) and gives constants
\(C_j,c_j>0\), independent of \(\varepsilon\), such that
\[
\left|\mu_k^{(j)}(\varepsilon)
-\widetilde\mu_k^{(j)}(\varepsilon)\right|
\le C_j e^{-c_j/\varepsilon},
\qquad k=1,2.
\]
Together with \eqref{eq:extended-local-harmonic-asymptotics}, this proves
\begin{equation}
\mu_k^{(j)}(\varepsilon)
=\varepsilon\nu_k^{(j)}+o(\varepsilon),
\qquad k=1,2,\quad j=1,2,3.
\label{eq:local-harmonic-asymptotics}
\end{equation}
The harmonic oscillator satisfies
\(\nu_2^{(j)}-\nu_1^{(j)}=2\sqrt{V''(x_j)/2}>0\).
Thus, with \(\gamma_j=(\nu_2^{(j)}-\nu_1^{(j)})/2\),
\eqref{eq:local-harmonic-asymptotics} yields
\begin{equation}
\mu_2^{(j)}(\varepsilon)-\mu_1^{(j)}(\varepsilon)
\ge\gamma_j\varepsilon,
\qquad j=1,2,3,
\label{eq:local-dirichlet-gap}
\end{equation}
for all sufficiently small \(\varepsilon>0\).

{\color{black}By the Dirichlet decay estimate in
\cite[Proposition~3.3.5]{Helffer1988}, the local eigenfunctions satisfy
\cite[(4.3.18)]{Helffer1988}.}
\begingroup
\color{black}
Moreover, \eqref{eq:local-eigenvalue-assumption} implies
\begin{equation}
\mu_1^{(j)}(\varepsilon)-\mu_1^{(k)}(\varepsilon)
=o\!\left(\varepsilon^{1/2}e^{-S/\varepsilon}\right)=O(\varepsilon^N),
\qquad j,k=1,2,3,
\label{eq:local-eigenvalue-matching}
\end{equation}
for every \(N>0\), which verifies condition
\cite[(4.3.17)]{Helffer1988}.
\par
\endgroup

In one dimension, \eqref{ass:tunneling}\textup{(\ref{item:tunneling-actions})}
gives \(S_{13}=S_{12}+S_{23}=2S>2r\), so
\(\Omega_1\cap\Omega_3=\varnothing\). By
\eqref{eq:interaction-matrix-definition},
$
w_{13}=w_{31}=0.
$
For each adjacent pair \((j,k)=(1,2),(2,3)\), the interval
\([x_j,x_k]\) is the unique shortest path joining the two points in
the Agmon metric. {\color{black}A point in \(\Omega_j\cap\Omega_k\) is a
transverse section, since there are no tangential directions in one dimension.}
Thus the geometric hypotheses
\cite[(4.4.41)--(4.4.42)]{Helffer1988} hold.
By \cite[Theorem~4.4.8]{Helffer1988} and
\eqref{ass:tunneling}\textup{(\ref{item:tunneling-coefficients})},
\begin{equation*}
\begin{aligned}
w_{12}
&=-\varepsilon^{1/2}\bigl(b_{12}^{\,0}+O(\varepsilon)\bigr)
e^{-S/\varepsilon}
=-\tau_\varepsilon+o(\tau_\varepsilon),\\
w_{23}
&=-\varepsilon^{1/2}\bigl(b_{23}^{\,0}+O(\varepsilon)\bigr)
e^{-S/\varepsilon}
=-\tau_\varepsilon+o(\tau_\varepsilon).
\end{aligned}
\end{equation*}
The same asymptotics hold for \(w_{21}\) and \(w_{32}\). Thus,
\begin{equation}
W_\varepsilon=\tau_\varepsilon B+o(\tau_\varepsilon).
\label{eq:interaction-matrix-limit}
\end{equation}
\end{proof}

\begin{lemma}
    \label{lem:low-energy-reduction}
    Assume \eqref{ass:tunneling}. There is an orthonormal basis of \(E_{\varepsilon,3}^{N}\) such that, for some
    \(\kappa>S\), one has
    \[
    \Pi_\varepsilon^NH_\varepsilon^N\Pi_\varepsilon^N
    =
    \mathcal M_{\varepsilon,0}^{D}+O(e^{-\kappa/\varepsilon})
    =
    \mathcal M_{\varepsilon,0}^{D}+o(\tau_\varepsilon).
    \]
\end{lemma}

\begin{proof}
    {\color{black}Let \(\chi\) be the cutoff fixed in \eqref{eq:nd-cutoff}. Then}
    \[
    d_A(\operatorname{supp}(1-\chi),V^{-1}(0))>S,\qquad
    \text{and}\qquad
    d_A(\operatorname{supp}\chi'\cup\operatorname{supp}\chi'',V^{-1}(0))>S.
    \]
    {\color{black}For \(u\in E_{\varepsilon,3}^{D}\), we use the extension \(J_\chi\) defined}
    before Lemma~\ref{lem:ND-comparison}. Applying
    \eqref{eq:dirichlet-cutoff-decay} with \(k=3\) and the cutoff in
    \eqref{eq:nd-cutoff} gives constants
    \(C_1>0\), \(\kappa_1>S\), and \(\varepsilon_1>0\) such that
    \begin{equation}
    \|J_\chi u-u\|_{L^2(\Omega_0)}
    \le
    C_1 e^{-\kappa_1/\varepsilon}\|u\|_{L^2(\Omega_0)},\qquad 0<\varepsilon<\varepsilon_1.
    \label{eq:transfer-cutoff-norm}
    \end{equation}
   On \(\Omega_0\),
    a direct calculation gives
    \begin{equation}
    (H_\varepsilon^N J_\chi-J_\chi H_\varepsilon^{D})u
    =
    -\varepsilon^2\bigl(2\chi' u'+\chi''u\bigr),
    \label{eq:transfer-commutator}
    \end{equation}
    and this function is extended by zero to \(\Omega\setminus\Omega_0\).
    {\color{black}By \eqref{eq:dirichlet-cutoff-decay}, we have estimates for both \(u\) and
    \(\varepsilon u'\) on}
    \(\operatorname{supp}\chi'\cup\operatorname{supp}\chi''\).
    {\color{black}Combining this estimate with \eqref{eq:transfer-cutoff-norm} and
    \eqref{eq:transfer-commutator}, for \(0<\varepsilon<\min\{1,\varepsilon_1\}\)
    sufficiently small, we obtain constants \(C_2>0\) and \(\kappa_2>S\) such that}
    \begin{equation}
    \|J_\chi u-u\|_{L^2(\Omega_0)}
    +
    \|(H_\varepsilon^N J_\chi
    -
    J_\chi H_\varepsilon^{D})u\|_{L^2(\Omega)}
    \le
    C_2 e^{-\kappa_2/\varepsilon}\|u\|_{L^2(\Omega_0)}.
    \label{eq:transfer-residual}
    \end{equation}
    
    With \(\gamma_*:=\min_{1\le j\le3}\gamma_j>0\), the local gap
    \eqref{eq:local-dirichlet-gap} and \eqref{ass:tunneling} give
    \[
    \mu_2^{(j)}(\varepsilon)-\mu_1^{(j)}(\varepsilon)
    \ge
    \gamma_*\varepsilon,
    \qquad\text{and}\qquad
    \max_j\mu_1^{(j)}(\varepsilon)-\min_j\mu_1^{(j)}(\varepsilon)
    =o(\varepsilon).
    \]
    {\color{black}To order the local eigenvalues, set}
\(P(\varepsilon)=P_1(\varepsilon)\oplus P_2(\varepsilon)\oplus
P_3(\varepsilon)\), and denote its eigenvalues by
\(\eta_1(\varepsilon)\le\eta_2(\varepsilon)\le\cdots\).
By the direct-sum spectral identity
\cite[Chapter~V, Section~3.9]{Kato1984},
\[
\sigma(P(\varepsilon))=\bigcup_{j=1}^3\sigma(P_j(\varepsilon))
=\{\mu_\ell^{(j)}(\varepsilon):\ell\ge1,\ 1\le j\le3\},
\]
with eigenvalues counted with multiplicity. The preceding estimates imply
\[
\min_j\mu_2^{(j)}-\max_j\mu_1^{(j)}
\ge\gamma_*\varepsilon-
\bigl(\max_j\mu_1^{(j)}-\min_j\mu_1^{(j)}\bigr)
\ge\frac{\gamma_*}{2}\varepsilon>0
\]
for all sufficiently small \(\varepsilon\). Then we have
\[
\{\eta_1,\eta_2,\eta_3\}
=\{\mu_1^{(1)},\mu_1^{(2)},\mu_1^{(3)}\},\qquad
\eta_4=\min_j\mu_2^{(j)}.
\]
In particular, with \(\gamma_2=\gamma_*/2\), one has
$
\eta_4(\varepsilon)-\eta_3(\varepsilon)\ge\gamma_2\varepsilon.
$

    Apply Helffer \cite[Theorem~4.2.1]{Helffer1988} to
    \(P(\varepsilon)\) and \(H_\varepsilon^{D}\) in a low-energy interval
    containing at least the first four eigenvalues, with its endpoints
    separated from the local harmonic spectra after division by \(\varepsilon\).
    Ordering the two spectra and restricting to the first four eigenvalues
    gives a bijection
    \[
    \beta_\varepsilon:
    \{\lambda_k^D(\varepsilon):1\le k\le4\}
    \to
    \{\eta_k(\varepsilon):1\le k\le4\}
    \]
    between the two finite lists, with eigenvalues counted according to
    multiplicity, such that
    \[
    |\beta_\varepsilon(\lambda)-\lambda|
    \le
    C_3 e^{-\kappa_3/\varepsilon}
    \]
    for some \(C_3>0\) and \(0<\kappa_3<S\).
  It is easy to see that 
    \[
    |\lambda_k^D(\varepsilon)-\eta_k(\varepsilon)|
    \le
    C_3 e^{-\kappa_3/\varepsilon},
    \qquad k=1,2,3,4,
    \]
    and therefore
    \[
    \lambda_4^D(\varepsilon)-\lambda_3^D(\varepsilon)
    \ge
    \frac{\gamma_2}{2}\varepsilon .
    \]
    {\color{black}By Lemma~\ref{lem:ND-comparison}\textup{(\ref{item:nd-spectral-spaces})}, there exist \(C_4>0\) and \(\kappa_4>S\) such that} 
    \[
    \|\widetilde\Pi_{\varepsilon,0}^D-\Pi_\varepsilon^N\|
    \le
    C_4 e^{-\kappa_4/\varepsilon},
    \quad\text{and}\quad
    \operatorname{Ran}\widetilde\Pi_{\varepsilon,0}^D
    =
    J_\chi E_{\varepsilon,3}^{D}.
    \] Fix
    \(\kappa_0>S\) with
    \(\kappa_0<\min\{\kappa_1,\kappa_2,\kappa_4\}\), choose \(C_0>0\) large
    enough, and put
    $
    \omega_\varepsilon:=C_0e^{-\kappa_0/\varepsilon}.
    $
    Then, for \(u\in E_{\varepsilon,3}^{D}\), one has
    \[
    \|J_\chi u-u\|_{L^2(\Omega_0)}
    +
    \|(H_\varepsilon^NJ_\chi-J_\chi H_\varepsilon^{D})u\|_{L^2(\Omega)}
    \le
    \omega_\varepsilon\|u\|_{L^2(\Omega_0)}.
    \]
    The preceding projection estimate also gives
    \[
    \|\widetilde\Pi_{\varepsilon,0}^D-\Pi_\varepsilon^N\|
    \le\omega_\varepsilon .
    \]
    
    Put \(T_\varepsilon:=\Pi_\varepsilon^NJ_\chi:
    E_{\varepsilon,3}^{D}\to E_{\varepsilon,3}^{N}\).
    For \(u\in E_{\varepsilon,3}^{D}\), one has
    \[
    \widetilde\Pi_{\varepsilon,0}^D J_\chi u=J_\chi u,
    \qquad\text{and}\qquad
    (I-\Pi_\varepsilon^N)J_\chi u
    =
    (\widetilde\Pi_{\varepsilon,0}^D-\Pi_\varepsilon^N)J_\chi u,
    \]
    Consequently, one has
    \[
    \begin{aligned}
    \|(I-\Pi_\varepsilon^N)J_\chi u\|_{L^2(\Omega)}
    \le
    \omega_\varepsilon\|J_\chi u\|_{L^2(\Omega)}
    \le
    \omega_\varepsilon\|u\|_{L^2(\Omega_0)}.
    \end{aligned}
    \]
    Combining this estimate with \eqref{eq:transfer-residual}, for
    \(u,v\in E_{\varepsilon,3}^{D}\), one has
    \[
    \begin{aligned}
    \langle T_\varepsilon u,T_\varepsilon v\rangle_{L^2(\Omega)}
    &=
    \langle J_\chi u,J_\chi v\rangle_{L^2(\Omega)}
    +O(\omega_\varepsilon)
    \|u\|_{L^2(\Omega_0)}\|v\|_{L^2(\Omega_0)}\\
    &=
    \langle u,v\rangle_{L^2(\Omega_0)}
    +O(\omega_\varepsilon)
    \|u\|_{L^2(\Omega_0)}\|v\|_{L^2(\Omega_0)}.
    \end{aligned}
    \]
    {\color{black}The adjoint \(T_\varepsilon^*\), with respect to the two Hilbert-space
    inner products, satisfies}
    \[
    \langle T_\varepsilon^* f,u\rangle_{L^2(\Omega_0)}
    =
    \langle f,T_\varepsilon u\rangle_{L^2(\Omega)},
    \qquad
    f\in E_{\varepsilon,3}^{N},\quad u\in E_{\varepsilon,3}^{D} .
    \]
    Hence, for \(u,v\in E_{\varepsilon,3}^{D}\), one has
    \[
    \begin{aligned}
    \left\langle
    (T_\varepsilon^*T_\varepsilon-I_{E_{\varepsilon,3}^{D}})u,
    v
    \right\rangle_{L^2(\Omega_0)}
    &=
    \langle T_\varepsilon u,T_\varepsilon v\rangle_{L^2(\Omega)}
    -
    \langle u,v\rangle_{L^2(\Omega_0)}=
    O(\omega_\varepsilon)
    \|u\|_{L^2(\Omega_0)}\|v\|_{L^2(\Omega_0)}.
    \end{aligned}
    \]
    Therefore
    \begin{equation}
    T_\varepsilon^*T_\varepsilon
    =
    I_{E_{\varepsilon,3}^{D}}+O(\omega_\varepsilon)
    \quad\text{on }E_{\varepsilon,3}^{D}.
    \label{eq:transfer-isometry}
    \end{equation}
    Hence \(T_\varepsilon\) is invertible from \(E_{\varepsilon,3}^{D}\) onto
    \(E_{\varepsilon,3}^{N}\) for small \(\varepsilon\). Define
    $
    U_\varepsilon
    =
    T_\varepsilon(T_\varepsilon^*T_\varepsilon)^{-1/2}.
   $
    Then \(U_\varepsilon:E_{\varepsilon,3}^{D}\to E_{\varepsilon,3}^{N}\) is unitary
    and
    \begin{equation}
    U_\varepsilon=T_\varepsilon+O(\omega_\varepsilon).
    \label{eq:transfer-unitary}
    \end{equation}
    Let \(u,v\in E_{\varepsilon,3}^{D}\) and
    \(\|u\|_{L^2(\Omega_0)}=\|v\|_{L^2(\Omega_0)}=1\).
    Lemmas~\ref{lem:bound-first-four} and
    \ref{lem:ND-comparison}\textup{(\ref{item:nd-eigenvalues})} give constants
    \(C_{\rm W}>0\) and \(\varepsilon_{\rm W}>0\) such that
    \[
    \lambda_3^N(\varepsilon)+\lambda_3^D(\varepsilon)
    \le C_{\rm W}\varepsilon
    \]
    for \(0<\varepsilon<\varepsilon_{\rm W}\).
    Since \(U_\varepsilon u,T_\varepsilon u\in E_{\varepsilon,3}^{N}\),
    the spectral definitions and \eqref{eq:transfer-residual} give a constant
    \(C_5>0\) such that
    \[
    \|H_\varepsilon^{D}u\|_{L^2(\Omega_0)}
    +
    \|H_\varepsilon^NU_\varepsilon u\|_{L^2(\Omega)}
    +
    \|H_\varepsilon^NT_\varepsilon u\|_{L^2(\Omega)}
    +
    \|H_\varepsilon^NJ_\chi u\|_{L^2(\Omega)}
    \le
    C_5\varepsilon .
    \]
    On \(E_{\varepsilon,3}^{D}\), the identity
    \((I-\Pi_\varepsilon^N)J_\chi u=J_\chi u-T_\varepsilon u\)
    and the estimate
    \(\|(I-\Pi_\varepsilon^N)J_\chi\|
    \le\omega_\varepsilon\) imply
    \(\|T_\varepsilon-J_\chi\|\le\omega_\varepsilon\). {\color{black}Combining this with
    \eqref{eq:transfer-unitary}, we obtain \(C_6>0\) such that}
    \[
    \|U_\varepsilon-T_\varepsilon\|
    +
    \|T_\varepsilon-J_\chi\|
    \le
    C_6\omega_\varepsilon .
    \]
    Hence
    \[
    \begin{aligned}
    &\left|
    \left\langle
    H_\varepsilon^NU_\varepsilon u,U_\varepsilon v
    \right\rangle_{L^2(\Omega)}
    -
    \left\langle
    H_\varepsilon^NT_\varepsilon u,T_\varepsilon v
    \right\rangle_{L^2(\Omega)}
    \right|\\
    &\quad\le
    \|H_\varepsilon^NU_\varepsilon u\|_{L^2(\Omega)}
    \|(U_\varepsilon-T_\varepsilon)v\|_{L^2(\Omega)}
    +
    \|H_\varepsilon^NT_\varepsilon v\|_{L^2(\Omega)}
    \|(U_\varepsilon-T_\varepsilon)u\|_{L^2(\Omega)}
    \le
    C_7\varepsilon\omega_\varepsilon ,
    \end{aligned}
    \]
    and
    \[
    \begin{aligned}
    &\left|
    \left\langle
    H_\varepsilon^NT_\varepsilon u,T_\varepsilon v
    \right\rangle_{L^2(\Omega)}
    -
    \left\langle
    H_\varepsilon^NJ_\chi u,J_\chi v
    \right\rangle_{L^2(\Omega)}
    \right|\\
    &\quad\le
    \|H_\varepsilon^NT_\varepsilon u\|_{L^2(\Omega)}
    \|(T_\varepsilon-J_\chi)v\|_{L^2(\Omega)}
    +
    \|H_\varepsilon^NJ_\chi v\|_{L^2(\Omega)}
    \|(T_\varepsilon-J_\chi)u\|_{L^2(\Omega)}
    \le
    C_8\varepsilon\omega_\varepsilon .
    \end{aligned}
    \]
    By \eqref{eq:transfer-residual}, there exists \(C_9>0\) such that
    \[
    \begin{aligned}
    &\left|
    \left\langle
    H_\varepsilon^NJ_\chi u,J_\chi v
    \right\rangle_{L^2(\Omega)}
    -
    \left\langle
    J_\chi H_\varepsilon^{D}u,J_\chi v
    \right\rangle_{L^2(\Omega)}
    \right|\\
    &\quad\le
    \|(H_\varepsilon^NJ_\chi-J_\chi H_\varepsilon^{D})u\|_{L^2(\Omega)}
    \|J_\chi v\|_{L^2(\Omega)}
    \le
    C_9\omega_\varepsilon .
    \end{aligned}
    \]
    Since \(H_\varepsilon^{D}u\in E_{\varepsilon,3}^{D}\),
    \eqref{eq:transfer-residual} also applies to
    \(H_\varepsilon^{D}u\). Consequently, there exists \(C_{10}>0\) such that
    \[
    \begin{aligned}
    &\left|
    \left\langle
    J_\chi H_\varepsilon^{D}u,J_\chi v
    \right\rangle_{L^2(\Omega)}
    -
    \left\langle
    H_\varepsilon^{D}u,v
    \right\rangle_{L^2(\Omega_0)}
    \right|\le
    C_{10}\varepsilon\|J_\chi v-v\|_{L^2(\Omega_0)}
    \le
    C_{10}\varepsilon\omega_\varepsilon .
    \end{aligned}
    \]
    Therefore
    \[
    \begin{aligned}
    \left\langle
    H_\varepsilon^NU_\varepsilon u,
    U_\varepsilon v
    \right\rangle_{L^2(\Omega)}
    &=
    \left\langle
    H_\varepsilon^NT_\varepsilon u,
    T_\varepsilon v
    \right\rangle_{L^2(\Omega)}
    +
    O(\omega_\varepsilon)\\
    &=
    \left\langle
    H_\varepsilon^NJ_\chi u,
    J_\chi v
    \right\rangle_{L^2(\Omega)}
    +
    O(\omega_\varepsilon)\\
    &=
    \left\langle
    J_\chi H_\varepsilon^{D}u,
    J_\chi v
    \right\rangle_{L^2(\Omega)}
    +
    O(\omega_\varepsilon)\\
    &=
    \left\langle
    H_\varepsilon^{D}u,
    v
    \right\rangle_{L^2(\Omega_0)}
    +
    O(\omega_\varepsilon).
    \end{aligned}
    \]
    {\color{black}It follows from these four equalities that there exists \(C_{11}>0\) such that}
    \[
    \left|
    \left\langle
    H_\varepsilon^NU_\varepsilon u,
    U_\varepsilon v
    \right\rangle_{L^2(\Omega)}
    -
    \left\langle
    H_\varepsilon^{D}u,
    v
    \right\rangle_{L^2(\Omega_0)}
    \right|
    \le
    C_{11}\omega_\varepsilon.
    \]
    Let \(\{\xi_{1,\varepsilon}^{D},\xi_{2,\varepsilon}^{D},
    \xi_{3,\varepsilon}^{D}\}\) be the real orthonormal basis of
    \(E_{\varepsilon,3}^{D}\) used to represent \(\mathcal M_{\varepsilon,0}^{D}\).
    By its definition,
    \begin{equation}
    (\mathcal M_{\varepsilon,0}^{D})_{ij}
    =\left\langle H_\varepsilon^{D}\xi_{j,\varepsilon}^{D},
    \xi_{i,\varepsilon}^{D}\right\rangle_{L^2(\Omega_0)},
    \qquad 1\le i,j\le3.
    \label{eq:dirichlet-matrix-entries}
    \end{equation}
    Set \(\zeta_{j,\varepsilon}:=U_\varepsilon\xi_{j,\varepsilon}^{D}\),
    \(j=1,2,3\). Since \(U_\varepsilon\) is unitary,
    \(\{\zeta_{1,\varepsilon},\zeta_{2,\varepsilon},
    \zeta_{3,\varepsilon}\}\) is an orthonormal basis of \(E_{\varepsilon,3}^{N}\).
    Substituting \(u=\xi_{j,\varepsilon}^{D}\) and
    \(v=\xi_{i,\varepsilon}^{D}\) into the preceding estimate and using
    \eqref{eq:dirichlet-matrix-entries} gives
    \begin{equation}
    \left|
    \left\langle H_\varepsilon^N\zeta_{j,\varepsilon},
    \zeta_{i,\varepsilon}\right\rangle_{L^2(\Omega)}
    -(\mathcal M_{\varepsilon,0}^{D})_{ij}
    \right|
    \le C_{11}\omega_\varepsilon,
    \qquad 1\le i,j\le3.
    \label{eq:transfer-matrix-entries}
    \end{equation}
    Since \(\omega_\varepsilon=C_0e^{-\kappa_0/\varepsilon}\) and
    \(\kappa_0>S\), one has \(\omega_\varepsilon=o(\tau_\varepsilon)\).
    Thus \eqref{eq:transfer-matrix-entries} yields, in this basis,
    \begin{equation}
    \Pi_\varepsilon^N H_\varepsilon^N\Pi_\varepsilon^N
    =\mathcal M_{\varepsilon,0}^{D}+O(e^{-\kappa_0/\varepsilon})
    =\mathcal M_{\varepsilon,0}^{D}+o(\tau_\varepsilon),
    \label{eq:transfer-matrix-conclusion}
    \end{equation}
    which proves the lemma.
\end{proof}

Set
\[
G_0
=
\begin{pmatrix}
    -\sqrt2&1&0\\
    1&-\sqrt2&1\\
    0&1&-\sqrt2
\end{pmatrix}.
\]

\begin{lemma}
    \label{lem:three-well-reduction}
    Assume \eqref{ass:tunneling}. Then the following hold.
    \begin{enumerate}
        \item\label{item:three-well-matrix} With
        \(\{\zeta_{1,\varepsilon},\zeta_{2,\varepsilon},\zeta_{3,\varepsilon}\}\)
        as a basis of \(E_{\varepsilon,3}^{N}\), one has
        \[
        \Pi_\varepsilon^N A_\varepsilon \Pi_\varepsilon^N
        =\tau_\varepsilon G_0+o(\tau_\varepsilon).
        \]
        
        \item\label{item:three-well-gap} There exists \(\gamma_0>0\), independent of \(\varepsilon\), such
        that, for all sufficiently small \(\varepsilon>0\), one has
        \[
        \sigma(A_\varepsilon|_{(E_{\varepsilon,3}^{N})^\perp})
        \subset(-\infty,-\gamma_0\varepsilon].
        \]
    \end{enumerate}
\end{lemma}

\begin{proof}
    \noindent\textup{(\ref{item:three-well-matrix})}
\begingroup
\color{black}
    Extend \(\vartheta_j\xi_{j,\varepsilon}\) by zero to \(\Omega_0\),
    using the cutoffs in \eqref{eq:interaction-cutoffs}.
    Projecting these functions onto \(E_{\varepsilon,3}^{D}\) and applying
    the Gram-matrix orthonormalization in
    Helffer \cite[(4.3.13) and Theorem~4.3.1]{Helffer1988}
    gives the basis
    \(\{\xi_{1,\varepsilon}^{D},\xi_{2,\varepsilon}^{D},
    \xi_{3,\varepsilon}^{D}\}\) used to represent
    \(\mathcal M_{\varepsilon,0}^{D}\).
\par
\endgroup
    If
    \(\widehat\xi_{j,\varepsilon}\) is the extension by zero to \(\Omega_0\) of
    the local positive first eigenfunction \(\xi_{j,\varepsilon}\) on
    \(\Omega_j\), then, after choosing signs, the same construction gives
    constants \(C_{\rm L}>0\) and \(\kappa_{\rm L}>0\) such that
    \begin{equation}
    \|\xi_{j,\varepsilon}^{D}-\widehat\xi_{j,\varepsilon}\|_{L^2(\Omega_0)}
    \le C_{\rm L} e^{-\kappa_{\rm L}/\varepsilon},
    \qquad j=1,2,3 .
    \label{eq:local-global-basis}
    \end{equation}
\begingroup
\color{black}
    Let \(\mathcal M_{\varepsilon,0}^{D}\) represent
    \(\Pi_{\varepsilon,0}^{D}H_\varepsilon^{D}\Pi_{\varepsilon,0}^{D}\)
    in this basis. For the cutoffs in \eqref{eq:interaction-cutoffs},
    Helffer \cite[Theorem~4.3.1]{Helffer1988} gives a remainder
    \(O(e^{-\kappa/\varepsilon})=o(\tau_\varepsilon)\) for some \(\kappa>S\).
    Thus, with \(W_\varepsilon\) defined in
    \eqref{eq:interaction-matrix-definition}, we have
    \begin{equation}
    \mathcal M_{\varepsilon,0}^{D}
    =\operatorname{diag}\bigl(
    \mu_1^{(1)}(\varepsilon),
    \mu_1^{(2)}(\varepsilon),
    \mu_1^{(3)}(\varepsilon)\bigr)
    +W_\varepsilon+o(\tau_\varepsilon).
    \label{eq:dirichlet-interaction-decomposition}
    \end{equation}
\par
\endgroup
    By
    \eqref{ass:tunneling}\textup{(\ref{item:tunneling-local-eigenvalues})}
    and Lemma~\ref{lem:tunneling-matrix}, we obtain
    \[
    \Pi_{\varepsilon,0}^{D}
    H_\varepsilon^{D}
    \Pi_{\varepsilon,0}^{D}
    =
    \mu_1^{(1)}(\varepsilon)I+\tau_\varepsilon B+o(\tau_\varepsilon),
    \]
    {\color{black}where \(B\) is the matrix defined before
    Lemma~\ref{lem:tunneling-matrix}. Lemma~\ref{lem:low-energy-reduction}
    gives the corresponding matrix on \(E_{\varepsilon,3}^{N}\). Thus, in the orthonormal basis
    \(\{\zeta_{1,\varepsilon},\zeta_{2,\varepsilon},
    \zeta_{3,\varepsilon}\}\) constructed there, we have}
    \[
    \Pi_\varepsilon^N H_\varepsilon^N\Pi_\varepsilon^N
    =
    \mu_1^{(1)}(\varepsilon)I
    +
    \tau_\varepsilon B
    +
    o(\tau_\varepsilon).
    \]
    {\color{black}The eigenvalues of \(B\) are \(-\sqrt2,0,\sqrt2\). Hence the smallest
    of these three eigenvalues of \(H_\varepsilon^N\) is}
    \[
    \lambda_1^N(\varepsilon)
    =
    \mu_1^{(1)}(\varepsilon)-\sqrt2\,\tau_\varepsilon
    +
    o(\tau_\varepsilon).
    \]
    {\color{black}The operator \(A_\varepsilon\) is \(H_\varepsilon^N\) shifted by this
    lowest eigenvalue and then multiplied by \(-1\). The definitions}
    of \(L_\varepsilon\), \(\nu_\varepsilon\), and \(A_\varepsilon\) give
    \[
    A_\varepsilon
    =
    L_\varepsilon-\nu_\varepsilon I
    =
    -(H_\varepsilon^N-\lambda_1^N(\varepsilon)I).
    \]
    Combining the expansion of
    \(\Pi_\varepsilon^NH_\varepsilon^N\Pi_\varepsilon^N\) with the formula
    for \(A_\varepsilon\) gives
    \[
    \Pi_\varepsilon^N A_\varepsilon \Pi_\varepsilon^N
    =
    -\tau_\varepsilon(B+\sqrt2I)+o(\tau_\varepsilon).
    \]
    Since
    \[
    -(B+\sqrt2I)
    =
    \begin{pmatrix}
        -\sqrt2&1&0\\
        1&-\sqrt2&1\\
        0&1&-\sqrt2
    \end{pmatrix}
    =
    G_0,
    \]
    \(\Pi_\varepsilon^N A_\varepsilon \Pi_\varepsilon^N
    =\tau_\varepsilon G_0+o(\tau_\varepsilon)\).
    This proves~\textup{(\ref{item:three-well-matrix})}.
    
    \medskip
    \noindent\textup{(\ref{item:three-well-gap})}
    The direct-sum ordering established in the proof of Lemma~\ref{lem:low-energy-reduction},
    Helffer~\cite[Theorem~4.2.1]{Helffer1988}, and
    Lemma~\ref{lem:ND-comparison}\textup{(\ref{item:nd-eigenvalues})} give
    \[
    \eta_4(\varepsilon)-\eta_3(\varepsilon)
    \ge
    \gamma_2\varepsilon,
    \qquad\text{and}\qquad
    |\lambda_k^N(\varepsilon)-\eta_k(\varepsilon)|
    \le
    \delta_\varepsilon,\quad k=1,2,3,4,
    \]
    where there exist constants \(C_{\rm P}>0\) and \(\kappa_{\rm P}>0\),
    independent of \(\varepsilon\), such that
    \(\delta_\varepsilon=C_{\rm P}e^{-\kappa_{\rm P}/\varepsilon}
    =o(\varepsilon)\). Hence
    \[
    \lambda_4^N(\varepsilon)-\lambda_3^N(\varepsilon)
    \ge
    \eta_4(\varepsilon)-\eta_3(\varepsilon)-2\delta_\varepsilon
    \ge
    \frac{\gamma_2}{2}\varepsilon
    \]
    for all sufficiently small \(\varepsilon\). Since
    \(\lambda_1^N(\varepsilon)\le\lambda_3^N(\varepsilon)\), there exists
    \(\gamma_0>0\) such that
    \[
    \lambda_4^N(\varepsilon)-\lambda_1^N(\varepsilon)
    \ge
    \gamma_0\varepsilon
    \]
    for all sufficiently small \(\varepsilon>0\). This is equivalent to
    \[
    \sigma\!\left(H_\varepsilon^N\big|_{(E_{\varepsilon,3}^{N})^\perp}\right)
    \subset
    [\lambda_1^N(\varepsilon)+\gamma_0\varepsilon,\infty).
    \]
   Thus,
    \[
    \sigma(A_\varepsilon|_{(E_{\varepsilon,3}^{N})^\perp})
    \subset
    (-\infty,-\gamma_0\varepsilon].
    \]
    This proves~\textup{(\ref{item:three-well-gap})}.
\end{proof}

\section{Instability for growth profiles with three maximum points}
\label{sec:instability-hopf}

{\color{black}In this section, we use the spectral reduction to construct an unstable positive
steady state and Hopf branches of positive periodic solutions. We first give a
matrix result in dimensions one and two, and then present the
three-dimensional construction.}

\begin{proposition}
    \label{prop:one-two-mode-reduction}
    Let \(m\in\{1,2\}\), \(G_m=G_m^T\le0\), and
    \(\ker G_m=\operatorname{span}\{p\}\) for some
    \(p\in(0,\infty)^m\). Then, for every
    \(v,w\in(0,\infty)^m\), one has
    \[
    \sigma(G_m-vw^T)\subset\{z\in\mathbb C:\operatorname{Re}z<0\}.
    \]
    The same conclusion holds for
    \(G_m-vw^T+R_\varepsilon\) whenever
    \(\|R_\varepsilon\|\to0\) as \(\varepsilon\to0^+\), for all sufficiently
    small \(\varepsilon>0\).
\end{proposition}

\begin{proof}
    If \(m=1\), then \(G_1=0\), and the only eigenvalue of
    \(G_1-vw^T\) is \(-vw<0\).

    Let \(m=2\), set \(A=-G_2\), and put
    \(\widehat p=p/\|p\|_2\). Choose a unit vector
    \(e\perp\widehat p\). Since \(G_2\) is symmetric, negative semidefinite,
    and \(\ker G_2=\operatorname{span}\{p\}\), there is a number
    \(\gamma>0\) such that
    \[
    G_2=-\gamma ee^T.
    \]
    The two components of \(e\) have opposite signs because
    \(e\perp p\) and \(p>0\). Hence \(A=\gamma ee^T\) has nonpositive
    off-diagonal entries and eigenvalues \(0\) and \(\gamma\), so it is a
    singular \(M\)-matrix. Its zero eigenvalue is algebraically simple, and its
    left and right nullvectors are both \(p\). Since \(v,w>0\), one has
    \[
    p^Tv>0,
    \qquad\text{and}\qquad
    w^Tp>0.
    \]
    {\color{black}Thus the two nonzero-projection conditions in
    \cite[Theorem~2.7]{BierkensRan2014} hold; moreover, \(v,w\ge0\). Hence
    \cite[Theorem~2.7(i)(a)]{BierkensRan2014}
    applies to \(A+vw^T=-G_2+vw^T\), and we have}
    \[
    \sigma(-G_2+vw^T)
    \subset\{z\in\mathbb C:\operatorname{Re}z>0\},
    \]
    {\color{black}which proves the required inclusion for \(G_2-vw^T\).}
    The spectrum of a finite-dimensional matrix depends continuously on its
    entries. Therefore the same strict inclusion holds for
    \(G_m-vw^T+R_\varepsilon\) when \(\|R_\varepsilon\|\) is sufficiently
    small.
\end{proof}

{\color{black}We now prove Theorem~\ref{thm:three-peak-hopf}. We use the reduction in Section~\ref{sec:semiclassical-reduction}
to construct unstable steady states.
The first three eigenfunctions of \(H_\varepsilon^N\) span
\(E_{\varepsilon,3}^{N}\). After adding the nonlocal term,
the compression to
\(E_{\varepsilon,3}^{N}\) is represented by a non-self-adjoint \(3\times3\) matrix. We
first choose a finite-dimensional matrix with a Hopf crossing, then construct
positive coefficients in the equation that give this matrix, and finally obtain
the corresponding eigenvalues of the full linearized operator.}
\begingroup
\color{black}
\par
Set \(x_j=(j+2)/8\), \(j=1,2,3\), and take
\(\ell\in(1/16,1/8)\) sufficiently close to \(1/8\).
Choose \(\theta\in C_c^\infty((1/4,3/4))\) with
\(0\le\theta\le1\), \(\theta(1-x)=\theta(x)\), and
\(\theta=1\) on \([x_1-\ell,x_3+\ell]\). Define
\begin{equation}
V(x)=\theta(x)\sin^2(8\pi x)+1-\theta(x),
\qquad
r=\frac{1-\cos(8\pi\ell)}{8\pi}.
\label{eq:counterexample-potential}
\end{equation}
Then \(V\ge0\), \(V^{-1}(0)=\{x_1,x_2,x_3\}\), and
\(V''(x_j)>0\). Direct calculation gives
\begin{equation}
S_{12}=S_{23}=S=\frac1{4\pi},\qquad
S_{13}=2S,\qquad S_\partial\ge\frac14>S.
\label{eq:counterexample-actions}
\end{equation}

Moreover, \(S/2<r<S\), with \(r\) arbitrarily close to \(S\), and
\begin{equation}
\Omega_j=(x_j-\ell,x_j+\ell),\qquad
V(x_j+y)=\sin^2(8\pi y)\quad (|y|\le\ell).
\label{eq:counterexample-translations}
\end{equation}
Since the local Dirichlet problems are translates of one another,
\begin{equation}
\mu_1^{(j)}(\varepsilon)=\mu_1^{(1)}(\varepsilon),
\qquad j=1,2,3.
\label{eq:counterexample-local-eigenvalues}
\end{equation}

Choose the cutoffs \(\vartheta_j\) in \eqref{eq:interaction-cutoffs}
as translates of one another. The normalized positive local eigenfunctions
are also translates. By the change of variables \(x\mapsto x+1/8\) in
\eqref{eq:interaction-matrix-definition}, \(w_{12}=w_{23}\), and hence
\eqref{eq:interaction-entry-asymptotics} gives
\(b_{12}^{\,0}=b_{23}^{\,0}>0\). Therefore \eqref{ass:tunneling} holds.
\par
\endgroup

{\color{black} The following
matrix lemma gives a Hopf crossing starting from a stable matrix at
\(\eta=0\), corresponding to the later choice
\(b_{\varepsilon,0}=b_0\). We will use this
finite-dimensional crossing when applying the standard bifurcation theorems
\cite{CrandallRabinowitz1977}.}
The vector \(p=\frac12(1,\sqrt2,1)^T\) is the unit positive vector spanning
\(\ker G_0\). {\color{black}In the reduction on \(E_{\varepsilon,3}^{N}\), it is the limiting coordinate vector of the
\(L^2(\Omega)\)-normalized first eigenfunction \(\phi_\varepsilon\) of
\(H_\varepsilon^N\) introduced in Section~\ref{sec:semiclassical-reduction}.}
Fix \(b_0>0\). For a vector path \(v(\eta)\in\mathbb R^3\) and a vector
\(w\in\mathbb R^3\), write
\begin{equation}
M(\eta):=G_0-v(\eta)w^T .
\label{eq:matrix-path}
\end{equation}
If \(v(\eta)>0\) and \(w>0\) componentwise, then every entry of
\(v(\eta)w^T\) is positive.

\Needspace{12\baselineskip}
\begin{lemma}
    \label{lem:matrix-choice}
    For every fixed \(b_0>0\), one can choose a positive diagonal matrix
    \(D\) and vectors \(w^{\rm s},w^{\rm u}>0\), with
    \(w_1^{\rm s}=w_1^{\rm u}\) and
    \((w^{\rm s})^Tp=(w^{\rm u})^Tp\), such that the equilibrium \(z=0\) of
    \begin{equation}
    \frac{\mathrm dz}{\mathrm ds}
    =(G_0-D_\eta p w^T)z-(w^Tz)D_\eta z,
    \qquad D_\eta=(1-\eta)b_0I+\eta D,
    \label{eq:matrix-nonlinear-family}
    \end{equation}
    is linearly stable at \(\eta=0\) and loses linear stability through
    a nondegenerate Hopf bifurcation at some \(\eta_0\in(0,1)\).
    The bifurcation is supercritical for \(w=w^{\rm s}\) and subcritical
    for \(w=w^{\rm u}\), with the same \(D\).
\end{lemma}

\begin{proof}
    We first take \(b_0=1\). Choose
    \begin{equation}
    \begin{gathered}
    D=\operatorname{diag}\!\left(\frac{51}{5000},\frac{297}{3125},
    \frac{4443}{100}\right),\qquad
    w^{\rm s}=\left(196,\frac{11}{4},4-\frac{11\sqrt2}{4}\right)^T,
    \qquad
    w^{\rm u}=\left(196,2,4-2\sqrt2\right)^T.
    \end{gathered}
    \label{eq:hopf-common-data}
    \end{equation}
    Both vectors are positive, and
    \begin{equation}
    w_1^{\rm s}=w_1^{\rm u}=196,\qquad
    (w^{\rm s})^Tp=(w^{\rm u})^Tp=100.
    \label{eq:hopf-common-mass}
    \end{equation}

    For either choice \(w\), the stability at \(\eta=0\) follows from
    the structure of \(G_0\). The vector \(p\) spans \(\ker G_0\), and
    \(p^\perp\) is invariant under \(G_0\). For \(y\in p^\perp\),
    \begin{equation*}
    (G_0-pw^T)y=G_0y-p(w^Ty),\qquad
    (G_0-pw^T)p=-(w^Tp)p.
    \end{equation*}
    Thus the matrix is block triangular in a basis whose first vector
    is \(p\) and whose remaining vectors span \(p^\perp\). Its eigenvalues
    are \(-100,-\sqrt2,-2\sqrt2\), so \(M(0)\) is stable.

    Write
    \begin{equation}
    M(\eta)=G_0-D_\eta p w^T,\qquad D_\eta=(1-\eta)I+\eta D.
    \label{eq:hopf-endpoint-path}
    \end{equation}
    Direct computation gives
    \begin{equation*}
    \begin{aligned}
    \det(\lambda I-M(\eta))
    =\lambda^3+a_1(\eta)\lambda^2+a_2(\eta)\lambda+a_3(\eta),
    \end{aligned}
    \label{eq:matrix-hurwitz-determinant}
    \end{equation*}
    where $a_1(\eta)=3\sqrt2+w^TD_\eta p,
    a_2(\eta)=4+w^T(G_0+3\sqrt2 I)D_\eta p,
    a_3(\eta)=4(w^Tp)(p^TD_\eta p).$
    These coefficients are positive on \([0,1]\). Since they are affine
    in \(\eta\), the equation
    \begin{equation}
    \Delta(\eta):=a_1(\eta)a_2(\eta)-a_3(\eta)=0
    \label{eq:hopf-endpoint-quadratic}
    \end{equation}
    is quadratic. Let \(\eta_{\rm s}\) and \(\eta_{\rm u}\) be its smaller
    positive roots for the two vectors in \eqref{eq:hopf-common-data}.
    Substituting these data into  \eqref{eq:hopf-endpoint-quadratic} gives
    \begin{equation}
    \begin{aligned}
    0.8217<\eta_{\rm s}<0.8218,&\qquad
    \Delta_{\rm s}'(\eta_{\rm s})<-26000,\\
    0.9951<\eta_{\rm u}<0.9952,&\qquad
    \Delta_{\rm u}'(\eta_{\rm u})<-23000.
    \end{aligned}
    \label{eq:hopf-critical-values}
    \end{equation}

    For either choice, write \(\eta_0=\eta_{\rm s}\) or \(\eta_{\rm u}\),
    respectively. At \(\eta=\eta_0\), the characteristic polynomial factors as
    \begin{equation*}
    \det(\lambda I-M(\eta_0))
    =(\lambda+a_1(\eta_0))(\lambda^2+a_2(\eta_0)).
    \end{equation*}
    Hence the eigenvalues are \(-a_1(\eta_0)<0\) and the simple pair
    \(\pm i\omega_0\), where \(\omega_0=\sqrt{a_2(\eta_0)}>0\).
    Implicit differentiation gives
    \begin{equation}
    \left.\frac{\mathrm d}{\mathrm d\eta}\operatorname{Re}\lambda_+(\eta)
    \right|_{\eta=\eta_0}
    =-\frac{\Delta'(\eta_0)}{2\bigl(a_1(\eta_0)^2+a_2(\eta_0)\bigr)}>0,
    \label{eq:hopf-endpoint-crossing-speed}
    \end{equation}
    by \eqref{eq:hopf-critical-values}.

    At this critical value, define
    \begin{equation}
    \begin{aligned}
    A(D_{\eta_0},w)&=G_0-D_{\eta_0}p\,w^T,\qquad
    \mathcal B_0(x,y)=-D_{\eta_0}\bigl[x(w^Ty)+y(w^Tx)\bigr].
    \end{aligned}
    \label{eq:hopf-endpoint-operators}
    \end{equation}
    Use \(\langle x,y\rangle=\overline x^Ty\) and choose
    \begin{equation*}
    \begin{gathered}
    A(D_{\eta_0},w)q=i\omega_0q,\qquad
    A(D_{\eta_0},w)^*p_H=-i\omega_0p_H,\qquad
    \|q\|_2=1,\qquad \langle p_H,q\rangle=1.
    \end{gathered}
    \label{eq:hopf-endpoint-normalization}
    \end{equation*}
    The matrices \(A(D_{\eta_0},w)\) and
    \(2i\omega_0I-A(D_{\eta_0},w)\) are invertible. The first Lyapunov coefficient
    is given by \cite[Section~5.4.2, formula~(5.62)]{Kuznetsov1998}:
    \begin{equation}
    \begin{aligned}
    \ell(D_{\eta_0},w)=\frac{1}{2\omega_0}\operatorname{Re}\bigl\langle p_H,
    &-2\mathcal B_0\bigl(q,A(D_{\eta_0},w)^{-1}\mathcal B_0(q,\overline q)\bigr)\\
    &+\mathcal B_0\bigl(\overline q,(2i\omega_0I-A(D_{\eta_0},w))^{-1}
    \mathcal B_0(q,q)\bigr)\bigr\rangle.
    \end{aligned}
    \label{eq:hopf-endpoint-lyapunov}
    \end{equation}
    Substitution of \eqref{eq:hopf-common-data} into
    \eqref{eq:hopf-endpoint-lyapunov} yields
    \begin{equation}
    \begin{aligned}
    -0.311&<\ell(D_{\eta_{\rm s}},w^{\rm s})<-0.310,\\
    0.064&<\ell(D_{\eta_{\rm u}},w^{\rm u})<0.065.
    \end{aligned}
    \label{eq:hopf-common-lyapunov-signs}
    \end{equation}
    The Hopf theorem, \eqref{eq:hopf-endpoint-crossing-speed}, and
    \eqref{eq:hopf-common-lyapunov-signs} give the two asserted bifurcations.

    For general \(b_0>0\), replace the common matrix \(D\) by \(b_0D\)
    and each vector \(w\) by \(b_0^{-1}w\). The products
    \(D_\eta p w^T\) and \((w^Tz)D_\eta z\) in
    \eqref{eq:matrix-nonlinear-family} remain unchanged. Both equalities
    between the two vectors in the statement are also preserved.
\end{proof}

{\color{black}To construct the coefficients, we fix the symmetric
effective profile \(q=q_*-V\) with three maximum points, choose positive
\(b_{\varepsilon,\eta}\) and \(c_\varepsilon\), set
\(u_\varepsilon^*=\phi_\varepsilon\), and then define}
\[
a_{\varepsilon,\eta}-S_\varepsilon^*b_{\varepsilon,\eta}
=q-\nu_\varepsilon,
\qquad\text{and}\qquad
S_\varepsilon^*=\int_\Omega c_\varepsilon(y)\phi_\varepsilon(y)\,\mathrm{d}y .
\]
{\color{black}The five profiles are shown in
Figure~\ref{fig:three-peak-coeff-profiles}.}

For the basis in Lemma~\ref{lem:three-well-reduction}, write
\(J_\varepsilon z=\sum_{j=1}^3z_j\zeta_{j,\varepsilon}\).

\begin{lemma}
\label{lem:positive-profiles}
Fix \(b_0>0\) and \(v^H,w\in(0,\infty)^3\), and set
\(M(\eta)=G_0-[(1-\eta)b_0p+\eta v^H]w^T\).
There exist \(\delta>0\) and smooth positive coefficients
\(b_{\varepsilon,\eta}=b_\eta\) and \(c_\varepsilon\), defined for
\(\eta\in I_\delta:=[-\delta,1+\delta]\), such that, for all sufficiently
small \(\varepsilon>0\),
\begin{equation}
J_\varepsilon^*\mathcal L_{\varepsilon,\eta}J_\varepsilon
=\tau_\varepsilon M(\eta)+o(\tau_\varepsilon)
\quad\text{in }C^1(I_\delta;\mathbb R^{3\times3}),
\label{eq:positive-profile-matrix-realization}
\end{equation}
where \(\mathcal L_{\varepsilon,\eta}
=A_\varepsilon-(b_\eta\phi_\varepsilon)\otimes c_\varepsilon\).
Here \(b_\eta\) is independent of \(\varepsilon\), affine in \(\eta\),
and satisfies \(b_{\varepsilon,0}=b_0\), while \(c_\varepsilon\)
is independent of \(\eta\) and satisfies
\(\langle c_\varepsilon,\phi_\varepsilon\rangle=\tau_\varepsilon w^Tp\).
\end{lemma}

\begin{proof}
We use the notation
\begin{equation}
\begin{gathered}
p=\tfrac12(1,\sqrt2,1)^T,\qquad
D=\operatorname{diag}(v_1^H/p_1,v_2^H/p_2,v_3^H/p_3),\\
D_\eta=(1-\eta)b_0I+\eta D,\qquad M(\eta)=G_0-D_\eta p\,w^T.
\end{gathered}
\label{eq:positive-profile-data}
\end{equation}
and the associated projections
\begin{equation}
J_\varepsilon z=\sum_{j=1}^3z_j\zeta_{j,\varepsilon},\qquad
J_\varepsilon^*J_\varepsilon=I,\qquad
P_\varepsilon=J_\varepsilon J_\varepsilon^*=\Pi_\varepsilon^N,\qquad
Q_\varepsilon=I-P_\varepsilon.
\label{eq:positive-profile-projections}
\end{equation}
All constants below are independent of \(\varepsilon\) and
\(\eta\in I_\delta\), but may depend on \(b_0,v^H,w\) and the
fixed local neighborhoods.

Let \(U_1,U_2,U_3\Subset\Omega\) be pairwise disjoint open intervals
such that \(x_j\in U_j\Subset\Omega_j\),
\(U_j\cap\overline\Omega_k=\varnothing\) for \(k\ne j\), and
\(\chi=1\) on \(U_j\), where \(\chi\) is the cutoff fixed in
\eqref{eq:nd-cutoff}. This choice is possible because \(r<S\).
Let \(\widehat\xi_{j,\varepsilon}\) be the extension by zero to
\(\Omega_0\) of the local positive first eigenfunction
\(\xi_{j,\varepsilon}\) on \(\Omega_j\). Formula
\eqref{eq:transfer-isometry} and the basis construction in
Lemma~\ref{lem:three-well-reduction} give
\begin{equation}
\|\zeta_{j,\varepsilon}-J_\chi\widehat\xi_{j,\varepsilon}\|_{L^2(\Omega)}
\le C_{\rm B}e^{-\kappa_{\rm B}/\varepsilon},\qquad j=1,2,3.
\label{eq:positive-profile-basis-comparison}
\end{equation}
Choose \(\widehat\eta_j\in C_c^\infty(U_j)\), nonnegative and equal
to one near \(x_j\). The local harmonic approximation gives
\begin{equation*}
m_{j,\varepsilon}:=
\int_{U_j}\widehat\eta_j\widehat\xi_{j,\varepsilon}\,\mathrm dx
\ge c_{\rm M}\varepsilon^{1/4}.
\end{equation*}
By \eqref{eq:positive-profile-basis-comparison} and the disjoint supports,
\begin{equation*}
\left|\int_\Omega\widehat\eta_i\zeta_{j,\varepsilon}\,\mathrm dx
-\delta_{ij}m_{i,\varepsilon}\right|
\le C_{\rm J}e^{-\kappa_{\rm B}/\varepsilon},\qquad 1\le i,j\le3.
\end{equation*}
These two estimates imply that the denominator below
is at least \(\frac12c_{\rm M}\varepsilon^{1/4}>0\) for small
\(\varepsilon\). Define
\begin{equation*}
\eta_{j,\varepsilon}
=\frac{\widehat\eta_j}
{\displaystyle\int_\Omega\widehat\eta_j\zeta_{j,\varepsilon}\,\mathrm dx}.
\end{equation*}
The functions \(\eta_{j,\varepsilon}\) satisfy
\begin{equation}
\begin{gathered}
\left|\langle\eta_{i,\varepsilon},\zeta_{j,\varepsilon}\rangle
-\delta_{ij}\right|
\le\frac{2C_{\rm J}}{c_{\rm M}}
\varepsilon^{-1/4}e^{-\kappa_{\rm B}/\varepsilon}=o(1),\quad
\|\eta_{j,\varepsilon}\|_{L^2(\Omega)}\le C_\eta\varepsilon^{-1/4},
\qquad 1\le i,j\le3.
\end{gathered}
\label{eq:positive-profile-cutoff-estimates}
\end{equation}
Recall that \(\phi_\varepsilon\) is the normalized positive first Neumann
eigenfunction, with \(A_\varepsilon\phi_\varepsilon=0\).
Since \(\phi_\varepsilon\in E_{\varepsilon,3}^{N}\), the isometry
\eqref{eq:positive-profile-projections} gives
\begin{equation}
\phi_\varepsilon=J_\varepsilon\rho_\varepsilon,
\qquad\rho_\varepsilon=(\rho_{1,\varepsilon},
\rho_{2,\varepsilon},\rho_{3,\varepsilon})^T,
\qquad \|\rho_\varepsilon\|_2=1.
\label{eq:positive-profile-ground-coordinates}
\end{equation}
 Lemma~\ref{lem:three-well-reduction} implies that
\begin{equation*}
0=J_\varepsilon^*A_\varepsilon\phi_\varepsilon
=\tau_\varepsilon(G_0+o(1))\rho_\varepsilon.
\end{equation*}
{\color{black}It follows that every limit point belongs to}
\(\ker G_0=\operatorname{span}\{(1,\sqrt2,1)^T\}\). The signs of
\(\zeta_{j,\varepsilon}\) are consistent with \(\phi_\varepsilon>0\);
hence \(\rho_\varepsilon\to p\) as \(\varepsilon\to0^+\).

By \eqref{eq:positive-profile-cutoff-estimates} and
\eqref{eq:positive-profile-ground-coordinates},
\begin{equation*}
\langle\eta_{j,\varepsilon},\phi_\varepsilon\rangle\longrightarrow p_j>0
\qquad(\varepsilon\to0^+),\qquad j=1,2,3.
\end{equation*}
For smooth \(h_0>0\) on \(\overline\Omega\), choose
\(\delta_\varepsilon>0\) with \(\delta_\varepsilon=o(\tau_\varepsilon)\)
as \(\varepsilon\to0^+\), and put
\begin{equation}
w_{3,\varepsilon}:=
\frac{w^Tp-\displaystyle\sum_{j=1}^2w_j
\langle\eta_{j,\varepsilon},\phi_\varepsilon\rangle
-(\delta_\varepsilon/\tau_\varepsilon)\langle h_0,\phi_\varepsilon\rangle}
{\langle\eta_{3,\varepsilon},\phi_\varepsilon\rangle}
\longrightarrow w_3>0\qquad(\varepsilon\to0^+).
\label{eq:positive-profile-third-weight}
\end{equation}
Define
\begin{equation}
c_\varepsilon
=\tau_\varepsilon\bigl(w_1\eta_{1,\varepsilon}
+w_2\eta_{2,\varepsilon}+w_{3,\varepsilon}\eta_{3,\varepsilon}\bigr)
+\delta_\varepsilon h_0.
\label{eq:positive-profile-observation-weight}
\end{equation}
Then \(c_\varepsilon>0\), and \eqref{eq:positive-profile-third-weight} gives
\begin{equation}
\langle c_\varepsilon,\phi_\varepsilon\rangle
=\tau_\varepsilon w^Tp.
\label{eq:positive-profile-exact-mass}
\end{equation}
By \eqref{eq:positive-profile-cutoff-estimates} and
\eqref{eq:positive-profile-third-weight}, we have
\begin{equation}
\langle c_\varepsilon,\zeta_{j,\varepsilon}\rangle
=\tau_\varepsilon w_j+o(\tau_\varepsilon),\qquad
\|c_\varepsilon\|_{L^2(\Omega)}\le C_c\tau_\varepsilon\varepsilon^{-1/4}.
\label{eq:positive-profile-observation-estimates}
\end{equation}
Set $\widehat c_\varepsilon:=c_\varepsilon/\tau_\varepsilon$.
Since \((J_\varepsilon^*c_\varepsilon)_j
=\langle c_\varepsilon,\zeta_{j,\varepsilon}\rangle\) by
\eqref{eq:positive-profile-projections}, one has, as \(\varepsilon\to0^+\),
\begin{equation}
J_\varepsilon^*\widehat c_\varepsilon\longrightarrow w.
\label{eq:hopf-weight-limit}
\end{equation}

With \(\beta_j=v_j^H/p_j\) as in \eqref{eq:positive-profile-data},
choose smooth cutoffs
\(0\le\theta_j\le1\), equal to one on \(U_j\), with pairwise disjoint
supports in \(\Omega_j\). Define
\begin{equation}
\begin{aligned}
b_H=b_0+\sum_{j=1}^3(\beta_j-b_0)\theta_j,\qquad
b_{\varepsilon,\eta}=b_\eta=(1-\eta)b_0+\eta b_H.
\end{aligned}
\label{eq:hopf-bounded-realization}
\end{equation}
The function \(b_H\)  lies between
the positive minimum and maximum of \(b_0,\beta_1,\beta_2,\beta_3\).
{\color{black}For \(\delta>0\) sufficiently small, a positive lower bound also holds
on \(I_\delta\).} Since \(\partial_\eta b_\eta=b_H-b_0\), one has
\begin{equation}
b_{\varepsilon,0}=b_0,\qquad
0<b_-\le b_\eta\le b_+\le C_b,\qquad
\|\partial_\eta b_\eta\|_{L^\infty(\Omega)}\le C_b.
\label{eq:hopf-multiplier-bounds}
\end{equation}

The local Dirichlet decay estimate in
\cite[Proposition~3.3.5]{Helffer1988}, together with
\eqref{eq:positive-profile-basis-comparison}, gives
\begin{equation}
\|\zeta_{j,\varepsilon}\|_{L^2(\Omega\setminus U_j)}
\le Ce^{-\kappa/\varepsilon},\qquad j=1,2,3.
\label{eq:positive-profile-basis-localization}
\end{equation}
On \(U_j\), the multiplier \(b_\eta\) equals the \(j\)-th diagonal
entry of \(D_\eta\), and the same holds for their \(\eta\)-derivatives.
By \eqref{eq:positive-profile-basis-localization}, we obtain
\begin{equation}
\sup_{\eta\in I_\delta}\max_{k=0,1}
\bigl\|\partial_\eta^k(b_\eta J_\varepsilon-J_\varepsilon D_\eta)\bigr\|
\le Ce^{-\kappa/\varepsilon}.
\label{eq:hopf-multiplier-localization}
\end{equation}
Together with \eqref{eq:positive-profile-ground-coordinates} and
\(\rho_\varepsilon\to p\) as \(\varepsilon\to0^+\), one has 
\begin{equation}
J_\varepsilon^*(b_\eta\phi_\varepsilon)
=D_\eta\rho_\varepsilon+O(e^{-\kappa/\varepsilon})
\longrightarrow D_\eta p
=(1-\eta)b_0p+\eta v^H
\quad\hbox{in }C^1(I_\delta;\mathbb R^3)\quad(\varepsilon\to0^+).
\label{eq:positive-profile-multiplier-projection}
\end{equation}
In particular, \eqref{eq:positive-profile-multiplier-projection} at
\(\eta=1\) gives \(J_\varepsilon^*(b_H\phi_\varepsilon)\to v^H\) as \(\varepsilon\to0^+\).
By \eqref{eq:hopf-multiplier-bounds},
\[
\max_{k=0,1}\|\partial_\eta^k(b_\eta\phi_\varepsilon)\|_{L^2(\Omega)}\le C_b.
\]
For \(f,g\in L^2(\Omega)\), the operator norm satisfies
\[
\|f\otimes g\|=\|f\|_{L^2(\Omega)}\|g\|_{L^2(\Omega)}.
\]
 Equation~\eqref{eq:positive-profile-observation-estimates} gives
\begin{equation}
\sup_{\eta\in I_\delta}\max_{k=0,1}
\bigl\|\partial_\eta^k
\bigl((b_\eta\phi_\varepsilon)\otimes c_\varepsilon\bigr)\bigr\|
\le C_bC_c\tau_\varepsilon\varepsilon^{-1/4}.
\label{eq:hopf-bounded-rank-one}
\end{equation}

Finally, \eqref{eq:hopf-weight-limit} and
\eqref{eq:positive-profile-multiplier-projection} imply
\begin{equation*}
\begin{aligned}
J_\varepsilon^*\bigl((b_\eta\phi_\varepsilon)\otimes c_\varepsilon\bigr)
J_\varepsilon
&=\bigl[J_\varepsilon^*(b_\eta\phi_\varepsilon)\bigr]
\bigl[J_\varepsilon^*c_\varepsilon\bigr]^T=\tau_\varepsilon D_\eta p\,w^T+o(\tau_\varepsilon)
\quad\hbox{in }C^1(I_\delta;\mathbb R^{3\times3}).
\end{aligned}
\end{equation*}
Lemma~\ref{lem:three-well-reduction} also gives
\(J_\varepsilon^*A_\varepsilon J_\varepsilon
=\tau_\varepsilon G_0+o(\tau_\varepsilon)\).
{\color{black}Subtracting the preceding identity for the nonlocal term, we obtain
\eqref{eq:positive-profile-matrix-realization}.}
\end{proof}

{\color{black}We now use Lemmas~\ref{lem:matrix-choice} and~\ref{lem:positive-profiles}
to prove the Hopf bifurcation result for the equation.}

\begin{proof}[Proof of Theorem~\ref{thm:three-peak-hopf}]
    \setcounter{proofstep}{0}

    Fix \(V\) as in \eqref{eq:counterexample-potential};
    the verification above shows that \eqref{ass:tunneling} holds.
    Use the common matrix \(D\) in \eqref{eq:hopf-common-data} and either
    \(w=w^{\rm s}\) or \(w=w^{\rm u}\), with \(b_0=1\).
    Let \(\eta_0\in(0,1)\) be the corresponding Hopf value from
    Lemma~\ref{lem:matrix-choice}. Write \(M(\eta)=G_0-D_\eta p\,w^T\), as in
    \eqref{eq:hopf-endpoint-path}, and use the coefficients constructed
    in \eqref{eq:positive-profile-observation-weight} and
    \eqref{eq:hopf-bounded-realization}, with \(v^H=Dp\) and
    \(b_{\varepsilon,\eta}=b_\eta\). {\color{black}We use
    \(J_\varepsilon,P_\varepsilon,Q_\varepsilon\) from
    \eqref{eq:positive-profile-projections}.}
    
    \medskip
\refstepcounter{proofstep}\label{step:hopf-crossing-realization}\label{step:hopf-crossing-resolvent}
\noindent\textbf{Step \theproofstep.}
    Let \(b_{\varepsilon,\eta}\) and \(c_\varepsilon\) be given by
    Lemma~\ref{lem:positive-profiles}. Put
    \begin{equation*}
    S_\varepsilon^*
    =
    \int_\Omega c_\varepsilon(y)\phi_\varepsilon(y)\,\mathrm{d}y
    =100\tau_\varepsilon,
    \qquad
    a_{\varepsilon,\eta}(x)
    =
    q(x)-\nu_\varepsilon+b_{\varepsilon,\eta}(x)S_\varepsilon^*.
    \end{equation*}
    Here \(S_\varepsilon^*\) is the same for both choices by
    \eqref{eq:hopf-common-mass} and \eqref{eq:positive-profile-exact-mass}.
    Hence \(a_{\varepsilon,\eta}\) and \(b_{\varepsilon,\eta}\) are also the same.
    Use the same cutoffs, \(h_0\), and \(\delta_\varepsilon\) in
    \eqref{eq:positive-profile-observation-weight}. Since
    \(w_1^{\rm s}=w_1^{\rm u}\), the two choices of \(c_\varepsilon\)
    agree outside \(U_2\cup U_3\), in particular on \(U_1\).
    Therefore,
    \begin{equation*}
    \varepsilon^2\phi_\varepsilon''
    +
    \phi_\varepsilon
    \left(
    a_{\varepsilon,\eta}(x)-b_{\varepsilon,\eta}(x)S_\varepsilon^*
    \right)
    =
    \varepsilon^2\phi_\varepsilon''
    +
    (q-\nu_\varepsilon)\phi_\varepsilon
    =
    0.
    \end{equation*}
 Thus,
    \(u_\varepsilon^*=\phi_\varepsilon\) is a positive steady
    state for every \(\eta\).
     The linearized operator at \(u_\varepsilon^*=\phi_\varepsilon\) is
    \begin{equation}
    \mathcal L_{\varepsilon,\eta}\xi
    =
    A_\varepsilon\xi
    -
    b_{\varepsilon,\eta}\phi_\varepsilon
    \int_\Omega c_\varepsilon(y)\xi(y)\,\mathrm{d}y .
    \label{eq:hopf-equilibrium-linearization}
    \end{equation}
    
    \medskip
    We identify \(E_{\varepsilon,3}^{N}\) with \(\mathbb C^3\) through
    \(J_\varepsilon\) in \eqref{eq:positive-profile-projections}.
    Formula~\eqref{eq:positive-profile-matrix-realization} gives
    \begin{equation*}
    J_\varepsilon^*\mathcal L_{\varepsilon,\eta}J_\varepsilon
    =
    \tau_\varepsilon M(\eta)+o(\tau_\varepsilon)
    \quad\text{in }C^1(I_\delta;\mathbb R^{3\times3}).
    \end{equation*} 
    On \(L^2(\Omega)=E_{\varepsilon,3}^{N}\oplus(E_{\varepsilon,3}^{N})^\perp\), write
    \begin{equation}
    \begin{gathered}
    \mathcal L_{\varepsilon,\eta}
    =
    \begin{pmatrix}
    L_{11}^{\varepsilon,\eta} & L_{12}^{\varepsilon,\eta}\\
    L_{21}^{\varepsilon,\eta} & L_{22}^{\varepsilon,\eta}
    \end{pmatrix}.
    \end{gathered}
    \label{eq:hopf-block-matrix}
    \end{equation}
    Since \(E_{\varepsilon,3}^{N}\) is a spectral subspace of \(A_\varepsilon\),
    \begin{equation*}
    \Pi_\varepsilon^NA_\varepsilon(I-\Pi_\varepsilon^N)=0,\qquad
    \text{and}\qquad
    (I-\Pi_\varepsilon^N)A_\varepsilon\Pi_\varepsilon^N=0 .
    \end{equation*}
    Hence, for any \(0<\varepsilon<1\),
    \eqref{eq:hopf-bounded-rank-one} gives
    \begin{equation}
    \label{eq:off-diagonal-block-bound}
    \|L_{12}^{\varepsilon,\eta}\|+\|L_{21}^{\varepsilon,\eta}\|
    \le
    2C_bC_c\tau_\varepsilon\varepsilon^{-1/2}.
    \end{equation}
    On \((E_{\varepsilon,3}^{N})^\perp\), Lemma~\ref{lem:three-well-reduction} gives
    \begin{equation}
    \sigma(A_\varepsilon|_{(E_{\varepsilon,3}^{N})^\perp})
    \subset(-\infty,-\gamma_0\varepsilon].
    \label{eq:hopf-complement-gap}
    \end{equation}
    Moreover, \eqref{eq:hopf-bounded-rank-one} gives
    \begin{equation*}
    \|(I-\Pi_\varepsilon^N)
    [(b_{\varepsilon,\eta}\phi_\varepsilon)\otimes c_\varepsilon]
    (I-\Pi_\varepsilon^N)\|
    \le
    C_bC_c\tau_\varepsilon\varepsilon^{-1/2}
    =o(\varepsilon).
    \end{equation*}
    Thus, for all sufficiently small \(\varepsilon>0\), one has
    \begin{equation*}
    C_bC_c\tau_\varepsilon\varepsilon^{-1/2}
    \le \frac{\gamma_0}{4}\varepsilon.
    \end{equation*}
    For \(\operatorname{Re}\lambda\ge-\gamma_0\varepsilon/2\),
    \eqref{eq:hopf-complement-gap} and self-adjointness give
    \(\|(\lambda-A_\varepsilon|_{(E_{\varepsilon,3}^{N})^\perp})^{-1}\|
    \le2/(\gamma_0\varepsilon)\).
   Then one has
    \begin{equation}
    \label{eq:complement-resolvent-bound}
    \|(\lambda-L_{22}^{\varepsilon,\eta})^{-1}\|
    \le
    \frac{4}{\gamma_0\varepsilon},
    \qquad \operatorname{Re}\lambda\ge-\gamma_0\varepsilon/2,
    \end{equation}
    uniformly for \(\eta\) near \(\eta_0\). Also,
  $
    \tau_\varepsilon\varepsilon^{-N}=o(\varepsilon)~
   \hbox{for every fixed }N>0.
    $
    
    \medskip
\refstepcounter{proofstep}\label{step:hopf-crossing-crossing}\label{step:hopf-crossing-bifurcation}
\noindent\textbf{Step \theproofstep.}
    Let \(\lambda_\pm(\eta)\) be the two eigenvalue branches of \(M(\eta)\)
    which satisfy
    \begin{equation*}
    \lambda_+(\eta_0)=i\omega_0,\qquad
    \lambda_-(\eta_0)=-i\omega_0,\qquad\text{and}\qquad
    \omega_0>0,
    \end{equation*}
    and, by Lemma~\ref{lem:matrix-choice},
    \begin{equation}
    \left.\frac{\mathrm{d}}{\mathrm{d}\eta}
    \operatorname{Re}\lambda_+(\eta)\right|_{\eta=\eta_0}>0 .
    \label{eq:hopf-reduced-speed}
    \end{equation}
    Choose \(r>0\) small and, for \(\eta\) near \(\eta_0\), define the circles
    \begin{equation*}
    \Gamma_{\varepsilon,\pm}(\eta)
    =
    \{\lambda:\ |\lambda-\tau_\varepsilon\lambda_\pm(\eta)|
    =r\tau_\varepsilon\}
    \end{equation*}
    Each circle \(\Gamma_{\varepsilon,\pm}(\eta)\) encloses only
    the corresponding point of \(\tau_\varepsilon\sigma(M(\eta))\). For
    $\operatorname{Re}\lambda\ge-\gamma_0\varepsilon/2$,
    \eqref{eq:off-diagonal-block-bound} and
    \eqref{eq:complement-resolvent-bound} give
    \begin{equation}
    \left\|
    L_{12}^{\varepsilon,\eta}
    (\lambda-L_{22}^{\varepsilon,\eta})^{-1}
    L_{21}^{\varepsilon,\eta}
    \right\|
    \le
    \frac{4C_b^2C_c^2}{\gamma_0}
    \tau_\varepsilon^2\varepsilon^{-2}
    =o(\tau_\varepsilon).
    \label{eq:hopf-schur-bound}
    \end{equation}
    Applying the Schur complement reduction to \eqref{eq:hopf-block-matrix}
    \cite[Theorem~1.2]{DussonSigalStamm2021}, for
    \(z=x+y\), \(x\in E_{\varepsilon,3}^{N}\) and
    \(y\in(E_{\varepsilon,3}^{N})^\perp\),
    \((\mathcal L_{\varepsilon,\eta}-\lambda)z=0\) is equivalent to
    \begin{equation}
    y=(\lambda-L_{22}^{\varepsilon,\eta})^{-1}
    L_{21}^{\varepsilon,\eta}x
    \label{eq:hopf-high-component}
    \end{equation}
    and
    \begin{equation}
    \begin{gathered}
    (\lambda I-\mathcal S_\varepsilon(\lambda,\eta))x=0,\quad
    \mathcal S_\varepsilon(\lambda,\eta)
    =
    L_{11}^{\varepsilon,\eta}
    +
    L_{12}^{\varepsilon,\eta}
    (\lambda-L_{22}^{\varepsilon,\eta})^{-1}
    L_{21}^{\varepsilon,\eta}.
    \end{gathered}
    \label{eq:hopf-schur-reduction}
    \end{equation}
    By
    \eqref{eq:positive-profile-matrix-realization} and \eqref{eq:hopf-schur-bound},
    uniformly on the half-plane in \eqref{eq:complement-resolvent-bound}, one has
    \begin{equation}
    \mathcal S_\varepsilon(\lambda,\eta)
    =
    \tau_\varepsilon M(\eta)+o(\tau_\varepsilon).
    \label{eq:hopf-schur-limit}
    \end{equation}
    Set
    \begin{equation}
    D_\varepsilon(\lambda,\eta)
    =
    \det(\lambda I-\mathcal S_\varepsilon(\lambda,\eta)),
    \qquad\text{and}\qquad
    D_\varepsilon^0(\lambda,\eta)
    =
    \det(\lambda I-\tau_\varepsilon M(\eta)).
    \label{eq:hopf-characteristic-determinants}
    \end{equation}
    {\color{black}Applying Rouch\'e's theorem to the determinants in
    \eqref{eq:hopf-characteristic-determinants}, using
    \eqref{eq:hopf-schur-limit} and the scaling \(\lambda=\tau_\varepsilon z\),
    we obtain exactly one zero of \(D_\varepsilon\) inside each
    \(\Gamma_{\varepsilon,\pm}(\eta)\).} Thus the corresponding eigenvalues of
    \(\mathcal L_{\varepsilon,\eta}\) satisfy
    \begin{equation}
    \Lambda_{\varepsilon,\pm}(\eta)
    =
    \tau_\varepsilon\lambda_\pm(\eta)+o(\tau_\varepsilon).
    \label{eq:hopf-eigenvalue-lifting}
    \end{equation}
    Since \(A_\varepsilon,c_\varepsilon\) are independent of \(\eta\)
    and \(b_{\varepsilon,\eta}\) is affine in \(\eta\),
    \eqref{eq:positive-profile-matrix-realization}, \eqref{eq:hopf-schur-bound}, and
    \eqref{eq:hopf-schur-limit} also hold after differentiation in \(\eta\).
    The implicit function theorem for the simple roots in
    \eqref{eq:hopf-characteristic-determinants} gives
    \begin{equation}
    \frac{\mathrm{d}}{\mathrm{d}\eta}
    \operatorname{Re}\Lambda_{\varepsilon,+}(\eta)
    =
    \tau_\varepsilon
    \frac{\mathrm{d}}{\mathrm{d}\eta}
    \operatorname{Re}\lambda_+(\eta)
    +
    o(\tau_\varepsilon).
    \label{eq:hopf-eigenvalue-derivative}
    \end{equation}
    By \eqref{eq:hopf-reduced-speed} and \eqref{eq:hopf-eigenvalue-derivative},
    the implicit function theorem gives \(\eta_\varepsilon=\eta_0+o(1)\) as \(\varepsilon\to0^+\) such that
    \begin{equation}
    \begin{gathered}
    \operatorname{Re}\Lambda_{\varepsilon,+}(\eta_\varepsilon)=0,\quad
    \omega_\varepsilon:=\operatorname{Im}\Lambda_{\varepsilon,+}(\eta_\varepsilon)
    =\tau_\varepsilon\omega_0+o(\tau_\varepsilon)>0,\\
    \alpha_\varepsilon:=
    \left.\frac{\mathrm{d}}{\mathrm{d}\eta}
    \operatorname{Re}\Lambda_{\varepsilon,+}(\eta)
    \right|_{\eta=\eta_\varepsilon}>0 .
    \end{gathered}
    \label{eq:hopf-critical-parameter}
    \end{equation}
    For \(\eta\) near \(\eta_0\), let \(\lambda_3(\eta)<0\) be the third
    eigenvalue of \(M(\eta)\). By \eqref{eq:hopf-schur-limit}, on the
    half-plane in \eqref{eq:complement-resolvent-bound}, for fixed \(R>C\),
    \begin{equation*}
    \|\mathcal S_\varepsilon(\lambda,\eta)\|\le C\tau_\varepsilon,
    \qquad
    |\lambda|>R\tau_\varepsilon
    \ \Longrightarrow\ D_\varepsilon(\lambda,\eta)\ne0.
    \end{equation*}
    Rouch\'e's theorem on \(|\lambda|=R\tau_\varepsilon\) and a small circle
    about \(\tau_\varepsilon\lambda_3(\eta)\) gives
    exactly three simple eigenvalues in this half-plane:
    \begin{equation*}
    \Lambda_{\varepsilon,\pm}(\eta),\qquad
    \tau_\varepsilon\lambda_3(\eta)+o(\tau_\varepsilon)<0
    \quad(\varepsilon\to0^+).
    \end{equation*}
    All remaining eigenvalues satisfy
    \(\operatorname{Re}\lambda<-\gamma_0\varepsilon/2\).
    Since \(\eta_\varepsilon\to\eta_0\in(0,1)\) as \(\varepsilon\to0^+\), we have
    \(\eta_\varepsilon\in I_\delta\) for small \(\varepsilon\), so
    \eqref{eq:hopf-multiplier-bounds} applies.

    By \eqref{eq:hopf-critical-parameter} and stability of the
    noncritical spectrum, there is
    \(\delta_\varepsilon>0\) such that
    \begin{equation*}
    \begin{aligned}
    \operatorname{Re}\Lambda_{\varepsilon,+}(\eta)&<0
    \quad\text{for }\eta\in(\eta_\varepsilon-\delta_\varepsilon,
    \eta_\varepsilon),\quad
    \operatorname{Re}\Lambda_{\varepsilon,+}(\eta)>0
    \quad\text{for }\eta\in(\eta_\varepsilon,
    \eta_\varepsilon+\delta_\varepsilon).
    \end{aligned}
    \end{equation*}
  Thus, \(u_\varepsilon^*\) is linearly stable on the
    left-hand interval and linearly unstable on the right-hand interval.
    This proves part~\textup{(\ref{item:hopf-crossing})}.

\medskip
\refstepcounter{proofstep}\label{step:hopf-branches-coefficient}\label{step:hopf-branches-convergence}
\noindent\textbf{Step \theproofstep.}
For \(v=u-\phi_\varepsilon\), the shifted equation is
\begin{equation}
\begin{aligned}
\partial_t v&=\mathcal L_{\varepsilon,\eta}v
+\tfrac12\mathcal B_{\varepsilon,\eta}(v,v),\quad
\mathcal B_{\varepsilon,\eta}(\xi,\zeta)
=-b_\eta\left[\xi\int_\Omega c_\varepsilon\zeta+
\zeta\int_\Omega c_\varepsilon\xi\right].
\end{aligned}
\label{eq:hopf-full-bilinear}
\end{equation}
For fixed \(\varepsilon\), the bilinear map in \eqref{eq:hopf-full-bilinear}
is bounded on complex \(L^2(\Omega)\), with complex-linear integrals.
\par\Needspace{5\baselineskip}
On \(X=C(\overline\Omega;\mathbb R)\), let
\(\mathcal A_\varepsilon v=\varepsilon^2v''+(q-\nu_\varepsilon)v\), with
\(D(\mathcal A_\varepsilon)=\{v\in C^2(\overline\Omega):v'=0
\text{ on }\partial\Omega\}\).
The operator \(\mathcal A_\varepsilon\) is sectorial with compact resolvent,
and the integral term in \eqref{eq:hopf-equilibrium-linearization} is
bounded on \(X\). Elliptic regularity identifies the spectrum of
\(\mathcal L_{\varepsilon,\eta}\) in \eqref{eq:hopf-equilibrium-linearization},
including algebraic multiplicities, on \(X\) and \(L^2(\Omega)\).
The map \(\tfrac12\mathcal B_{\varepsilon,\eta}(v,v)\) in
\eqref{eq:hopf-full-bilinear} is \(C^\infty\) in \((v,\eta)\).
Using inner products conjugate-linear in the first variable, choose
\begin{equation*}
\begin{gathered}
\mathcal L_{\varepsilon,\eta_\varepsilon} q_\varepsilon=i\omega_\varepsilon q_\varepsilon,
\quad \mathcal L_{\varepsilon,\eta_\varepsilon}^*p_\varepsilon=-i\omega_\varepsilon p_\varepsilon,\quad
\|q_\varepsilon\|_{L^2(\Omega)}=1,\quad
\langle p_\varepsilon,q_\varepsilon\rangle=1.
\end{gathered}
\label{eq:hopf-full-normalization}
\end{equation*}
The center-manifold reduction for \eqref{eq:hopf-full-bilinear}
\cite[Sections~6.1 and~6.4]{Henry1981} gives the quadratic coefficients
\(h_{11,\varepsilon}\) and \(h_{20,\varepsilon}\):
\begin{equation}
\begin{aligned}
h_{11,\varepsilon}&=-\mathcal L_{\varepsilon,\eta_\varepsilon}^{-1}
\mathcal B_{\varepsilon,\eta_\varepsilon}(q_\varepsilon,\overline q_\varepsilon),\quad
h_{20,\varepsilon}=(2i\omega_\varepsilon I-\mathcal L_{\varepsilon,\eta_\varepsilon})^{-1}
\mathcal B_{\varepsilon,\eta_\varepsilon}(q_\varepsilon,q_\varepsilon).
\end{aligned}
\label{eq:hopf-full-homological}
\end{equation}
Since \eqref{eq:hopf-full-bilinear} is quadratic, applying the formula in
\cite[Section~5.4.2, formula~(5.62)]{Kuznetsov1998} on the center manifold
gives the cubic coefficient \(G_{21,\varepsilon}\) and the first Lyapunov
coefficient \(\ell_\varepsilon\):
\begin{equation*}
\begin{aligned}
\ell_\varepsilon=\frac{\operatorname{Re}G_{21,\varepsilon}}{2\omega_\varepsilon},~ G_{21,\varepsilon}
=\left\langle p_\varepsilon,
2\mathcal B_{\varepsilon,\eta_\varepsilon}(q_\varepsilon,h_{11,\varepsilon})
+\mathcal B_{\varepsilon,\eta_\varepsilon}(\overline q_\varepsilon,h_{20,\varepsilon})
\right\rangle.
\end{aligned}
\end{equation*}

\medskip
Set
\begin{equation*}
\widehat{\mathcal B}_\varepsilon=
\mathcal B_{\varepsilon,\eta_\varepsilon}/\tau_\varepsilon.
\end{equation*}
Then 
\begin{equation}
\ell_\varepsilon
=\frac{\tau_\varepsilon}{2\omega_\varepsilon}\operatorname{Re}\left\langle p_\varepsilon,
2\widehat{\mathcal B}_\varepsilon(q_\varepsilon,h_{11,\varepsilon})
+\widehat{\mathcal B}_\varepsilon(\overline q_\varepsilon,h_{20,\varepsilon})
\right\rangle.
\label{eq:hopf-scaled-lyapunov}
\end{equation}
Estimates \eqref{eq:hopf-multiplier-bounds} and
\eqref{eq:hopf-weight-limit} give
\(\|\widehat{\mathcal B}_\varepsilon\|\le C\varepsilon^{-1/4}\).
By \eqref{eq:hopf-weight-limit},
\(J_\varepsilon^*\widehat c_\varepsilon\to w\) as \(\varepsilon\to0^+\).
By \eqref{eq:hopf-multiplier-localization}, we have
\(b_{\eta_\varepsilon}J_\varepsilon
=J_\varepsilon D_{\eta_\varepsilon}+O(e^{-\kappa/\varepsilon})\),
where \(D_{\eta_\varepsilon}\to D_{\eta_0}\) as \(\varepsilon\to0^+\).
It also gives \(Q_\varepsilon b_{\eta_\varepsilon}\phi_\varepsilon
=O(e^{-\kappa/\varepsilon})\).
Substituting into \eqref{eq:hopf-full-bilinear} and using the definition
of \(\mathcal B_0\) in \eqref{eq:hopf-endpoint-operators}, we obtain, as
\(\varepsilon\to0^+\), uniformly for bounded \(x,y\in\mathbb C^3\),
\begin{equation}
\begin{aligned}
J_\varepsilon^*\widehat{\mathcal B}_\varepsilon
(J_\varepsilon x,J_\varepsilon y)&\longrightarrow\mathcal B_0(x,y),\quad
\|Q_\varepsilon\widehat{\mathcal B}_\varepsilon
(J_\varepsilon x,J_\varepsilon y)\|_{L^2(\Omega)}
\le Ce^{-\kappa/\varepsilon}\|x\|_2\|y\|_2.
\end{aligned}
\label{eq:hopf-bilinear-projection}
\end{equation}

Recall \(P_\varepsilon,Q_\varepsilon\) from
\eqref{eq:positive-profile-projections}. Put
\begin{equation*}
\widetilde p_\varepsilon=\frac{p_\varepsilon}{\|p_\varepsilon\|_{L^2(\Omega)}},
\qquad
\|\widetilde p_\varepsilon\|_{L^2(\Omega)}=\|q_\varepsilon\|_{L^2(\Omega)}=1,
\qquad
\langle\widetilde p_\varepsilon,q_\varepsilon\rangle
=\frac1{\|p_\varepsilon\|_{L^2(\Omega)}}>0.
\end{equation*}
Equation~\eqref{eq:hopf-high-component} and the corresponding adjoint equation give
\begin{equation*}
\begin{aligned}
Q_\varepsilon q_\varepsilon
&=(i\omega_\varepsilon I-L_{22}^{\varepsilon,\eta_\varepsilon})^{-1}
L_{21}^{\varepsilon,\eta_\varepsilon}P_\varepsilon q_\varepsilon,\quad
Q_\varepsilon\widetilde p_\varepsilon
=\bigl(-i\omega_\varepsilon I-(L_{22}^{\varepsilon,\eta_\varepsilon})^*\bigr)^{-1}
(L_{12}^{\varepsilon,\eta_\varepsilon})^*P_\varepsilon\widetilde p_\varepsilon.
\end{aligned}
\end{equation*}
By \eqref{eq:complement-resolvent-bound} and \eqref{eq:hopf-bounded-rank-one},
and the orthogonality of \(P_\varepsilon,Q_\varepsilon\), as \(\varepsilon\to0^+\),
\begin{equation}
\|Q_\varepsilon q_\varepsilon\|_{L^2(\Omega)}+
\|Q_\varepsilon\widetilde p_\varepsilon\|_{L^2(\Omega)}
\le C\tau_\varepsilon\varepsilon^{-5/4}\longrightarrow0.
\label{eq:hopf-unit-complement}
\end{equation}
Consequently, as \(\varepsilon\to0^+\),
\begin{equation*}
\begin{gathered}
\|J_\varepsilon^*q_\varepsilon\|_2^2
=1-\|Q_\varepsilon q_\varepsilon\|_{L^2(\Omega)}^2\longrightarrow1,\qquad
\|J_\varepsilon^*\widetilde p_\varepsilon\|_2^2
=1-\|Q_\varepsilon\widetilde p_\varepsilon\|_{L^2(\Omega)}^2\longrightarrow1.
\end{gathered}
\end{equation*}
Equations \eqref{eq:hopf-schur-reduction}--\eqref{eq:hopf-schur-limit}
and \eqref{eq:hopf-critical-parameter} further give, as \(\varepsilon\to0^+\),
\begin{equation*}
\begin{aligned}
\bigl(A(D_{\eta_0},w)-i\omega_0I\bigr)J_\varepsilon^*q_\varepsilon
&\longrightarrow0,\quad
\bigl(A(D_{\eta_0},w)^*+i\omega_0I\bigr)J_\varepsilon^*\widetilde p_\varepsilon
\longrightarrow0.
\end{aligned}
\end{equation*}
Since \(i\omega_0\) is simple and
\(\langle\widetilde p_\varepsilon,q_\varepsilon\rangle>0\),
we can choose a common phase of \(q_\varepsilon,p_\varepsilon\), consistently
with \eqref{eq:hopf-endpoint-normalization}, so that, as \(\varepsilon\to0^+\),
\begin{equation*}
J_\varepsilon^*q_\varepsilon\longrightarrow q,\qquad
J_\varepsilon^*\widetilde p_\varepsilon
\longrightarrow\frac{p_H}{\|p_H\|_2}.
\end{equation*}
The orthogonal decomposition then yields, as \(\varepsilon\to0^+\),
\begin{equation*}
\begin{aligned}
\frac1{\|p_\varepsilon\|_{L^2(\Omega)}}
&=\langle\widetilde p_\varepsilon,q_\varepsilon\rangle
=\langle J_\varepsilon^*\widetilde p_\varepsilon,
J_\varepsilon^*q_\varepsilon\rangle
+\langle Q_\varepsilon\widetilde p_\varepsilon,
Q_\varepsilon q_\varepsilon\rangle
\longrightarrow\frac1{\|p_H\|_2}>0.
\end{aligned}
\end{equation*}
Thus \(\|p_\varepsilon\|_{L^2(\Omega)}\to\|p_H\|_2\) as \(\varepsilon\to0^+\).
Using \(p_\varepsilon=\|p_\varepsilon\|_{L^2(\Omega)}\widetilde p_\varepsilon\), we obtain,
as \(\varepsilon\to0^+\),
\begin{equation}
J_\varepsilon^*q_\varepsilon\to q,\qquad
J_\varepsilon^*p_\varepsilon\to p_H.
\label{eq:hopf-eigenvector-convergence}
\end{equation}
Since \(\|p_\varepsilon\|_{L^2(\Omega)}\to\|p_H\|_2\) as \(\varepsilon\to0^+\),
we have \(\|p_\varepsilon\|_{L^2(\Omega)}\le1+\|p_H\|_2\) for all sufficiently small
\(\varepsilon>0\). Using
\(p_\varepsilon=\|p_\varepsilon\|_{L^2(\Omega)}\widetilde p_\varepsilon\)
and \eqref{eq:hopf-unit-complement}, we obtain
\begin{equation}
\begin{aligned}
\|Q_\varepsilon q_\varepsilon\|_{L^2(\Omega)}+\|Q_\varepsilon p_\varepsilon\|_{L^2(\Omega)}
&=\|Q_\varepsilon q_\varepsilon\|_{L^2(\Omega)}
+\|p_\varepsilon\|_{L^2(\Omega)}\|Q_\varepsilon\widetilde p_\varepsilon\|_{L^2(\Omega)}\\
&\le(1+\|p_H\|_2)
\bigl(\|Q_\varepsilon q_\varepsilon\|_{L^2(\Omega)}
+\|Q_\varepsilon\widetilde p_\varepsilon\|_{L^2(\Omega)}\bigr)\\
&\le C\tau_\varepsilon\varepsilon^{-5/4}.
\end{aligned}
\label{eq:hopf-eigenfunction-complement}
\end{equation}
Here \(C\) has been enlarged to absorb the fixed factor
\(1+\|p_H\|_2\) and is independent of \(\varepsilon\).

Set
\begin{equation*}
f_{11,\varepsilon}=\widehat{\mathcal B}_\varepsilon
(q_\varepsilon,\overline q_\varepsilon),\qquad
f_{20,\varepsilon}=\widehat{\mathcal B}_\varepsilon(q_\varepsilon,q_\varepsilon).
\end{equation*}
Since \(\|\widehat{\mathcal B}_\varepsilon\|\le C\varepsilon^{-1/4}\) and
\(\|q_\varepsilon\|_{L^2(\Omega)}=1\), we have
\(\|f_{11,\varepsilon}\|_{L^2(\Omega)}+\|f_{20,\varepsilon}\|_{L^2(\Omega)}\le C\varepsilon^{-1/4}\).
By bilinearity and \eqref{eq:hopf-eigenfunction-complement}, for
\(v\in\{q_\varepsilon,\overline q_\varepsilon\}\), 
\begin{equation*}
\begin{aligned}
&\bigl\|\widehat{\mathcal B}_\varepsilon(q_\varepsilon,v)
-\widehat{\mathcal B}_\varepsilon(P_\varepsilon q_\varepsilon,P_\varepsilon v)\bigr\|_{L^2(\Omega)}\le C\varepsilon^{-1/4}\|Q_\varepsilon q_\varepsilon\|_{L^2(\Omega)}
\le C\tau_\varepsilon\varepsilon^{-3/2}.
\end{aligned}
\end{equation*}
Thus \eqref{eq:hopf-bilinear-projection} and
\eqref{eq:hopf-eigenvector-convergence} give, as \(\varepsilon\to0^+\),
\begin{equation*}
J_\varepsilon^*f_{11,\varepsilon}\longrightarrow\mathcal B_0(q,\overline q),
\qquad
J_\varepsilon^*f_{20,\varepsilon}\longrightarrow\mathcal B_0(q,q).
\end{equation*}
By \eqref{eq:hopf-full-homological} and
\(\mathcal B_{\varepsilon,\eta_\varepsilon}=\tau_\varepsilon\widehat{\mathcal B}_\varepsilon\),
we have \((\lambda I-\mathcal L_{\varepsilon,\eta_\varepsilon})h=\tau_\varepsilon f\) for
\begin{equation*}
(\lambda,h,f)=(0,h_{11,\varepsilon},f_{11,\varepsilon})
\quad\hbox{or}\quad
(2i\omega_\varepsilon,h_{20,\varepsilon},f_{20,\varepsilon}).
\end{equation*}
Using \eqref{eq:hopf-block-matrix} and \eqref{eq:hopf-schur-reduction}, we obtain
\begin{equation*}
\begin{aligned}
&\tau_\varepsilon^{-1}
\bigl[\lambda I-J_\varepsilon^*\mathcal S_\varepsilon
(\lambda,\eta_\varepsilon)J_\varepsilon\bigr]J_\varepsilon^*h=J_\varepsilon^*f+
J_\varepsilon^*L_{12}^{\varepsilon,\eta_\varepsilon}
(\lambda I-L_{22}^{\varepsilon,\eta_\varepsilon})^{-1}Q_\varepsilon f,
\end{aligned}
\end{equation*}
By \eqref{eq:hopf-schur-limit} and \eqref{eq:hopf-critical-parameter}, as
\(\varepsilon\to0^+\),
\begin{equation*}
\tau_\varepsilon^{-1}
\bigl[\lambda I-J_\varepsilon^*\mathcal S_\varepsilon
(\lambda,\eta_\varepsilon)J_\varepsilon\bigr]
\longrightarrow
\begin{cases}
-A(D_{\eta_0},w),&\lambda=0,\\
2i\omega_0I-A(D_{\eta_0},w),&\lambda=2i\omega_\varepsilon.
\end{cases}
\end{equation*}
Both limiting matrices are invertible. Moreover,
\eqref{eq:complement-resolvent-bound} and \eqref{eq:hopf-bounded-rank-one}
give
\begin{equation*}
\begin{aligned}
&\bigl\|J_\varepsilon^*L_{12}^{\varepsilon,\eta_\varepsilon}
(\lambda I-L_{22}^{\varepsilon,\eta_\varepsilon})^{-1}Q_\varepsilon f\bigr\|_2\le
(C\tau_\varepsilon\varepsilon^{-1/4})\frac{C}{\varepsilon}
(C\varepsilon^{-1/4})
\le C\tau_\varepsilon\varepsilon^{-3/2}.
\end{aligned}
\end{equation*}
Solving for \(J_\varepsilon^*h\), we obtain, as \(\varepsilon\to0^+\),
\begin{equation}
\begin{aligned}
J_\varepsilon^*h_{11,\varepsilon}
&\longrightarrow-A(D_{\eta_0},w)^{-1}\mathcal B_0(q,\overline q),\quad
J_\varepsilon^*h_{20,\varepsilon}
&\longrightarrow(2i\omega_0I-A(D_{\eta_0},w))^{-1}\mathcal B_0(q,q).
\end{aligned}
\label{eq:hopf-homological-convergence}
\end{equation}
In particular, \(\|P_\varepsilon h_{11,\varepsilon}\|_{L^2(\Omega)}
+\|P_\varepsilon h_{20,\varepsilon}\|_{L^2(\Omega)}\le C\).
For the same choices of \((\lambda,h,f)\), applying \(Q_\varepsilon\)
to \((\lambda I-\mathcal L_{\varepsilon,\eta_\varepsilon})h=\tau_\varepsilon f\) gives
\begin{equation*}
Q_\varepsilon h=(\lambda I-L_{22}^{\varepsilon,\eta_\varepsilon})^{-1}
\bigl(\tau_\varepsilon Q_\varepsilon f+
L_{21}^{\varepsilon,\eta_\varepsilon}P_\varepsilon h\bigr).
\end{equation*}
Using \eqref{eq:complement-resolvent-bound},
\eqref{eq:hopf-bounded-rank-one}, and the bounds on
\(f_{j,\varepsilon}\) and \(P_\varepsilon h_{j,\varepsilon}\) above, we obtain
\begin{equation}
\begin{aligned}
\|Q_\varepsilon h_{11,\varepsilon}\|_{L^2(\Omega)}+
\|Q_\varepsilon h_{20,\varepsilon}\|_{L^2(\Omega)}
&\le\frac{C}{\varepsilon}\sum_{j\in\{11,20\}}
\left(\tau_\varepsilon\|f_{j,\varepsilon}\|_{L^2(\Omega)}
+\|L_{21}^{\varepsilon,\eta_\varepsilon}\|\,
\|P_\varepsilon h_{j,\varepsilon}\|_{L^2(\Omega)}\right)\\
&\le C\tau_\varepsilon\varepsilon^{-5/4}.
\end{aligned}
\label{eq:hopf-homological-complement}
\end{equation}
Thus \(\|h_{11,\varepsilon}\|_{L^2(\Omega)}+\|h_{20,\varepsilon}\|_{L^2(\Omega)}\le C\).
For \((u,v)=(q_\varepsilon,h_{11,\varepsilon})\) or
\((\overline q_\varepsilon,h_{20,\varepsilon})\), the bound
\(\|\widehat{\mathcal B}_\varepsilon\|\le C\varepsilon^{-1/4}\),
\eqref{eq:hopf-eigenfunction-complement}, and
\eqref{eq:hopf-homological-complement} yield
\begin{equation*}
\begin{aligned}
&\left|\langle p_\varepsilon,\widehat{\mathcal B}_\varepsilon(u,v)\rangle
-\langle P_\varepsilon p_\varepsilon,
\widehat{\mathcal B}_\varepsilon(P_\varepsilon u,P_\varepsilon v)\rangle\right|\\
&\quad\le C\varepsilon^{-1/4}
\bigl(\|Q_\varepsilon p_\varepsilon\|_{L^2(\Omega)}+
\|Q_\varepsilon u\|_{L^2(\Omega)}+\|Q_\varepsilon v\|_{L^2(\Omega)}\bigr)\\
&\quad\le C\tau_\varepsilon\varepsilon^{-3/2}.
\end{aligned}
\end{equation*}
Since \(P_\varepsilon=J_\varepsilon J_\varepsilon^*\),
\eqref{eq:hopf-bilinear-projection},
\eqref{eq:hopf-eigenvector-convergence}, and
\eqref{eq:hopf-homological-convergence} give, as \(\varepsilon\to0^+\),
\begin{equation*}
\begin{aligned}
&\langle P_\varepsilon p_\varepsilon,
\widehat{\mathcal B}_\varepsilon
(P_\varepsilon q_\varepsilon,P_\varepsilon h_{11,\varepsilon})\rangle\longrightarrow
-\bigl\langle p_H,\mathcal B_0
\bigl(q,A(D_{\eta_0},w)^{-1}\mathcal B_0(q,\overline q)\bigr)\bigr\rangle,\\
&\langle P_\varepsilon p_\varepsilon,
\widehat{\mathcal B}_\varepsilon
(P_\varepsilon\overline q_\varepsilon,P_\varepsilon h_{20,\varepsilon})\rangle\longrightarrow
\bigl\langle p_H,\mathcal B_0
\bigl(\overline q,(2i\omega_0I-A(D_{\eta_0},w))^{-1}\mathcal B_0(q,q)\bigr)\bigr\rangle.
\end{aligned}
\end{equation*}
By \eqref{eq:hopf-critical-parameter},
\(\tau_\varepsilon/(2\omega_\varepsilon)\to1/(2\omega_0)\)
as \(\varepsilon\to0^+\). Substituting these limits into
\eqref{eq:hopf-scaled-lyapunov} and using
\eqref{eq:hopf-endpoint-lyapunov}, we conclude that
\begin{equation}
\ell_\varepsilon\longrightarrow\ell(D_{\eta_0},w)
\qquad(\varepsilon\to0^+).
\label{eq:hopf-lyapunov-limit}
\end{equation}

\medskip
\refstepcounter{proofstep}\label{step:hopf-branches-stability}
\noindent\textbf{Step \theproofstep.}
For each fixed sufficiently small \(\varepsilon>0\),
Step~\ref{step:hopf-crossing-crossing} gives the simple pair
\(\pm i\omega_\varepsilon\) and \(\alpha_\varepsilon>0\) in
\eqref{eq:hopf-critical-parameter};
all other eigenvalues have negative real parts.
The signs of \(\ell(D_{\eta_0},w)\) in Lemma~\ref{lem:matrix-choice}, together with
\eqref{eq:hopf-lyapunov-limit}, give \(\ell_\varepsilon\ne0\).
After the time change
\(s=\omega_\varepsilon t\), the parabolic Hopf theorem and its stability
result \cite[Theorems~1.11 and~2.13, Remark~2.24]{CrandallRabinowitz1977}
apply and yield a local branch of nonstationary periodic solutions;
see also \cite[Section~8.2]{Henry1981}.
With \(v=zq_\varepsilon+\overline z\,\overline q_\varepsilon+O(|z|^2)\) as \(z\to0\)
and \(r=|z|\), the equation for the amplitude \(r\) in the original time \(t\) is
\begin{equation}
\frac{\mathrm d r}{\mathrm d t}=r\left[\alpha_\varepsilon(\eta-\eta_\varepsilon)
+\omega_\varepsilon\ell_\varepsilon r^2
+o(|\eta-\eta_\varepsilon|+r^2)\right].
\label{eq:hopf-radial-normal-form}
\end{equation}
Along the nonstationary periodic branch,
\eqref{eq:hopf-radial-normal-form} yields, for fixed sufficiently small \(\varepsilon>0\),
\begin{equation*}
r^2=-\frac{\alpha_\varepsilon(\eta-\eta_\varepsilon)}
{\omega_\varepsilon\ell_\varepsilon}
+o(|\eta-\eta_\varepsilon|)
\qquad(\eta\to\eta_\varepsilon).
\end{equation*}
By \eqref{eq:hopf-critical-parameter},
\(\omega_\varepsilon>0\) and \(\alpha_\varepsilon>0\). Thus, if
\(\ell_\varepsilon<0\), the periodic solutions occur for
\(\eta>\eta_\varepsilon\), and the bifurcation is supercritical; if
\(\ell_\varepsilon>0\), they occur for \(\eta<\eta_\varepsilon\), and
the bifurcation is subcritical. Since the other eigenvalues have negative
real parts by Step~\ref{step:hopf-crossing-crossing}, the periodic solutions
are orbitally asymptotically stable in the former case and unstable in the latter.

For fixed \(\varepsilon>0\), parabolic regularity gives convergence to
\(\phi_\varepsilon\) in the uniform spatial norm as the branch amplitude
\(r\to0^+\). Choose the amplitude small enough that
\begin{equation*}
\sup_t\|u(t,\cdot)-\phi_\varepsilon\|_{C(\overline\Omega)}
<\tfrac12\min_{\overline\Omega}\phi_\varepsilon.
\end{equation*}
Hence the periodic solutions are strictly positive. The size of this
amplitude neighborhood may depend on \(\varepsilon\), and the period tends to
\(2\pi/\omega_\varepsilon\) as \(r\to0^+\), with \(\varepsilon\) fixed. Applying these
conclusions to the two choices \((D,w^{\rm s})\) and
\((D,w^{\rm u})\) in \eqref{eq:hopf-common-data} proves part~\textup{(\ref{item:hopf-branches})}
and completes the proof of Theorem~\ref{thm:three-peak-hopf}.

\end{proof}

\begin{figure}[!htbp]
\centering
\setcounter{figurepanel}{0}
\renewcommand{\thefigurepanel}{\alph{figurepanel}}
\refstepcounter{figurepanel}\label{panel:supercritical}%
\refstepcounter{figurepanel}\label{panel:subcritical}%
\includegraphics[trim=2pt 2pt 2pt 2pt,clip]{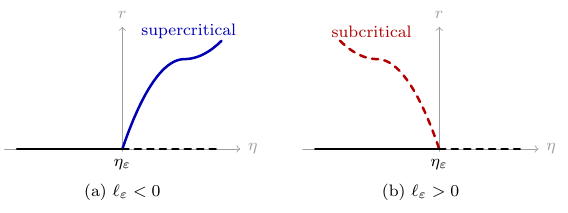}
\caption{Supercritical (left) and subcritical (right) Hopf bifurcations.}
\label{fig:hopf-branch-stability}
\end{figure}

\FloatBarrier
\section*{Acknowledgments}
We acknowledge the assistance of AI tools in identifying references that helped us learn semiclassical analysis techniques used in the spectral reduction of the non-self-adjoint linearized operator. These tools also assisted with some calculations, English-language editing, and the presentation of some mathematical arguments during the preparation of this manuscript.

Y.~Lou was partially supported by the Interdisciplinary Disciplines
Breakthrough Plan of the Ministry of Education of China
(JYB2025XDXM902), and by NSF of China
(12631017 and 12261160366). M.-Z.~Xin is partially supported by
NSF of China (12601334), the China National
Postdoctoral Program for Innovative Talents (BX2026009), and the China
Postdoctoral Science Foundation (2026M793394).

\end{document}